\documentclass[10pt]{article}
\usepackage{amsmath,amssymb,amsthm}
\usepackage{algorithm}
\usepackage{algorithmic}
\usepackage[utf8]{inputenc} 
\usepackage[T1]{fontenc}    
\usepackage{geometry}

\usepackage{authblk}
\newcommand{\myfootnote}[1]{
\renewcommand{\thefootnote}{}
\footnotetext{\hspace{-16.5pt}\footnotesize#1}
\renewcommand{\thefootnote}{\arabic{footnote}}}

\usepackage{hhline}
\usepackage{hyperref}
\usepackage{tablefootnote}
\usepackage{graphicx}
\usepackage{subfig}
\usepackage{epstopdf}

\usepackage{dsfont}
\usepackage{amssymb}
\usepackage[shortlabels]{enumitem}
\usepackage{graphicx}
\usepackage{mathrsfs}
\usepackage{float}
\usepackage{bbm}
\usepackage{amsmath}
\usepackage{comment}
\usepackage{listings}
\usepackage{color}

\newcommand{\RV}[1]{\textcolor{black}{#1}}

\usepackage[dvipsnames]{xcolor}
\usepackage{stmaryrd}
\usepackage{multicol}
\usepackage{multirow}
\allowdisplaybreaks

\theoremstyle{plain} 
\newtheorem{theorem}{Theorem}[section]
\newtheorem*{theorem*}{Theorem}
\newtheorem{lemma}[theorem]{Lemma}
\newtheorem{proposition}[theorem]{Proposition}
\newtheorem{corollary}[theorem]{Corollary}

\newtheorem{example}[theorem]{Example}
\theoremstyle{definition} 
\newtheorem{remark}[theorem]{Remark}
\newtheorem{definition}{Definition}

\newtheorem{question}{Question}

\def\R{\mathbb{R}}

\def\cO{\mathcal{O}}
\newcommand{\Prob}{\mathbb{P}}

\newcommand{\app}{\textnormal{app}}

\newcommand{\rank}{\mathrm{rank}\,}

\newcommand{\kron}{\otimes}  
\newcommand{\out}{\pmb\kron}  
\usepackage[dvipsnames]{xcolor}

\graphicspath{{figs/}}

\begin{document}

\title
{Mode-wise Tensor Decompositions: Multi-dimensional Generalizations of CUR Decompositions}

\author[1]{HanQin Cai}
\author[2]{Keaton Hamm}
\author[1]{Longxiu Huang}
\author[1]{Deanna Needell}
\affil[1]{Department of Mathematics, \protect\\ University of California, Los Angeles,\protect\\ Los Angeles, CA 90095, USA.\vspace{.15cm}}
\affil[2]{Department of Mathematics, \protect\\ University of Texas at Arlington \protect \\ Arlington, TX 76019, USA \vspace{.15cm}}
\myfootnote{\indent\indent Email addresses: hqcai@math.ucla.edu (H.Q. Cai), keaton.hamm@uta.edu (K. Hamm), huangl3@math.ucla.edu (L. Huang), and deanna@math.ucla.edu (D. Needell).}

 \date{}

\maketitle
\begin{abstract}
Low rank tensor approximation is a fundamental tool in modern machine learning and data science. In this paper, we study the characterization, perturbation analysis, and an efficient sampling strategy for two primary tensor CUR approximations, namely Chidori and Fiber CUR. We characterize exact tensor CUR decompositions for low multilinear rank tensors. We also present theoretical error bounds of the tensor CUR approximations when (adversarial or Gaussian) noise appears. Moreover, we show that low cost uniform sampling is sufficient for tensor CUR approximations if the tensor has an incoherent structure. Empirical performance evaluations, with both synthetic and real-world datasets, establish the speed advantage of the tensor CUR approximations over other state-of-the-art low multilinear rank tensor approximations.
\end{abstract}

\textbf{Keywords:}
tensor decomposition, low-rank tensor approximation, CUR decomposition, randomized linear algebra, hyperspectral image compression

\section{Introduction}

A tensor is a multi-dimensional array of numbers, and is the higher-order generalization of vectors and matrices; thus tensors can express more complex intrinsic structures of higher-order data.
In various data-rich domains such as computer vision, recommendation systems, medical imaging, data mining, and multi-class learning  consisting of multi-modal and multi-relational data, tensors have emerged as a powerful paradigm for managing the data deluge.  Indeed, data is often more naturally represented as a tensor than a vector or matrix; for instance hyperspectral images result in 3--mode tensors, and color videos can be represented as 4--mode tensors.  The tensor structure of such data can carry with it more information; for instance, spatial information is kept along the spectral or time slices in these applications.  Thus, tensors are not merely a convenient tool for analysis, but are rather a fundamental structure exhibited by data in many domains.

As in the matrix setting, but to an even greater extent, an important tool for handling tensor data efficiently is tensor decomposition, in which the tensor data is represented by a few succinctly representable components. For example, tensor decompositions are used in
  computer vision \cite{vasilescu2002multilinear,yan2006multilinear} to enable the extraction of patterns that generalize well across common modes of variation, whereas in bioinformatics \cite{yener2008multiway,omberg2009global}, tensor decomposition has proven useful for the understanding of cellular states
and biological processes.  Similar to matrices, tensor decompositions can be utilized for compression and low-rank approximation \cite{MMD2008,zhang2015compression,du2016pltd,li2021correlation}. However, there are a greater variety of natural low-rank tensor decompositions than for matrices, in part because the notion of rank for higher order tensors is not unique.
 
Tensor decompositions have been widely studied both in theory and application for some time (e.g.,  \cite{HF1927,kolda2009tensor,ZOIA2018}). Examples of different tensor decompositions include the CANDECOMP/PARAFAC (CP) decomposition \cite{HF1927}, Tucker decomposition \cite{tucker1966}, Hierarchical-Tucker (H-Tucker) decomposition \cite{grasedyck2010hierarchical},  Tensor Train (TT) decomposition \cite{oseledets2011tensor}, and Tensor Singular Value Decomposition (t-SVD)\cite{KBHH2013}.  However, most current tensor decomposition schemes require one to first unfold the tensor along a given mode into a matrix, implement a matrix decomposition method, and then relate the result to a tensor form via mode multiplication.  Such algorithms are matricial in nature, and fail to properly utilize the tensor structure of the data.  Moreover, tensors are often vast in size, so matrix-based algorithms on tensors often have severe computational complexity, whereas algorithms that are fundamentally tensorial in nature will have drastically lower complexity.  
Let us be concrete here.  Suppose a tensor has order $n$ and each dimension is $d$; most of the algorithms to compute the CP decomposition are iterative. For example, the ALS algorithm with line search has a  complexity of   order $\cO(2^nrd^n+nr^3)$ \cite{PTC2013} where $r$ is the CP--rank. Similarly,  if the multilinear rank of the tensor is $(r,\dots,r)$, computing the HOSVD by computing the compact SVD of the matrices obtained by unfolding the tensor would result in complexity of order $\cO(rd^{n})$. \RV{To accelerate both CP and Tucker decomposition computations, many works have applied different randomized techniques.  For instance, there are several sketching algorithms for CP decompositions \cite{battaglino2018practical,song2019relative,erichson2020randomized,gittens2020adaptive,cheng2016spals} based on the sketching techniques for low-rank matrix approximation \cite{woodruff2014sketching}, and these techniques have been shown to greatly improve computational efficiency compared to the original ALS algorithm. For Tucker decomposition, several randomized algorithms \cite{ahmadi2021randomized, che2021randomized} based on random projection have been developed to accelerate HOSVD and HOOI. }

The purpose of this work is to explore tensor-based methods for low-rank tensor decompositions. In particular, we present two flexible tensor decompositions inspired by matrix CUR decompositions, which utilize small core subtensors to reconstruct the whole.  Moreover, our algorithms for forming tensor CUR decompositions do not require unfolding the tensor into a matrix. The uniform sampling based tensor decompositions we discuss subsequently have computational complexity $\cO(r^2n\log^2(d))$, which is a dramatic improvement over those discussed above.

Matrix CUR decompositions \cite{HH2020}, which are sometimes called pseudoskeleton decompositions \cite{Goreinov,chiu2013sublinear}, utilize the fact that any matrix $X\in\mathbb{R}^{m\times n}$ with $\rank(X)=r$ can be perfectly reconstructed via $X=CU^\dagger R$ by choosing submatrices $C=X(:,J)$, $R=X(I,:)$, and $U=X(I,J)$ such that $\rank(U)=r$. Consequently, CUR decompositions only require some of the entries of $X$ governed by the selected column and row submatrices to recover the whole. This observation makes CUR decompositions extremely attractive for large-scale low-rank matrix approximation problems. 

It is natural to consider extending the CUR decomposition to tensors in a way that fundamentally utilizes the tensor structure.  We will see subsequently that there is no single canonical extension of matrix CUR decompositions to the tensor setting; however, we propose that a \emph{natural} tensor CUR decomposition must be one that selects subsets of each mode, and which is not based on unfolding operations on the tensor.

\subsection{Contributions}

In this paper, our main contributions can be summarized as follows:

\begin{enumerate}[1.]
\item We first provide some new characterizations for CUR decompositions for tensors and show by example that the characterization for tensors is different from that for matrices. In particular, we show that exact CUR decompositions of low multilinear rank tensors are equivalent to the multilinear rank of the core subtensor chosen being the same as that of the data tensor.
\item Real data tensors rarely exhibit exactly low-rank structure, but they can be modeled as a low multilinear rank tensor plus noise. We undertake a novel perturbation analysis for low multilinear rank tensor CUR approximations, and prove error bounds for the two primary tensor CUR decompositions discussed here. These bounds are qualitative and are given in terms of submatrices of singular vectors of unfoldings of the tensor, and represent the first approximation bounds for these decompositions.  We additionally provide some specialized bounds for the case of low multilinear rank tensors perturbed by random Gaussian noise. Our methods and analysis affirmatively answer a question of \cite{MMD2008} regarding finding tensor CUR decompositions that preserve multilinear structure of the tensor.
\item When the underlying low multilinear rank tensor has an incoherence property and is perturbed by noise, we show that uniformly sampling indices along each mode yields a low multilinear rank approximation to the tensor with good error bounds.
\item We give guarantees on random sampling procedures of the indices for each mode that ensure that  an exact CUR decomposition of a low multilinear rank tensor holds with high probability.
\item We illustrate the effectiveness of the various decompositions proposed here on synthetic tensor datasets and real hyperspectral imaging datasets.
\end{enumerate}

\subsection{Organization}

The rest of the paper is laid out as follows.  Section~\ref{SEC:PriorArt} contains a comparison of our work with prior art in both the low-rank tensor approximation and tensor CUR literature. Section~\ref{SEC:Prelim} contains the notations and descriptions of low multilinear rank tensor decompositions necessary for the subsequent analysis. Section~\ref{SEC:Main} contains the statements of the main results of the paper, with Section~\ref{SEC:Characterization} containing the characterization theorems for exact tensor CUR decompositions, Section~\ref{SEC:MainPerturbation} containing the statements of the main approximation bounds for CUR approximations of arbitrary tensors, and Section~\ref{SEC:Sampling} giving error bounds for CUR decompositions obtained via randomly sampling indices.  Section~\ref{SEC:CharacterizationProof} contains the proofs of the characterization theorems (Theorems~\ref{cor: CUR Char1} and \ref{thm: CUR Char}), Section~\ref{SEC:PerturbationProof} contains proofs of the approximation bounds, and Section~\ref{SEC:SamplingProof} contains proofs related to random sampling of core subtensors to either achieve an exact decomposition of low multilinear rank tensors or achieve good approximation of arbitrary tensors.  Experiments on synthetic and real hyperspectral data are contained in Section~\ref{SEC:Experiments}, and the paper ends with some concluding remarks and future directions in Section~\ref{SEC:Conclusion}.

\subsection{Prior Art}\label{SEC:PriorArt}

CUR decompositions for matrices have a long history; a similar matrix form expressed via Schur decompositions goes back at least to \cite{Penrose56}; more recently, they were studied as pseudoskeleton decompositions \cite{Goreinov,Goreinov3}.  Within the randomized linear algebra, theoretical computer science, and machine learning literature, they have been studied beginning with the work of Drineas, Kannan, and Mahoney \cite{DKMIII,DMM08,DMPNAS,chiu2013sublinear,SorensenDEIMCUR,voronin2017efficient}.  For a more thorough historical discussion, see \cite{HH2020}.  

The first extension of CUR decompositions to tensors \cite{MMD2008} involved a single-mode unfolding of 3--mode tensors.  Accuracy of the tensor-CUR decomposition was transferred from existing guarantees for matrix CUR decompositions, and gave additive error guarantees (containing a factor of $\varepsilon\|\mathcal{A}\|_F^2$). Later, \cite{caiafa2010generalizing} proposed a different variant of tensor CUR that accounts for all modes, which they termed Types 1 and 2 Fiber Sampling Tensor Decompositions. In this work, we dub these with more descriptive monikers Fiber and Chidori CUR decompositions (Section~\ref{SEC:Main}).  The decompositions we discuss later are generalizations of those of Caiafa and Cichoki, as their work considers tensors with multilinear rank $(r,\dots,r)$ and identically sized index sets ($I_i$ described in Section~\ref{SEC:Main}).

In \cite{HH2020}, there are several equivalent characterizations for CUR decompositions in the matrix case. We find that there are characterizations for CUR decompositions in the tensor case (Theorems~\ref{cor: CUR Char1} and \ref{thm: CUR Char} here).  Interestingly, the tensor characterization has some significant differences from that of matrices (see Example \ref{exmp: non-eqi-CUR}).

There are very few approximation bounds for any tensor CUR decompositions. \cite{MMD2008} provide some basic additive error approximation bounds for a tensor CUR decomposition made from only taking slices along a single mode. In contrast, the bounds given in this paper (Theorems~\ref{cor: pt_ana_tensor-cur}, \ref{thm: pt_ana_tensor-cur_tube}, and \ref{thm: per_uniform_sampling}) are the first general bounds for both Fiber and Chidori CUR decompositions which subsample along all modes. Our analysis positively answers the question of Mahoney et al. by giving sampling methods which are able to truly account for the multilinear structure of the data tensor in a nontrivial way.

A tensor CUR decomposition based on the t-SVD was explored by \cite{wang2017missing}, and the authors give relative error guarantees for a specified oversampling rate along each mode.  Their decomposition is fundamentally different from ours as it uses t-SVD techniques which use block circulant matrices related to the original data tensor. Consequently, their decomposition is much more costly than those described here.  Additionally, their factor tensors are typically larger than the constituent parts of the decompositions described here.

For matrices, uniform random sampling is known to provide good CUR approximations under incoherence assumptions, e.g., \cite{chiu2013sublinear}. Theorem~\ref{thm: pt_ana_tensor-cur_tube} extends this analysis to tensors of arbitrary number of modes.  Additionally, there are several standard randomized sampling procedures known to give exact matrix CUR decompositions with high probability \cite{HH2019}; Theorem~\ref{THM:UniformExactCUR} provides a sample extension of these results for tensors.  

The next section contains further comparison of tensor CUR decompositions with other standard low-rank tensor decompositions such as the HOSVD.

\section{Preliminaries and Notation}\label{SEC:Prelim}

Tensors, matrices, vectors, and scalars are denoted in different typeface for clarity below. In the sequel, calligraphic capital letters are used for tensors,  capital letters are used for matrices, lower boldface letters for vectors, and regular letters for scalars.  The set of the first $d$ natural numbers is denoted by $[d]:=\{1,\cdots,d\}$.
We include here some basic notions relating to tensors, and refer the reader to, e.g., \cite{kolda2009tensor} for a more thorough introduction.

A tensor is a multidimensional array whose dimension is called the \emph{order} (or also \emph{mode}). The space of real tensors of order $n$ and size $d_1\times\cdots\times d_n$ is denoted as $\mathbb{R}^{d_1\times  \cdots\times d_n}$. 
The elements of a tensor $\mathcal{X}\in\mathbb{R}^{d_1\times \cdots\times d_n}$ are denoted by $\mathcal{X}_{i_1,\cdots, i_n}$.

An $n$--mode tensor $\mathcal{X}$ can be matricized, or reshaped into a matrix, in $n$ ways by unfolding it along each of the $n$ modes.  
The mode-$k$ matricization/unfolding of tensor $\mathcal{X}\in\mathbb{R}^{d_1\times  \cdots\times d_n}$ is the matrix denoted by $\mathcal{X}_{(k)}\in\mathbb{R}^{d_k\times\prod_{j\neq k}d_j}$ whose columns are composed of all the vectors obtained from $\mathcal{X}$ by fixing all indices except for the $k$-th dimension.  The mapping $\mathcal{X}\mapsto \mathcal{X}_{(k)}$ is called the mode-$k$ unfolding operator.

Given $\mathcal{X}\in\mathbb{R}^{d_1\times \cdots\times d_n}$, the norm 
$\left\|\mathcal{X}\right\|_F$ is defined via
\begin{equation*}
\left\|\mathcal{X}\right\|_F=\Bigg(\sum_{i_1,\cdots, i_n}\mathcal{X}_{i_1,\cdots, i_n}^2 \Bigg)^{\frac{1}{2}}.
\end{equation*}

There are various product operations related to tensors; the ones that will be utilized in this paper are the following.

\begin{itemize}
 \item Outer product: Let $\mathbf{a_1}\in\mathbb{R}^{d_1}, \cdots, \mathbf{a_n}\in\mathbb{R}^{d_n}$. The outer product among these $n$ vectors is a tensor $\mathcal{X}\in\mathbb{R}^{d_1\times\cdots\times d_n}$ defined as:
\[\mathcal{X}=\mathbf{a}_1\out \cdots\out \mathbf{a}_n, ~~ \mathcal{X}_{i_1,\cdots,i_n}=\prod\limits_{k=1}^n \mathbf{a}_k(i_k).\]
 The tensor $\mathcal{X}\in\mathbb{R}^{d_1\times\cdots\times d_n}$ is of rank one if it can be written as the outer product of $n$ vectors.
\item  Kronecker product of matrices: The Kronecker product of $A\in\mathbb{R}^{I\times J}$ and $B\in\mathbb{R}^{K\times L}$ is denoted by $A\kron B$. The result is a matrix of size $(KI)\times (JL)$ defined by
\begin{eqnarray*}
A\kron B&=&\begin{bmatrix}
A_{11}B&A_{12}B&\cdots &A_{1J}B\\
A_{21}B&A_{22}B&\cdots &A_{2J}B\\
\vdots&\vdots&\ddots&\vdots\\
A_{I1}B&A_{I2}B&\cdots&A_{IJ}B
\end{bmatrix}.
\end{eqnarray*}
\item Mode-$k$ product: Let $\mathcal{X}\in\mathbb{R}^{d_1\times \cdots\times d_n}$ and $A\in\mathbb{R}^{J\times d_k}$, the multiplication between $\mathcal{X}$ on its $k$-th mode with $A$ is denoted as $\mathcal{Y}=\mathcal{X}\times_k A$ with 
\[\mathcal{Y}_{i_1,\cdots,i_{k-1},j,i_{k+1},\cdots,i_{n}}=\sum_{s=1}^{d_k}\mathcal{X}_{i_1,\cdots,i_{k-1},s,i_{k+1},\cdots,i_{n}}A_{j,s}.
\] Note this can be written as a matrix product by noting that $\mathcal{Y}_{(k)}=A \mathcal{X}_{(k)}$. If we have multiple tensor matrix product from different modes, we  use the notation $\mathcal{X}\times_{i=t}^{s} A_i$ to denote the product $\mathcal{X}\times_{t}A_{t}\times_{t+1}\cdots\times_{s}A_{s}$.
\end{itemize}

\RV{For the reader's convenience, we also summarize the notation in Table~\ref{tab:notation}.}

\begin{table}[h!]
\caption{Table of Notation.} \label{tab:notation}
\vspace{-0.15in}
\begin{center}
\begin{small}
\RV{
\begin{tabular}{|c|c|}
\hline
Notation & Definition  \\
\hline
$\mathcal{X}$                   & tensor\\
$X$                  & matrix\\
$\mathbf{x}$                     &vector\\
$x$                  & scalar\\
$\mathcal{X}\times_k X$                & tensor $\mathcal{X}$ times matrix $X$ along the $k$-th mode\\
$X\kron Y$                 & Kronecker product of the matrices $X$ and $Y$\\
$\bf{x}\out \bf{y}$                  & outer product of vectors $\bf{x}$ and $\bf{y}$\\
$\mathcal{X}_{(k)}$         & mode-$k$ unfolding of $\mathcal{X}$ \\
$X(I,:)$ & row submatrix of $X$ with row indices $I$\\
$X(:,J)$ & column submatrix of $X$ with column indices $J$\\
$\mathcal{X}(I_1,\cdots,I_n)$ & subtensor of $\mathcal{X}$ with indices $I_k$ at mode $k$\\
$[d]=\{1,\cdots,d\}$        &  the set of the first $d$ natural numbers\\ 
$\mathbf{r}=(r_1,\cdots,r_n)$ & multilinear rank \\
$\|\cdot\|_2$  & $\ell_2$ norm for vector, spectral norm for matrix \\
$\|\cdot\|_F$  &  Frobenius norm \\
$(\,\cdot\,)^\top$ & transpose\\
$(\,\cdot\,)^\dagger$ & Moore--Penrose pseudoinverse\\
\hline
\end{tabular}
}
\end{small}
\end{center}
\end{table}

\subsection{Tensor Rank}

The notion of \emph{rank} for tensors is more complicated than it is for matrices.  Indeed, rank is non-uniquely defined. In this work, we will primarily utilize the \emph{multilinear rank} (also called \emph{Tucker rank}) of tensors \cite{HF1928}.  The multilinear rank of $\mathcal{X}$ is a tuple $\mathbf{r}=(r_1,\cdots,r_n)\in\mathbb{N}^{n}$, where $r_k=\text{rank}(\mathcal{X}_{(k)})$.  

Multilinear rank is relatively easily computed, but is not necessarily the most natural notion of rank.  Indeed, the CP rank of a tensor \cite{HF1927,HF1928} is the smallest integer $r$ for which $\mathcal{X}$ can be written as the sum or rank-1 tensors. That is, we may write
\begin{equation}\label{EQN:CP}\mathcal{X} = \sum_{i=1}^r \lambda_i\;\mathbf{a}^{(i)}_1\out \cdots\out\mathbf{a}^{(i)}_n\end{equation} for some $\{\lambda_i\}\subseteq\R$ and $\mathbf{a}^{(i)}_k\in\R^{d_k}$.  Regard that the notion of rank-1 tensors is unambiguous.

\subsection{Tensor Decompositions}

Tensor decompositions are powerful tools for extracting meaningful, latent structures
in heterogeneous, multidimensional data (see, e.g., \cite{kolda2009tensor}).  
Similar to the matrix case, there are a wide array of tensor decompositions, and one may select a different one based on the task at hand.  For instance, CP decompositions (of the form \eqref{EQN:CP}) are typically the most compact representation of a tensor, but the substantial drawback is that they are NP--hard to compute \cite{k1989}.  On the other hand, Tucker decompositions and Higher-order Singular Value Decompositions (HOSVD) are natural extensions of the matrix SVD, and thus useful for describing features of the data.  Matrix CUR decompositions give factorizations in terms of actual column and row submatrices, and can be cheap to form by random sampling.  CUR decompositions are known to provide interpretable representations of data in contrast to the SVD, for example \cite{DMPNAS}.  Similarly, tensor analogues of CUR decompositions represent a tensor via subtubes and fibers of it.  We will discuss this further in the following subsection.

The Tucker decomposition was proposed by \cite{tucker1966} and further developed in \cite{KD1980,DDV2000}. A special case of Tucker decompositions is called the Higher-order SVD (HOSVD): given an $n$-order tensor $\mathcal{X}$, its HOSVD is defined as the modewise product of a core tensor $\mathcal{T}\in\mathbb{R}^{r_1\times\cdots \times r_n}$ with $n$ factor matrices $W_k\in\mathbb{R}^{d_k\times r_k}$ (whose columns are orthonormal) along each mode such that
\[\mathcal{X}=\mathcal{T}\times_1 W_1\times_2\cdots\times_n W_n=:\llbracket \mathcal{T};W_1,\cdots,W_n\rrbracket,
\] where $r_k = \rank(\mathcal{X}_{(k)})$.  
  If we unfold $\mathcal{X}$ along its $k$-th mode, we have
\[\mathcal{X}_{(k)}= W_{k}\mathcal{T}_{(k)}(W_1\kron \cdots \kron W_{k-1}\kron W_{k+1}\kron\cdots\kron W_{n})^\top.
\]

The HOSVD can be computed as follows:
\begin{enumerate}
    \item Unfold $\mathcal{X}$ along mode $k$ to get matrix $\mathcal{X}_{(k)}$;
    \item Compute the compact SVD of $\mathcal{X}_{(k)}=W_k\Sigma_k V_k^\top$;
    \item $\mathcal{T}=\mathcal{X}\times_1 W_1^\top\times_2\cdots\times_n W_n^\top$.
\end{enumerate}

For a more comprehensive introduction to tensor decompositions, readers are referred to \cite{AY2008,kolda2009tensor,SDFHPF2017}. 

In the statements below, $a\gtrsim b$ means that $a\geq cb$ for some absolute constant $c>0$.

\subsection{Matrix CUR Decompositions}

To better compare the matrix and tensor case, we first discuss the characterization of CUR decompositions for matrices obtained in \cite{HH2020}.

\begin{theorem}[\cite{HH2020}]\label{THM:MatrixCUR}
Let $A\in\R^{m\times n}$ and $I\subseteq[m]$, $J\subseteq[n]$.  Let $C=A(:,J)$, $U=A(I,J)$, and $R=A(I,:)$.  Then the following are equivalent:
\begin{enumerate}[(i)]
    \item\label{item:matrixrankU} $\rank(U)=\rank(A)$,
    \item\label{item:matrixACUR} $A=CU^\dagger R$,
    \item\label{item:matrixprojection} $A = CC^\dagger AR^\dagger R$,
    \item\label{item:matrixAdagger} $A^\dagger = R^\dagger UC^\dagger$,
    \item\label{item:matrixranks} $\rank(C)=\rank(R)=\rank(A)$,
    \item\label{item:matrixSchur} Suppose columns and rows are rearranged so that $A = \begin{bmatrix} U & B \\ D & E\end{bmatrix}$, and the generalized Schur complement of $A$ with respect to $U$ is defined by $A/U:=E-DU^\dagger B$. Then $A/U = 0$.
\end{enumerate}
Moreover, if any of the equivalent conditions above hold, then $U^\dagger = C^\dagger AR^\dagger$.
\end{theorem}

We note that equivalent condition $\ref{item:matrixSchur}$ in Theorem~\ref{THM:MatrixCUR} is not proven in \cite{HH2020}, but can readily be deduced using basic methods and Corollary 19.1 of \cite{matsaglia1974equalities}.

Notice that if the matrix is treated as a two-way tensor,  the matrix CUR decomposition as in $\ref{item:matrixACUR}$ can be written in the form: 
\begin{equation*}
A=CU^\dagger R=U\times_1 CU^\dagger\times_2 R^\top(U^\top)^\dagger.
\end{equation*} 

CUR decompositions provide a representation of data in terms of other data, hence allowing for more ease of interpretation of results.  They have been used to practical effect in exploratory data analysis related to natural language processing \cite{DMPNAS}, subspace clustering \cite{aldroubi2018similarity,AHKS}, and basic science \cite{yip2014objective,yang2015identifying}.  Additionally, CUR is often used as a fast approximation to the SVD \cite{DKMIII,DMM08,boutsidis2017optimal,voronin2017efficient} (see also the more general survey of randomized methods \cite{halko2011finding}).  Recently, CUR decompositions have been used to accelerate algorithms for Robust PCA \cite{CHHLW2021,CHHN2020}. See \cite{HH2020} for a more thorough survey of CUR decompositions and their utility.

In Machine Learning, the Nystr\"{o}m method (CUR decompositions in which the same columns and rows are selected to approximate symmetric positive semi-definite matrices) is heavily utilized to estimate large kernel matrices \cite{williams2001using,gittens2016revisiting,bertozzi2016diffuse}.

For multidimensional data, \cite{MMD2008} proposed tensor-CUR decompositions which only takes advantage of linear but not multilinear structure in data tensors.  The authors therein raise the following open problem.
\begin{question}[{\cite[Remark 3.3]{MMD2008}}]
Can one  choose slabs and/or fibers to preserve some nontrivial multilinear tensor structure in the original tensor?
\end{question}

We address this question in this work, and show that the answer is affirmative.  Our methods are inspired by matrix CUR techniques, but are determined to account for the tensorial nature of data beyond simple unfolding based methods.

\section{Main Results}\label{SEC:Main}

Our main results are broken into three main themes: characterizations of modewise decompositions of tensors inspired by Theorem~\ref{THM:MatrixCUR}, perturbation analysis for the variants of tensor decompositions proposed here, and upper bounds for low multilinear rank tensor approximations via randomized sampling as well as randomized sampling guarantees for exact reconstruction of low multilinear rank tensors.

Now let us be concrete as to the type of tensor CUR decompositions we will consider here. The first is the most direct analogue of the matrix CUR decomposition in which a core subtensor is selected which we call $\mathcal{R}$, and the other factors are chosen by extruding $\mathcal{R}$ along each mode to produce various subtensors.  That is, given indices $I_i\subseteq[d_i]$, we set $\mathcal{R}=\mathcal{A}(I_1,\dots,I_n)$, $C_i=\mathcal{A}_{(i)}(:,\otimes_{j\neq i}I_j):=(\mathcal{A}(I_1,\cdots,I_{i-1},:,I_{i+1},\cdots,I_n))_{(i)}$, and $U_i=C_i(I_i,:)$.  This decomposition is illustrated in Figure \ref{FIG:TensorCUR}, and we call it the {\em Chidori CUR decomposition}\footnote{Chidori joints are used in contemporary Japanese woodworking and are based on an old toy design.  The joints bear a striking resemblance to Figure \ref{FIG:TensorCUR}.}.

\begin{figure}[ht]
\centering\includegraphics[width=.45\textwidth]{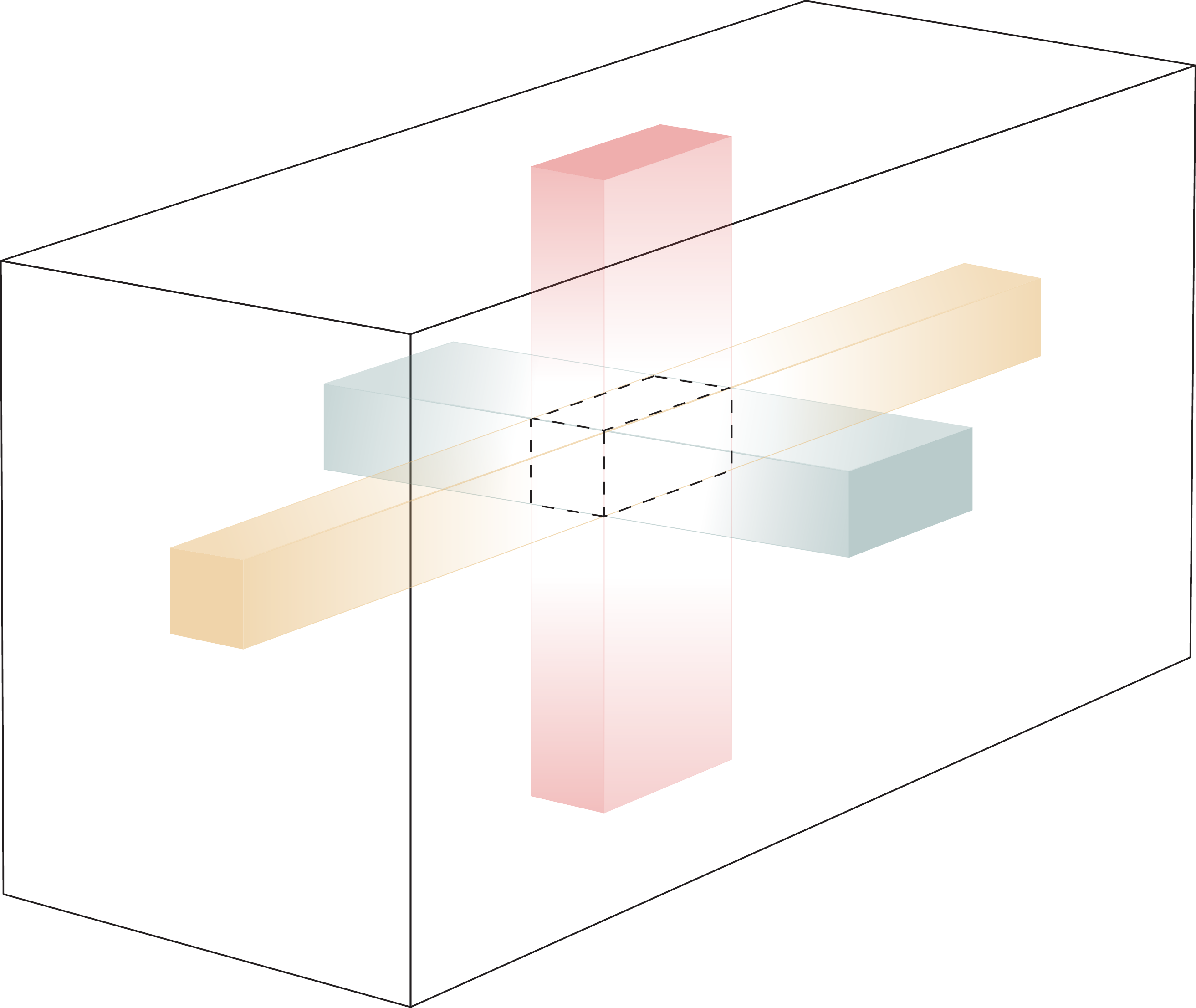}
\caption{Illustration of Chidori CUR decomposition \`{a} la Theorem~\ref{cor: CUR Char1} of a 3-mode tensor in the case when the indices $I_i$ are each an interval and $J_i=\otimes_{j\neq i} I_j$.  The matrix $C_1$ is obtained by unfolding the red subtensor along mode 1, $C_2$ by unfolding the green subtensor along mode 2, and $C_3$ by unfolding the yellow subtensor along mode 3.  The dotted line shows the boundaries of $\mathcal{R}$. In this case $U_i=\mathcal{R}_{(i)}$ for all $i$. 
}\label{FIG:TensorCUR}
\end{figure}

The second, more general tensor CUR decomposition discussed here, we call the Fiber CUR decomposition. In this case, to form $C_i$, one is allowed to choose $J_i\subseteq[\prod_{j\neq i}d_j]$ without reference to $I_i$.  Thus, the $C_i$ are formed from mode-$i$ fibers which may or may not interact with the core subtensor $\mathcal{R}$.  Fiber CUR decompositions are illustrated in Figure \ref{FIG:TensorCURIndependent}.

\begin{figure}[h!]
    \centering
    \includegraphics[width=.45\textwidth]{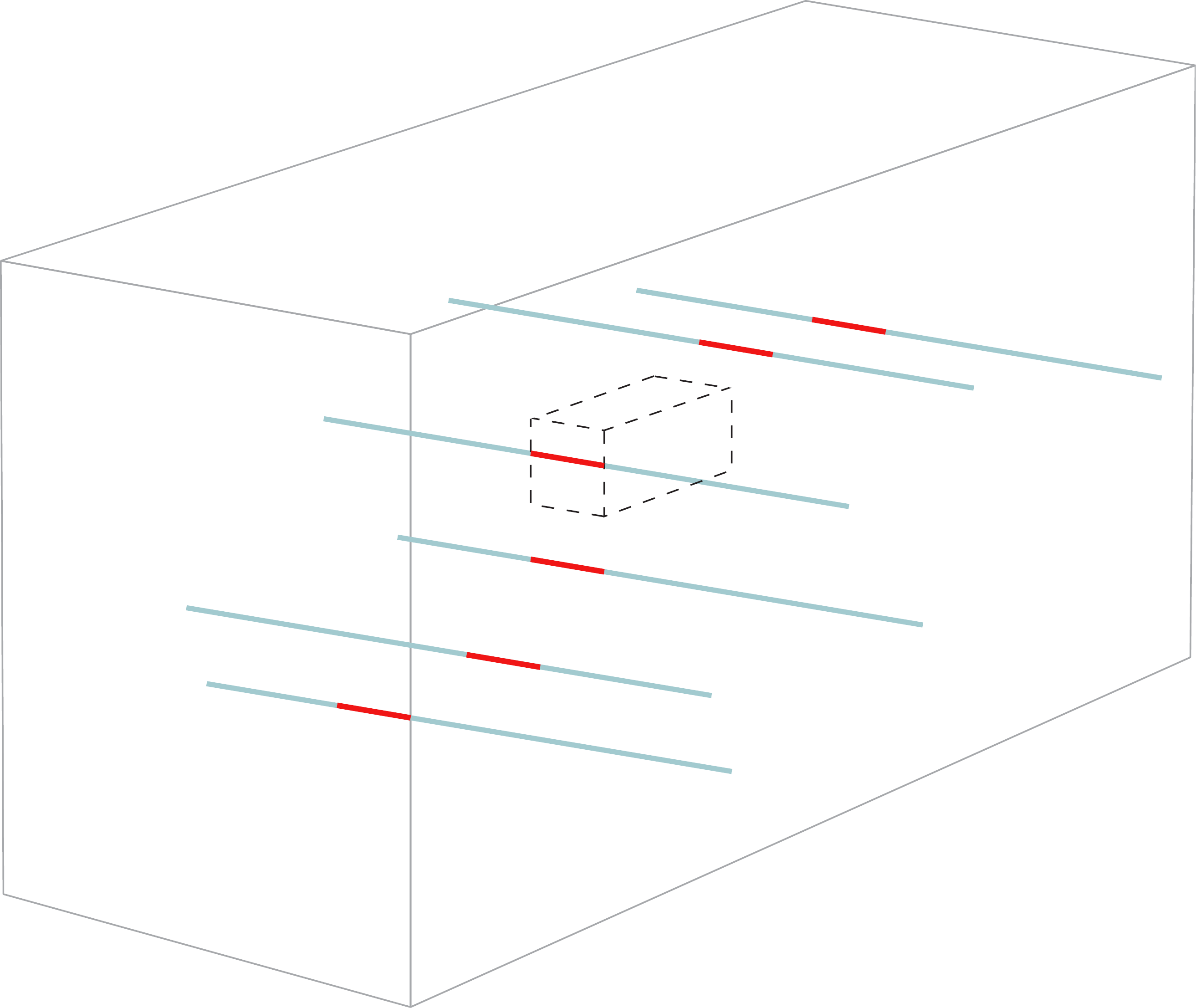}
    \caption{Illustration of the Fiber CUR Decomposition of Theorem~\ref{thm: CUR Char} in which $J_i$ is not necessarily related to $I_i$.  The lines correspond to rows of $C_2$, and red indices within  correspond to rows of $U_2$.  Note that the lines may (but do not have to) pass through the core subtensor $\mathcal{R}$ outlined by dotted lines.  Fibers used to form $C_1$ and $C_3$ are not shown for clarity.}
    \label{FIG:TensorCURIndependent}
\end{figure}

\subsection{Characterization Theorems}\label{SEC:Characterization}

First, we characterize Chidori CUR decompositions and compare with the matrix CUR decomposition of Theorem~\ref{THM:MatrixCUR}.

\begin{theorem}\label{cor: CUR Char1}
Let $\mathcal{A}\in\mathbb{R}^{d_1\times\cdots\times d_n}$ with multilinear rank $(r_1,\dots,r_n)$. Let $I_i\subseteq [d_i]$. Set $\mathcal{R}=\mathcal{A}(I_1,\cdots,I_n)$, $C_i=\mathcal{A}_{(i)}(:,\otimes_{j\neq i}I_j)$, and $U_i=C_i(I_i,:)$. Then the following are equivalent:
\begin{enumerate}[label=(\roman*)]
    \item  \label{CUR Char1:item1} $\rank(U_i)=r_i$,
    \item  \label{CUR Char1:item2} $\mathcal{A}=\mathcal{R}\times_1(C_1U_1^\dagger)\times_2\cdots\times_n(C_nU_n^\dagger)$,
    \item  \label{CUR Char1:item3} the multilinear rank of $\mathcal{R}$ is $(r_1,\cdots,r_n)$,
    \item  \label{CUR Char1:item4}   $\rank(\mathcal{A}_{(i)}(I_i,:))=r_i$ for all $i=1,\cdots,n$.
\end{enumerate}
Moreover, if any of the equivalent statements above hold, then $\mathcal{A}=\mathcal{A}\times_{i=1}^n (C_iC_i^\dagger)$. 
\end{theorem}

Note that for the indices chosen in Theorem~\ref{cor: CUR Char1}, we have $U_i=\mathcal{R}_{(i)}$. Interestingly enough, unlike the matrix case, the projection based decomposition $\mathcal{A}=\mathcal{A}\times_{i=1}^n (C_iC_i^\dagger)$ is not equivalent to the other parts of Theorem~\ref{cor: CUR Char1} as is shown in the following example.    

\begin{example}\label{exmp: non-eqi-CUR}
Let $\mathcal{A}\in\mathbb{R}^{3\times 3\times 2}$ with the following frontal slices: 
\[ A_1=\begin{bmatrix}
1 &2 &1\\
2&4&2\\
3&8&5
\end{bmatrix},~ A_2=\begin{bmatrix}
2 &5 &3\\
4&10&6\\
3&7&4
\end{bmatrix}.
\]
Then the multilinear rank of $\mathcal{A}$ is $(2, 2,2)$. Set $I_1=\{1,2\}$, $I_2=\{1,2\}$ and $I_3=\{1,2\}$. Then the frontal slices of $\mathcal{R}$ are 
\[R_1=\begin{bmatrix}
1 &2 \\
2&4
\end{bmatrix},~ R_2=\begin{bmatrix}
2 &5 \\
4&10
\end{bmatrix}.
\]
The multilinear rank of $\mathcal{R}$ is $(1,2,2)$. Thus the Chidori CUR decomposition does not hold, i.e., $\mathcal{A}\neq\mathcal{R}\times_1(C_1U_1^\dagger)\times_2(C_2U_2^\dagger)\times_3(C_3U_3^\dagger)$. However, $\rank(C_i)=2$ for $i=1,2,3$ which implies that $\mathcal{A}=\mathcal{A}\times_{i=1}^3(C_iC_i^\dagger)$.
\end{example}

Next we find that, in contrast to the matrix case, the indices $\{I_i\}$ and $\{J_i\}$ do not necessarily have to be correlated (see Figure \ref{FIG:TensorCURIndependent}). In the case that the indices $J_i$ are independent from $I_i$, we have the following characterization of the Fiber CUR decomposition.

\begin{theorem}\label{thm: CUR Char}
Let $\mathcal{A}\in\mathbb{R}^{d_1\times\cdots\times d_n}$ with multilinear rank $(r_1,\dots,r_n)$. Let $I_i\subseteq [d_i]$ and $J_i\subseteq[\prod_{j\neq i}d_j]$. Set $\mathcal{R}=\mathcal{A}(I_1,\cdots,I_n)$, $C_i=\mathcal{A}_{(i)}(:,J_i)$ and $U_i=C_i(I_i,:)$. 
Then the following statements are equivalent
\begin{enumerate}[label=(\roman*)]
    \item  \label{CUR Char:item1} $\rank(U_i)=r_i$,
    \item  \label{CUR Char:item2} $\mathcal{A}=\mathcal{R}\times_1(C_1U_1^\dagger)\times_2\cdots\times_n(C_nU_n^\dagger)$,
    \item  \label{CUR Char:item3} $\rank(C_i)=r_i$ for all $i$ and the multilinear rank of $\mathcal{R}=(r_1,\cdots,r_n)$.
\end{enumerate}
Moreover, if any of the equivalent statements above hold, then $\mathcal{A}=\mathcal{A}\times_{i=1}^n (C_iC_i^\dagger)$. 
\end{theorem}

The implications $(i)\Longrightarrow (ii)$ in Theorems~\ref{cor: CUR Char1} and \ref{thm: CUR Char} are essentially contained in \cite{caiafa2010generalizing}, though they consider tensors with constant multilinear rank $(r,\dots,r)$, and force the condition $|I_i|=|I_j|$ for all $i$ and $j$.  All other directions of these characterization theorems are new to the best of our knowledge, and it is of interest that the characterizations of each type of tensor CUR decomposition presented here (Fiber and Chidori) are different from the matrix case (e.g., the projection based version is no longer equivalent as Example \ref{exmp: non-eqi-CUR} shows).

\begin{remark}
\cite{MMD2008} consider tensor CUR decompositions for order-3 tensors of the form $\mathcal{R}\times_3(C_3U_3)$ (after translating to our notation). One can show that this is essentially a particular case of the Chidori CUR decomposition in which $C_1=U_1=\mathcal{A}_{(1)}$ and $C_2=U_2=\mathcal{A}_{(2)}$, although the matrix $U_3$ undergoes some additional scaling in their algorithm not present here.
\end{remark}

\subsection{Perturbation Bounds}\label{SEC:MainPerturbation}

The above characterizations are interesting from a mathematical viewpoint, but here we turn to more practical purposes and undertake an analysis of how well the decompositions mentioned above perform as low multilinear rank approximations to arbitrary tensors.  To state our results, we consider the additive noise model that we observe $\widetilde{\mathcal{A}} = \mathcal{A}+\mathcal{E}$, where $\mathcal{A}$ has low multilinear rank $(r_1,\dots,r_n)$ with $r_i<d_i$ for all $i$, and $\mathcal{E}$ is an arbitrary noise tensor. Our main concern is to address the following.

\begin{question}
If we choose the subtensors of $\widetilde{\mathcal{A}}$ in the manner of Section~\ref{SEC:Characterization}, how do the Chidori and Fiber CUR approximations of $\widetilde{\mathcal{A}}$ as suggested by Theorems~\ref{cor: CUR Char1} and \ref{thm: CUR Char} relate to Chidori and Fiber CUR decompositions of $\mathcal{A}$?
\end{question}

To understand the results answering this question, we first set some notation.  In what follows, tildes represent subtensors or submatrices of $\widetilde{\mathcal{A}}$, the letter $\mathcal{R}$ corresponds to core subtensors of the appropriate tensor, $C$ corresponds to either fibers or subtubes depending on if we are discussing Fiber or Chidori CUR, and $U$ corresponds to submatrices of $C$.  If $\mathcal{R}$, $C$, or $U$ appears without a tilde, it is a subtensor/matrix of $\mathcal{A}$.  In particular, consider
\begin{equation}\label{EQN:Tilde}
\begin{gathered}
    \widetilde{\mathcal{R}}=\mathcal{A}(I_1,\cdots,I_n),  \qquad \mathcal{E}_{\mathcal{R}}=\mathcal{E}(I_1,\cdots,I_n), \\
    \widetilde{C}_{i} = \widetilde{\mathcal{A}}_{(i)}(:, J_i), \qquad   E_{J_i}=\mathcal{E}_{(i)}(:, J_i),   \qquad  \widetilde{U}_{i} = \widetilde{C}_{i}(I_i,:),\qquad  E_{I_i,J_i}=E_{J_i}(I_i,:)
\end{gathered}
\end{equation}
for some index
sets $I_i\subseteq[d_i]$ and $J_i\subseteq[\prod_{j\neq i}d_j]$,  
and we write
\begin{equation}\label{EQN:Tilde2}
 \widetilde{\mathcal{R}}=  \mathcal{R}  +\mathcal{E}_{\mathcal{R}}, \qquad \widetilde{C}_{i}=C_i+E_{J_i}, \qquad  \widetilde{U}_{i}= U_i +   E_{I_i, J_i}, 
\end{equation} 
where $\mathcal{R}=\mathcal{A}(I_1,\dots,I_n)\in\R^{|I_1|\times\dots\times |I_n|}$, $C_i = \mathcal{A}_i(:, J_i)\in\R^{d_i\times |J_i|}$ and $U_i = C_i(I_i, :)\in\R^{|I_i|\times |J_i|}$. We also consider enforcing the rank on the submatrices $\widetilde{U}_i$ formed from the Chidori and Fiber CUR decompositions. Here, $\widetilde{U}_{i,r_i}$ is the best rank $r_i$ approximation of $\widetilde{U}_{i}$, and $\widetilde{U}_{i,r_i}^\dagger$ is its Moore--Penrose pseudoinverse. With these notations, our goal is to estimate the error 
\begin{equation}\label{EQN:Aapp}\mathcal{A}-\mathcal{A}_{\app}:=\mathcal{A}-\widetilde{\mathcal{R}}\times_{i=1}^n(\widetilde{C}_i\widetilde{U}_{i,r_i}^\dagger).
\end{equation} 
To measure accuracy, we consider the Frobenius norm of the difference of $\mathcal{A}$ and its low multilinear rank approximation $\mathcal{A}_{\app}$.  

\begin{theorem}\label{cor: pt_ana_tensor-cur}
Let $\widetilde{\mathcal{A}}=\mathcal{A}+\mathcal{E}$, where the multilinear rank of $\mathcal{A}$ is $(r_1,\cdots,r_n)$ and the compact SVD of $\mathcal{A}_{(i)}$ is  $\mathcal{A}_{(i)}=W_{i}\Sigma_i V_i^\top$. 
Let $I_i\subseteq [d_i]$ and $J_i\subseteq[\prod_{j\neq i}d_j]$. Invoke the notations of \eqref{EQN:Tilde}--\eqref{EQN:Aapp}, and
suppose that  $\sigma_{r_i}(U_i)>8\left\|E_{I_i,J_i}\right\|_2$ for all $i$. Then,
\begin{multline*}
     \left\|\mathcal{A}-\mathcal{A}_{\textnormal{app}}\right\|_F
     \leq\frac{9^n}{4^n}\left(\prod_{i=1}^n\left\|W_{i,I_i}^\dagger\right\|_2\right)\|\mathcal{E}_{\mathcal{R}}\|_F\\
     +\sum_{j=1}^n\frac{9^{n-j}}{4^{n-j}}\left\|\mathcal{R}_{(j)}\right\|_2\left(\prod_{i\neq j}\left\|W_{i,I_i}^\dagger\right\|_2\right)\left(5\left\|U_j^\dagger\right\|_2\left\|W_{j,I_j}^\dagger\right\|_2\left\|E_{I_j,J_j}\right\|_F+2\left\|U_j^\dagger\right\|_2\left\|E_{J_j}\right\|_F\right),
\end{multline*}
where $W_{i,I_i}:=W_{i}(I_i,:)$.
\end{theorem}
\RV{\begin{remark}
The bounds in Theorem~\ref{cor: pt_ana_tensor-cur} contain the exponential terms $\left(\frac{9}{4}\right)^{n}$, which is unavoidable, but the base number $\frac{9}{4}$ is not necessarily sharp. For more details see Equation \eqref{eqn: per_CU}.
\end{remark}
}

The error bounds in Theorem~\ref{cor: pt_ana_tensor-cur} are qualitative in that they are given in terms of the Frobenius norms of subtensors of the noise tensor $\mathcal{E}$.  These bounds can be applied generally, but can also be applied to give error estimates for approximating the HOSVD by low multilinear tensor approximation; this can be achieved by setting $\mathcal{A}$ to be the HOSVD approximation of a certain multilinear rank of $\widetilde{\mathcal{A}}$.  Note also that $\|\mathcal{R}_{(j)}\|_2\leq\|\mathcal{R}\|_F$ for all $j$ by basic norm inequalities, so if desired, this can be taken out of the summation for simpler error bounds.

The drawback of the above bounds is that they include norms of subtensors of $\mathcal{A}$, which can be difficult to estimate in general.  In the case that the indices $J_i$ depend on $I_i$, we can achieve more specialized bounds that can be expressed in terms of submatrices of singular vectors of the unfoldings of $\mathcal{A}$ as in the following theorem. 

\begin{theorem}
\label{thm: pt_ana_tensor-cur_tube}
Invoke the notations of Theorem~\ref{cor: pt_ana_tensor-cur} and \eqref{EQN:Tilde}--\eqref{EQN:Aapp}, and let $J_i = \otimes_{j\neq i}I_j$.  Suppose that $\sigma_{r_i}(U_i)>8\|E_{I_i,J_i}\|_2$ for every $1\leq i\leq n$,  and let the compact SVD of $\mathcal{A}_{(i)}$ be $W_i\Sigma_iV_i^\top$.  Then
\begin{multline*}
    \left\|\mathcal{A}-\mathcal{A}_{\app}\right\|_F\leq\frac{9^n}{4^n}\bigg(\prod_{i=1}^n\left\|W_{i,I_i}^\dagger\right\|_2\bigg)\|\mathcal{E}_{\mathcal{R}}\|_F \\
    + \sum_{j=1}^n\frac{9^{n-j}}{4^{n-j}}\left\|\mathcal{R}_{(j)}\right\|_2\bigg(\prod_{i\neq j}\left\|W_{i,I_i}^\dagger \right\|_2^2\bigg)\left\|\mathcal{A}_{(j)}^\dagger\right\|_2\left\|W_{j,I_j}^\dagger\right\|_2\left(5\left\|W_{j,I_j}^\dagger\right\|_2\left\|E_{I_j,J_j}\right\|_F+2\left\|E_{J_j}\right\|_F\right).
\end{multline*}
\end{theorem}

\begin{remark}
\RV{Notice that the bounds in Theorem \ref{cor: pt_ana_tensor-cur} and \ref{thm: pt_ana_tensor-cur_tube} are in terms of the submatrices of singular vectors of unfoldings of the tensor. These bounds are quite general, but they illustrate how one ought to sample $I_i$ and $J_i$ to ensure good bounds.  Theorems \ref{thm: per_uniform_sampling}  and \ref{cor:per_uniform_sampling}  give more user friendly versions of these bounds for the concrete case when a low multilinear rank tensor $\mathcal{A}$ is perturbed by noise.}

\RV{In the matrix case, there are error bounds for CUR decompositions in terms of the optimal low rank approximations of the given matrix. It is interesting, yet challenging, to derive upper bounds for tensor CUR decompositions in terms of the optimal low multilinear rank approximations of the given tensor.}
\end{remark}

The norms of pseudoinverses of submatrices of singular vectors can vary extensively depending on the sampling method.  In particular, if maximum volume sampling is used \cite{civril2009selecting,goreinov2010find,mikhalev2018rectangular}, then one can give generic bounds on these terms (see \cite{HH_Perturbation2019} for examples of upper bounds for maximum volume sampling for matrix CUR decompositions). However, maximum volume sampling is often intractable in practice.  However, if the matrices $\mathcal{A}_{(i)}$ have good incoherence, then uniform sampling yields submatrices $W_{i,J_i}^\dagger$ with small norm (cf. \cite[Lemma 3.4]{tropp2011improved}).

\subsection{Error Bounds for Random Sampling Core Subtensors}\label{SEC:Sampling}

Randomized sampling for column and row submatrices has been shown to be an effective method for low-rank approximation, and can provide a fast and reliable estimation of the SVD of a matrix \cite{frieze2004fast,DKMIII,RV2007,wang2013improving}.   First, let us state the formation of both Chidori and Fiber CUR decompositions via random sampling in algorithmic form here. 

\begin{algorithm}[h!]
\caption{Randomized Chidori CUR Decomposition}\label{ALGO:Chidori}
\begin{algorithmic}[1]
\STATE\textbf{Input:} {$\mathcal{A}\in\R^{d_1\times\cdots\times d_n}$, sample sizes $t_i,$ and probability distributions $\{p^{(i)}\}$ over $[d_i]$, $i=1,\dots,n$.}
\FOR{$i=1:n$}
\STATE{Sample $t_i$ indices from $[d_i]$ without replacement from $\{p^{(i)}\}$, and denote the index set $I_i$}
\STATE{$J_i = \otimes_{j\neq i} I_i$}
\STATE{$C_i = \mathcal{A}_{(i)}(:,J_i)$}
\STATE{$U_i = C_i(I_i,:)$}
\ENDFOR
\STATE{$\mathcal{R}=\mathcal{A}(I_1,\dots,I_n)$}
\STATE\textbf{Output: } $\mathcal{R}, C_i, U_i$ such that $\mathcal{A}\approx\mathcal{R}\times_{i=1}^n(C_iU_i^\dagger)$.
\end{algorithmic}
\end{algorithm}

\begin{algorithm}[h!]
\caption{Randomized Fiber CUR Decomposition}\label{ALGO:Fiber}
\begin{algorithmic}[1]
\STATE\textbf{Input:} {$\mathcal{A}\in\R^{d_1\times\cdots\times d_n}$, sample sizes $t_i, s_i$, and probability distributions $\{p^{(i)}\}, \{q^{(i)}\}$ over $[d_i]$ and $[\prod_{j\neq i}d_j]$, respectively, $i=1,\dots,n$.}
\FOR{$i=1:n$}
\STATE{Sample $t_i$ indices from $[d_i]$ without replacement from $\{p^{(i)}\}$, and denote the index set $I_i$}
\STATE{Sample $s_i$ indices from $[\prod_{j\neq i}d_j]$ without replacement from $\{q^{(i)}\}$, and denote the index set $J_i$}
\STATE{$C_i = \mathcal{A}_{(i)}(:,J_i)$}
\STATE{$U_i = C_i(I_i,:)$}
\ENDFOR
\STATE{$\mathcal{R}=\mathcal{A}(I_1,\dots,I_n)$}
\STATE\textbf{Output: }$\mathcal{R}, C_i, U_i$ such that $\mathcal{A}\approx\mathcal{R}\times_{i=1}^n(C_iU_i^\dagger)$.
\end{algorithmic}
\end{algorithm}

For general matrices, sampling columns with or without replacement from sophisticated distributions such as leverage scores is known to provide quality submatrices that represent the column space of the data faithfully \cite{DMPNAS}, but these distributions come at the cost of being expensive to compute.  On the other hand, uniform sampling of columns is cheap, but is not always reliable (for instance on extremely sparse matrices). Nonetheless, it is well understood that uniformly sampling column submatrices is both cheap and effective when the initial matrix is incoherent, e.g., \cite{talwalkar2010matrix,chiu2013sublinear}.  Here we extend these ideas  to tensors -- a task that first requires defining what tensor incoherence even is. 

\begin{definition}[{\cite{xia2021statistically}}]
Let $W\in\R^{d\times r}$ have orthonormal columns. Its coherence is defined as
\begin{equation*}
    \mu(W)=\frac{d}{r}\max_{1\leq i\leq d}\|W(i,:)\|_2^2.
\end{equation*}
For a tensor $\mathcal{A}\in\mathbb{R}^{d_1\times\cdots\times d_n}$ such that $\mathcal{A}_{(j)}=W_j\Sigma_j V_j^\top$ is its compact singular value decomposition with $\sigma_{j1},\cdots,\sigma_{jr_{j}}$ on the diagonal of $~\Sigma_j$, we define its coherence as
\begin{equation*}
\begin{aligned}
\mu(\mathcal{A})&:=\max\left\{\mu(W_1),\cdots,\mu(W_n) \right\},  \\
\end{aligned}
\end{equation*}
and define
\begin{equation*}
\begin{aligned}
\sigma_{\min}(\mathcal{A})&=\min\left\{\sigma_{1r_1},\cdots,\sigma_{nr_n}\right\},\\
\sigma_{\max}(\mathcal{A})&=\max\left\{\sigma_{11},\cdots,\sigma_{n1}\right\}.
\end{aligned}
\end{equation*}
\end{definition}

From the above definition, a tensor's incoherence is defined to be the maximal incoherence of all its unfoldings.  With this definition in hand, we may state our main result on low multilinear rank tensor approximation via uniform sampling. 

Note that in Theorem~\ref{thm: per_uniform_sampling} below, the incoherence of the tensor impacts how large the index sets $I_i$ must be to guarantee a good approximation; this is similar to existing matrix constructions \cite{candes2007sparsity,tropp2011improved,chiu2013sublinear}, and one should expect a (hopefully mild) oversampling factor for good approximation bounds in random sampling methods. 

\begin{theorem}\label{thm: per_uniform_sampling}
Let $\mathcal{A}\in\mathbb{R}^{d_1\times\cdots\times d_n}$ with multilinear rank $\left(r_1,\cdots,r_n\right)$ and  $\mathcal{A}_{(i)}=W_i\Sigma_i V_i^\top$ be the compact singular value decomposition of $\mathcal{A}_{(i)}$. Suppose that $\widetilde{\mathcal{A}}=\mathcal{A}+\mathcal{E}$.  Suppose that $\mathcal{A}_{\textnormal{app}}=\widetilde{\mathcal{R}}\times_{i=1}^n\left(\widetilde{C}_i\widetilde{U}_{i,r_i}^\dagger\right)$ is formed via the Chidori CUR decomposition (Algorithm \ref{ALGO:Chidori}) with input $\widetilde{\mathcal{A}}$, uniform probabilities, and  
$|I_i|\geq \gamma_{i}\mu(W_i)r_i$  for some $\gamma_{i}>0$.  
If $\delta\in[0,1)$ such that $(1-\delta)^{\frac{n}{2}}\sqrt{\frac{\prod_{i=1}^n|I_i|}{\prod_{i=1}^nd_i}} \sigma_{\min}(\mathcal{A})\geq 8\|E_{I_i,J_i}\|_2$, then
\begin{equation}\label{eqn:per_uniform}
\begin{split}
    \left\|\mathcal{A}-\mathcal{A}_{\textnormal{app}}\right\|_F&\leq \frac{9^n\sqrt{\prod_{i=1}^nd_i} }{4^n(1-\delta)^{\frac n 2}\sqrt{\prod_{i=1}^n|I_i|}}\|\mathcal{E}_{\mathcal{R}}\|_F\\
   &\quad~+\frac{9^{n}}{4^{n-1}}\frac{\sigma_{\max}(\mathcal{A})} {\sigma_{\min}(\mathcal{A})}\left(\frac{1+\eta}{1-\delta}\right)^{\frac{n}{2}}\sqrt{\prod_{i=1}^{n}\frac{d_i}{(1-\delta)|I_i|}}\|\mathcal{E}_{\mathcal{R}}\|_F \\
   & \quad~  +\frac{2\sigma_{\max}(\mathcal{A})} {\sigma_{\min}(\mathcal{A})}\left(\frac{1+\eta}{1-\delta}\right)^{\frac{n}{2}} \sum_{j=1}^n\frac{9^{n-j}}{4^{n-j}}\sqrt{\prod_{i\neq j}\frac{d_i}{(1-\delta)|I_i|}}
       \left\|E_{J_j}\right\|_F
      \end{split}
\end{equation}
 with probability at least $1-\sum\limits_{i=1}^n r_i\left(\left(\frac{e^{-\delta}}{(1-\delta)^{1-\delta}}\right)^{\gamma_{i}}+\left(\frac{e^{\eta}}{(1+\eta)^{1+\eta}}\right)^{\gamma_{i}}\right)$ for every $\eta\geq0$.
\end{theorem}

Let us provide a concrete choice of parameters for illustration of the bounds above: we assume a simple case where $d_i=d$ and $r_i=r$ for all $i$, we set $\delta=\eta=0.5$ and $\gamma=10\log(d)$, and consider $\mathcal{E}$ to have i.i.d.~Gaussian entries. The result is the following.

\begin{theorem}\label{cor:per_uniform_sampling}
Let $\mathcal{A}\in\mathbb{R}^{d\times\cdots\times d}$ be an $n$-mode tensor with multilinear rank $\left(r,\cdots,r\right)$ and  $\mathcal{A}_{(i)}=W_i\Sigma_i V_i^\top$ be the compact SVD of $\mathcal{A}_{(i)}$. Suppose that $\widetilde{\mathcal{A}}=\mathcal{A}+\mathcal{E}$ where   $\mathcal{E}$ is random tensor with i.i.d.~Gaussian distributed entries  with mean $0$ and variance $\sigma$ i.e., $ \mathcal{E}_{i_1,\cdots,i_n}\sim \mathcal{N}(0,\sigma)$. Suppose $I_i\subseteq[d]$  with $|I_1|=\cdots=|I_n|=\ell\geq 10\mu(\mathcal{A})r\log(d)$, and $\mathcal{A}_{\textnormal{app}}=\widetilde{\mathcal{R}}\times_{i=1}^n\left(\widetilde{C}_i\widetilde{U}_{i,r_i}^\dagger\right)$ is formed via the Chidori CUR decomposition (Algorithm \ref{ALGO:Chidori}) with input $\widetilde{\mathcal{A}}$ and uniform probabilities.  
Let $p$ be in the open interval $(1,1+\frac{\log(2)}{\log(\ell)})$. If $\ell^p\sigma_{\min}^2(\mathcal{A})\geq 2^{n+8} (\log(\ell^{n-1}+d)) d^{n}\sigma^2$, then
\begin{equation*}
\begin{aligned}
 \left\|\mathcal{A}-\mathcal{A}_{\textnormal{app}}\right\|_F&\leq\left(\frac{3^{2n}}{2^{\frac{3}{2}n-1}}+
      \frac{\sigma_{\max}(\mathcal{A})} {\sigma_{\min}(\mathcal{A})}\cdot \frac{ 3^{\frac{5n}{2}}}{2^{\frac{3n}{2}-4} } \right)
     \ell^{\frac{1-p}{2}} \sqrt{\log(\ell^{n-1}+d)}d^{\frac{n}{2}}\sigma
\end{aligned}
\end{equation*}
 with probability at least $1-\frac{2rn}{d}-\frac{2n}{(\ell^{n-1}+d)^{2\ell^{1-p}-1}}$. 
\end{theorem}

\begin{remark}
The conclusion of Theorem \ref{cor:per_uniform_sampling} holds unchanged for Rademacher random noise with each entry of $\mathcal{E}$ being $\pm\sigma$ with equal probability.  This is due to the use of \cite[Theorem 1.5]{tropp2012user} in the proof, which holds for both types of noise (see Section \ref{SEC:Sampling} for more detail).
\end{remark}

To conclude our results section, we show an example of how one can guarantee an exact decomposition of a low multilinear rank tensor with high probability via random sampling of indices.  The theorem is stated in terms of so-called column length sampling of the unfolded versions of the tensor.  Given a matrix $A$, we define probability distributions over its columns and rows, respectively, via
\[p_j(A):=\frac{\|A(j,)\|_2^2}{\|A\|_F^2},\qquad q_j(A):=\frac{\|A(:,j)\|_2^2}{\|A\|_F^2}.\]
Subsequently, for a tensor $\mathcal{A}$, we set
\[p_j^{(i)} := p_j(\mathcal{A}_{(i)}),\quad q_j^{(i)}:= q_j(\mathcal{A}_{(i)}).\]  
In the following theorems, $c>0$ is an absolute constant, and comes from the analysis of \cite{RV2007}.

\begin{theorem}\label{THM:UniformExactCUR}
Let $\mathcal{A}\in\mathbb{R}^{d_1\times  \cdots\times d_n}$ with multilinear rank  $(r_1,\cdots,r_n)$.  Let $~0<\varepsilon_i<\kappa(\mathcal{A}_{(i)})^{-1}$. Let $I_i\subseteq[d_i]$, $J_i\subseteq\left[\prod\limits_{j\neq i}d_j\right]$ satisfy
\begin{equation*}
		|I_i|\gtrsim \left(\frac{r_i\log(d_i)}{\varepsilon_i^4}\right)\log\left(\frac{r_i\log(d_i)}{\varepsilon_i^4}\right) , \qquad
		|J_i|\gtrsim \left(\frac{r_i\log(\prod_{j\neq i}d_j)}{\varepsilon_i^4}\right)\log\left(\frac{r_i\log(\prod_{j\neq i}d_j)}{\varepsilon_i^4}\right).
\end{equation*}
{and let $\mathcal{R},C_i,U_i$ be obtained from the Fiber CUR decomposition (Algorithm \ref{ALGO:Fiber}) with probabilities
$p^{(i)}_j$ and $q^{(i)}_j.$}
Then with probability at least \[\prod_{i=1}^n\left(1-\frac{2}{d_i^{c}}\right)\left(1-\frac{2}{\left(\prod_{j\neq i}d_j\right)^{c}}\right),\] we have $\rank(U_i)=r_i$, hence $\mathcal{A} = \mathcal{R}\times_{i=1}^{n} (C_iU_i^\dagger)$. 
\end{theorem}

\begin{theorem}\label{THM:ExactCURChidori}
Let $\mathcal{A}\in\mathbb{R}^{d_1\times  \cdots\times d_n}$ with multilinear rank  $(r_1,\cdots,r_n)$.  Let $0<\varepsilon_i<\kappa(\mathcal{A}_{(i)})^{-1}$. Let $I_i\subseteq[d_i]$  satisfy $
		|I_i|\gtrsim \left(\frac{r_i\log(d_i)}{\varepsilon_i^4}\right)\log\left(\frac{r_i\log(d_i)}{\varepsilon^4}\right)$.
Let $\mathcal{R}, C_i, U_i$ be obtained via the Chidori CUR decomposition (Algorithm \ref{ALGO:Chidori})
with probabilities $p^{(i)}_j$.
Then with probability at least $\prod_{i=1}^n\left(1-\frac{2}{d_i^{c}}\right)$, we have \textbf{$\mathcal{A} = \mathcal{R}\times_{i=1}^{n} (C_iU_i^\dagger)$}.
\end{theorem}

More results of a similar nature to Theorems~\ref{THM:UniformExactCUR} and \ref{THM:ExactCURChidori} can be obtained by extending the analysis done here to other sampling probabilities as was done in \cite{HH2020}, but we only state a sample of what one can obtain here for simplicity. \RV{One can readily state analogues of the above results for uniform sampling, or small perturbations of the distributions given above mimicking Theorem 5.1, and Corollaries 5.2, 5.4, and 5.5 of \cite{HH2020}, which would yield exact Fiber or Chidori CUR decompositions with high probability via uniform and leverage score sampling (here one would need to utilize the leverage score distributions for each of the unfoldings of $\mathcal{A}$, which would be computationally intensive). }

\subsection{Computational Complexity} \label{subsec:complxity section}

Let $\mathcal{A}\in\mathbb{R}^{d_1\times\cdots\times d_n}$ with multilinear rank $(r_1,\cdots,r_n)$. The complexity of randomized HOSVD decomposition is $\cO(r\prod_{i=1}^nd_i)$ with $r=\min_{i}\{r_i\}$. Whereas, computing the tensor-CUR approximation  only involves the computations of the pseudoinverse of $U_i$.
    \begin{enumerate}
        \item (Chidori CUR) If we sample $I_i$ uniformly with $|I_i|=\cO(r_i\log(d_i))$ and set $J_i=\otimes_{j\neq i}I_j$, then the complexity of computing the pseudoinverse of $U_i$ is $\cO(r_i\prod_{j=1}^n(r_j\log(d_j)))$.
        \item (Fiber CUR) If we sample $I_i$ and $J_i$ uniformly, then the size of $I_i$ and $J_i$ should be $\cO(r_i\log(d_i))$ and $\cO\left(r_i\log(\prod_{j\neq i}d_j)\right)$. Thus the complexity of computing the pseudoinverse of $U_i$ is $\cO\left(r_i^2\log(d_i)\left(\sum_{j\neq i}\log(d_j)\right)\right)$.
    \end{enumerate}

\paragraph{Conversion of tensor CUR to HOSVD:}

The low multilinear rank approximations discussed here are in terms of modewise products of a core subtensor of the data tensor and matrices formed by subsampling unfolded versions of the tensor, but if a user wishes to obtain an approximation in Tucker format, it is easily done via Algorithm \ref{ALG:Conversion}.  Note that converting from HOSVD to a tensor CUR decomposition is not as straightforward.

\begin{algorithm}[h!]
\caption{Efficient conversion from CUR to HOSVD}\label{ALG:Conversion} \label{Algo:cur2svd}
\begin{algorithmic}[1]
\STATE \textbf{Input: }{$\mathcal{R},C_i,U_i$: CUR decomposition of the tensor $\mathcal{A}$.}
\STATE{$[Q_i,R_i]=\mathrm{qr}(C_iU_i^\dagger)$ for $i=1,\cdots,n$}
\STATE{$\mathcal{T}_1=\mathcal{R}\times_1 R_1\times_2\cdots\times_n R_n$}
\STATE{Compute HOSVD of $\mathcal{T}_1$ to find $\mathcal{T}_1=\mathcal{T}\times_1 V_1 \times_2\cdots\times_n V_n$}
\STATE \textbf{Output: }{$\llbracket \mathcal{T};Q_1V_1,\cdots,Q_nV_n\rrbracket$: HOSVD decomposition of $\mathcal{A}$.}
\end{algorithmic}
\end{algorithm}

\section{Proofs of Characterization Theorems}\label{SEC:CharacterizationProof}

Here we provide the proofs of the characterization theorems in Section~\ref{SEC:Characterization}.  We do so in reverse order to the presentation due to the fact that Theorem~\ref{thm: CUR Char} is the more general statement, and thus implies some of the statements of Theorem~\ref{cor: CUR Char1}.

\begin{proof}[Proof of Theorem~\ref{thm: CUR Char}]
$\ref{CUR Char:item1}\Longrightarrow\ref{CUR Char:item2}$: A special case of this was proved by  \cite{caiafa2010generalizing}, but we give a complete and simplified proof here.

Since $\rank(U_1)=r_1$, we have $\mathcal{A}_{(1)}=C_1U_1^\dagger (\mathcal{R}_1)_{(1)}$ with $\mathcal{R}_1\in\mathbb{R}^{|I_1|\times d_2\times\cdots \times d_n}$ defined to be $\mathcal{A}(I_1,:,\cdots)$, then we have $\mathcal{A}=\mathcal{R}_1\times_1(C_1U_1^\dagger) $. 
   Similarly, we have
    \begin{eqnarray*}
\mathcal{A}_{(2)}=C_2U_2^\dagger(\mathcal{A}(:,I_2,:,\cdots))_{(2)},
    \end{eqnarray*}
    which implies that
    \begin{eqnarray*}
(\mathcal{R}_1)_{(2)}=(\mathcal{A}(I_1,:,\cdots))_{(2)}&=&C_2U_2^\dagger(\mathcal{A}(I_1,I_2,:,\cdots))_{(2)}.
    \end{eqnarray*} 
    Now set $\mathcal{R}_2=\mathcal{A}(I_1,I_2,:,\cdots)$.
    Thus,
    $(\mathcal{R}_1)_{(2)}=C_2U_2^\dagger (\mathcal{R}_2)_{(2)}$ and $\mathcal{R}_2=\mathcal{R}_1\times_2(C_2U_2^\dagger) $. Continuing in this manner, we have $\mathcal{R}_{n-1}=\mathcal{R}_{n}\times_n(C_nU_n^\dagger) $. Thus we have
  \begin{equation}
      \mathcal{A}=\mathcal{R}_1\times_1(C_1U_1^\dagger) =\cdots=\mathcal{R}\times_1(C_1U_1^\dagger) \times_2\cdots\times_n(C_nU_n^\dagger) .
  \end{equation}

$\ref{CUR Char:item2}\Longrightarrow\ref{CUR Char:item3}$: Since $\mathcal{R}$ and $C_i$ only contain part of $\mathcal{A}$, we have multilinear rank of $\mathcal{R}\leq$ multilinear rank of $\mathcal{A}=(r_1,\cdots,r_n)$ and
$\rank(C_i)\leq\rank(\mathcal{A}_{(i)})=r_i$. In addition,  $\mathcal{A}=\mathcal{R}\times_1(C_1U_1^\dagger) \times_2\cdots\times_n(C_nU_n^\dagger),$ so $r_i=\rank(\mathcal{A}_{(i)})\leq\min\{\rank(\mathcal{R}_{(i)}),\rank(C_i)\}$. Therefore, $\ref{CUR Char:item3}$ holds.

$\ref{CUR Char:item3}\Longrightarrow\ref{CUR Char:item1}$ Notice that since $U_i$ is a submatrix of $C_i$, we have $\rank(U_i)\leq r_i$. Since $\rank(\mathcal{R}_{(i)})=r_i$ and $\mathcal{R}_{(i)}$ is a submatrix of $(\mathcal{A}(\cdots,I_i,\cdots))_{(i)}$, we have  $$\rank((\mathcal{A}(\cdots,I_i,\cdots))_{(i)})=r_i.$$ By Theorem~\ref{THM:MatrixCUR}, we see that $\rank(U_i)=r_i$, completing the proof.
\end{proof}

Now we turn to the proof of Theorem~\ref{cor: CUR Char1}, which handles the case when the column indices $J_i$ of $C_i$ are chosen to be $J_i=\otimes_{j\neq i}I_j$ i.e., $C_i=(\mathcal{A}(I_1,\cdots,I_{i-1},:,I_{i+1},\cdots,I_n))_{(i)}$.

\begin{proof}[Proof of Theorem~\ref{cor: CUR Char1}]
Evidently, the equivalence $\ref{CUR Char1:item1}\Longleftrightarrow\ref{CUR Char1:item2}$ follows as a special case of Theorem~\ref{thm: CUR Char}.  Similarly, we have $\ref{CUR Char1:item2}\Longrightarrow \ref{CUR Char1:item3}$.  

To see $\ref{CUR Char1:item3}\Longrightarrow\ref{CUR Char1:item2}$, note that since $\mathcal{A}$ has multilinear rank $\left(r_1,\cdots,r_n\right)$, $\mathcal{A}$ has HOSVD decomposition $\mathcal{A}=\mathcal{T}\times_1 W_1 \times_2 \cdots\times_n W_n $ such that $\mathcal{T}\in\mathbb{R}^{r_1\times\cdots\times r_n }$ satisfies $\rank(\mathcal{T}_{(i)})=r_i$ and $W_i\in\mathbb{R}^{d_i\times r_i}$  has rank $r_i$. By the condition $\rank(\mathcal{A}_{(i)}(I_i,:))=r_i$, we have that
$\rank\left(W_i(I_i,:)\mathcal{T}_{(i)}\left( \kron_{j\neq i}W_j^\top\right)\right)=r_i$. Thus $\rank(W_i(I_i,:))=r_i$. Notice that
\begin{equation*}
    \begin{aligned}
    C_i&=\mathcal{A}_{(i)}(:\otimes_{j\neq i}I_j)=W_i\mathcal{T}_{(i)}\left( \kron_{j\neq i}(W_j(I_j,:))\right)^\top.
    \end{aligned}
\end{equation*}
Sylvester’s rank inequality implies that
\begin{equation*}
    \begin{aligned}
    \rank(C_i)&=\rank(W_i\mathcal{T}_{(i)}\left( \kron_{j\neq i}(W_j(I_j,:))\right)^\top)\\
    &\geq \rank(W_i\mathcal{T}_{(i)})+\rank\left( \kron_{j\neq i}(W_j(I_j,:)) \right)-\prod_{j\neq i}r_j\\
    &= r_i+\prod_{j\neq i}r_j-\prod_{j\neq i}r_j = r_i.
    \end{aligned}
\end{equation*}
The proof that $\mathcal{R}$ has multilinear rank $\left(r_1,\cdots,r_n\right)$ is similar, which yields $\ref{CUR Char1:item4}\Longrightarrow\ref{CUR Char1:item3}$. The condition on the rank of $C_i$ and $\mathcal{R}$ imply that  $\mathcal{A}=\mathcal{R}\times_{i=1}^{n}(C_iU_i^\dagger)$ by Theorem~\ref{thm: CUR Char}.

To see $\ref{CUR Char1:item1}\Longrightarrow \ref{CUR Char1:item4}$, notice that $U_i=C_i(I_i,:)=\mathcal{A}_{(i)}(I_i,\kron_{j\neq i}I_j)$ which is a submatrix of $\mathcal{A}_{(i)}(I_i,:)$. Since $\rank(U_i)=r_i$, we have $\rank(\mathcal{A}_{(i)}(I_i,:))\geq\rank(U_i)=r_i$. Additionally, $\rank(\mathcal{A}_{(i)})=r_i$, so $\rank(\mathcal{A}_{(i)}(I_i,:))=r_i$. Thus, $\ref{CUR Char1:item4}$ holds.
\end{proof}

\section{Perturbation Analysis}\label{SEC:PerturbationProof}

Real data tensors are rarely exactly low rank, but can be modeled as low multilinear rank data + noise. In this section, we give proofs for the perturbation analysis expoused in Section~\ref{SEC:MainPerturbation}. Our primary task is to consider a tensor of the form:
\[ \widetilde{\mathcal{A}}=\mathcal{A}+\mathcal{E},
\]
where $\mathcal{A}\in\mathbb{R}^{d_1\times\cdots\times d_n}$ has low multilinear rank $(r_1,\cdots,r_n)$ with $r_i<d_i$, and $\mathcal{E}\in\mathbb{R}^{d_1\times\cdots\times d_n}$ is an arbitrary noise tensor.

\subsection{Ingredients for the Proof of Theorem~\ref{cor: pt_ana_tensor-cur}}

The main ingredients for proving Theorem~\ref{cor: pt_ana_tensor-cur} is the following.

\begin{theorem}\label{thm: pt_ana_tensor-cur}
Let $\widetilde{\mathcal{A}}=\mathcal{A}+\mathcal{E}\in\R^{d_1\times\dots\times d_n}$, where the multilinear rank of $\mathcal{A}$ is $(r_1,\dots,r_n)$.  
Let $I_i\subseteq [d_i]$ and $J_i\subseteq[\prod_{j\neq i}d_j]$. Invoke the notations of \eqref{EQN:Tilde}--\eqref{EQN:Aapp}, and suppose that $\sigma_{r_i}(U_i)>8\|E_{I_i,J_i}\|_2$ for all $i$. Then, 
\begin{multline*}
     \left\|\mathcal{A}-\mathcal{A}_{\app}\right\|_F
     \leq\frac{9^n}{4^n}\left(\prod_{i=1}^n\left\|C_iU_i^\dagger\right\|_2\right)\|\mathcal{E}_{\mathcal{R}}\|_F\\
     +\sum_{j=1}^n\frac{9^{n-j}}{4^{n-j}}\left\|\mathcal{R}_{(j)}\right\|_2\left(\prod_{i\neq j}\left\|C_iU_i^\dagger\right\|_2\right)\left(5\left\|U_j^\dagger\right\|_2\left\|C_jU_j^\dagger\right\|_2\left\|E_{I_j,J_j}\right\|_F+2\left\|U_j^\dagger\right\|_2\left\|E_{J_j}\right\|_F\right).
\end{multline*}
\end{theorem}

\begin{proposition}[{\cite[Proposition 3.1]{HH_Perturbation2019}}]\label{PROP:CUWI}
Suppose that $A\in\mathbb{R}^{m\times n}$ has rank $r$  and its compact SVD is $A=W_A\Sigma_A V_A^\top$. Let $C$, $U$, and $R$ be the submatrices of $A$ (with selected row and column indices $I$ and $J$, respectively) such that $A=CU^\dagger R$. Then 
\[ \|CU^\dagger\|_2=\|W_{A,I}^\dagger\|_2, 
\]
where $W_{A,I}=W_A(I,:)$.
\end{proposition}

\begin{proof}[Proof of Theorem~\ref{cor: pt_ana_tensor-cur}]
Combining Theorem~\ref{thm: pt_ana_tensor-cur} with Proposition~\ref{PROP:CUWI} directly yields the conclusion of Theorem~\ref{cor: pt_ana_tensor-cur}.
\end{proof}


\begin{proof}[Proof of Theorem~\ref{thm: pt_ana_tensor-cur_tube}]
Note that when the column index  $J_i$ of $C_i$ can be written as $J_i=\otimes_{j\neq i}I_j$, we have $U_i=\mathcal{R}_{(i)}$. If $\mathcal{A}$ has HOSVD $\mathcal{A}=\mathcal{T}\times_1 W_1 \times_2\cdots\times_n W_n $, then $U_i=W_{i,I_i}\mathcal{T}_{(i)}\left( \kron_{j\neq i}W_{j,I_j}\right)^\top$. Thus $\left\|U_i^\dagger\right\|_2\leq \left\|\mathcal{T}_{(i)}^\dagger\right\|_2\prod_{j=1}^n\left\|W_{j,I_j}^\dagger\right\|_2=  \left\|\mathcal{A}_{(i)}^\dagger\right\|_2\prod_{j=1}^n\left\|W_{j,I_j}^\dagger\right\|_2$. Combining this bound with Theorem~\ref{cor: pt_ana_tensor-cur} yields the required bound.
\end{proof}

\subsection{Ingredients for the proof of Theorem~\ref{thm: pt_ana_tensor-cur}}

To prove Theorem~\ref{thm: pt_ana_tensor-cur}, we first need some preliminary observations.

\begin{lemma}\label{LEM:Cai} Let $A$
and $B$
be two rank $r$ matrices, then
\[ \left\|AA^\dagger-BB^\dagger \right\|_2\leq\frac{\left\|A-B\right\|_2}{\sigma_r(A)},\quad \left\|A^\dagger A-B^\dagger B\right\|_2\leq\frac{\|A-B\|_2}{\sigma_r(A)},\] and
\[ \left\|AA^\dagger-BB^\dagger \right\|_F\leq \frac{\sqrt{2}\|A-B\|_F}{\sigma_r(A)}, \quad \left\|A^\dagger A-B^\dagger B\right\|_F\leq\frac{\sqrt{2}\|A-B\|_F}{\sigma_r(A)}. \]
\end{lemma}

\begin{proof}
Let $A=W\Sigma V^\top$ be the compact SVD of $A$. Then $AA^\dagger=W\Sigma V^\top V\Sigma^{-1}W^\top=WW^\top$. Similarly, $A^\dagger A=VV^\top$.  Then, we achieve the claims by applying \cite[Lemma~6]{CCW2019}.
\end{proof}

Prior to stating a corollary of this lemma, note that $E_{I,J}=E(I,J)$ in the matrix case.  Likewise, note that $\widetilde{U}_r$ is the best rank--$r$ approximation of $\widetilde{U}$, and $\widetilde{U}_r^\dagger$ is the Moore--Penrose pseudoinverse of $\widetilde{U}_r$.

\begin{corollary}\label{Cor: err_diff}
Suppose that  $A\in\mathbb{R}^{m\times n}$ has rank $r$ with compact SVD $A=W\Sigma V^*$ and $\widetilde{A}=A+E$. Let $\widetilde{C}=\widetilde{A}(:,J)$ and $\widetilde{U}=\widetilde{A}(I,J)$ with $I\subseteq[m]$ and $J\subseteq[n]$. 
If $\sigma_r(U)>4\left\|E_{I,J}\right\|_2$, then 
\[\left\|\widetilde{C}\widetilde{U}_{r}^\dagger-CU^\dagger\right\|\leq\dfrac{\left\|U^\dagger\right\|_2\left(\left\|CU^\dagger\right\|_2\left\|E_{I,J}\right\|+\|E_J\|\right)}{1-4\left\|U^\dagger\right\|_2\left\|E_{I,J}\right\|_2}+2\sqrt{2}\left\|CU^\dagger\right\|_2\left\|U^\dagger\right\|_2\left\|E_{I,J}\right\|.\] Here $\|\cdot\|$ can be $\|\cdot\|_2$ or $\|\cdot\|_F$.
\end{corollary}
\begin{proof}
Since $\sigma_r(U)>4\|E_{I,J}\|_2$, we have $\sigma_r(\widetilde{U})\geq\sigma_r(U)-\|E_{I,J}\|_2>3\|E_{I,J}\|_2\geq 0$.
Thus $\rank(\widetilde{U}_r)=r$. 
Notice that $C=CU^\dagger U$. Then
\begin{align*}
\|\widetilde{C}\widetilde{U}_{r}^\dagger-CU^\dagger\|&\leq\|C\widetilde{U}^\dagger_r-CU^\dagger\|+\|E_{J}\widetilde{U}_r^\dagger\|\\
&=\|CU^\dagger U\widetilde{U}^\dagger_r-CU^\dagger UU^\dagger\|+\|E_{J}\widetilde{U}_r^\dagger\|\\
&\leq \|CU^\dagger\|_2\|(\widetilde{U}-E_{I,J})\widetilde{U}_r^\dagger-UU^\dagger\|+\|E_J\widetilde{U}_r^\dagger\|\\
&\leq \|CU^\dagger\|_2\|\widetilde{U}\widetilde{U}^\dagger_r-UU^\dagger\|+\|CU^\dagger\|_2\|\widetilde{U}^\dagger_r\|_2\left\|E_{I,J}\right\|+\|\widetilde{U}_r^\dagger\|_2\|E_{J}\|\\
&\leq \sqrt{2}\|CU^\dagger\|_2\|\widetilde{U}_r-U\|\|U^\dagger\|+\|CU^\dagger\|_2\|\widetilde{U}^\dagger_r\|_2\left\|E_{I,J}\right\|+\|\widetilde{U}_r^\dagger\|_2\|E_{J}\|\\
&\leq 2\sqrt{2}\left\|CU^\dagger\right\|_2\left\|U^\dagger\right\|_2\left\|E_{I,J}\right\|+\frac{ \left\|U^\dagger\right\|_2\left(\left\|CU^\dagger\right\|_2\left\|E_{I,J}\right\|+\|E_{J}\|\right)}{1-4\left\|U^\dagger\right\|_2\left\|E_{I,J}\right\|_2},
\end{align*}
where the penultimate inequality uses Lemma~\ref{LEM:Cai}.
We also used the facts that $\widetilde{U}\widetilde{U}_r^\dagger = \widetilde{U}_r\widetilde{U}_r^\dagger$ (similar to \cite{HH_Perturbation2019}), and that $\|\widetilde{U}_r-U\|\leq \|\widetilde{U}_r-\widetilde{U}\| + \|\widetilde{U}-U\|\leq 2\|E_{I,J}\|$. 
\end{proof}

\begin{lemma}[{\cite[Lemma 8.3]{HH_Perturbation2019}}]\label{lmm: per_CU}
Suppose that $A\in\mathbb{R}^{m\times n}$ has rank $r$ with compact SVD $A=W\Sigma V^\top$ and $\widetilde{A}=A+E$. Let $\widetilde{C}=\widetilde{A}(:,J)$ and $\widetilde{U}=\widetilde{A}(I,J)$ with $I\subseteq[m]$ and $J\subseteq[n]$. 
If $\sigma_r(U)>4\left\|E_{I,J}\right\|_2$, then 
\begin{equation*}
\left\|\widetilde{C}\widetilde{U}_{r}^\dagger\right\|\leq \left\|CU^\dagger\right\|\left(1+\frac{\left\|U^\dagger\right\|_2\left\|E_{I,J}\right\|_2}{1-4\left\|U^\dagger\right\|_2\left\|E_{I,J}\right\|_2}\right)+\dfrac{\left\|U^\dagger\right\|_2\|E_J\|}{1-4\left\|U^\dagger\right\|_2\left\|E_{I,J}\right\|_2}.
\end{equation*}
\end{lemma}

 Now, we are in the stage to prove Theorem~\ref{thm: pt_ana_tensor-cur}.

\begin{proof}[Proof  of Theorem~\ref{thm: pt_ana_tensor-cur}]
Notice that $\sigma_{r_i}(U_i)>8\|E_{I_i,J_i}\|_2\geq 0$ implies that $ \rank(U_i)=r_i$, thus $\mathcal{A}=\mathcal{R}\times_{i=1}^n\left(C_iU_i^\dagger\right)$ by Theorem~\ref{thm: CUR Char}. Therefore,
\begin{align*}
&\quad~\left\|\mathcal{A}-\mathcal{A}_{\textnormal{app}}\right\|_F\\
&=\left\|\mathcal{R}\times_{i=1}^n\left(C_iU_i^\dagger\right) -\widetilde{\mathcal{R}}\times_{i=1}^n\left( \widetilde{C}_i\widetilde{U}_{i,r_i}^\dagger\right) \right\|_F\\
&\leq\left\|(\mathcal{R}-\widetilde{\mathcal{R}})\times_{i=1}^n (\widetilde{C}_i\widetilde{U}_{i,r_i}^\dagger)\right\|_F+\sum_{j=1}^n\left\|\mathcal{R}\times_{i=1}^{j-1}C_iU^\dagger_i\times (C_jU^\dagger_j-\widetilde{C}_{j}\widetilde{U}^\dagger_{j,r_j})\times_{i=j+1}^n\widetilde{C}_{i}\widetilde{U}^\dagger_{i,r_i}\right\|_F\\
&\leq\|\mathcal{E}_{\mathcal{R}}\|_F\prod_{i=1}^n\left\|\widetilde{C}_i\widetilde{U}_{i,r}^\dagger\right\|_2+\sum_{j=1}^n\left\|\mathcal{R}_{(j)}\right\|_2\left(\prod_{i=1}^{j-1}\left\|C_iU_i^\dagger\right\|_2\right)\left\|C_jU^\dagger_j-\widetilde{C}_{j}\widetilde{U}^\dagger_{j,r_j}\right\|_F\left(\prod_{i=j+1}^{n}\left\|\widetilde{C}_i\widetilde{U}_{i,r_i}^\dagger\right\|_2\right).
\end{align*}
By Lemma~\ref{lmm: per_CU}, we have
\begin{equation}\label{eqn: per_CU}
\begin{split}
    \left\|\widetilde{C}_i\widetilde{U}_{i,r}^\dagger\right\|_2\leq&~   \left\|C_iU_i^\dagger\right\|_2\left(1+\frac{\|U_i^\dagger\|_2\left\|E_{I_i,J_i}\right\|_2}{1-4\|U_i^\dagger\|_2\left\|E_{I_i,J_i}\right\|_2}\right)+\dfrac{\|U_i^\dagger\|_2\|E_{J_i}\|_2}{1-4\|U_i^\dagger\|_2\|E_{I_{i},J_{i}}\|_2}\\
    \leq&~\frac{5}{4}\left\|C_iU_i^\dagger\right\|_2+\frac{1}{4}\\
    \leq&~\frac{9}{4}\left\|C_iU_i^\dagger\right\|_2.
    \end{split}
\end{equation}
By Corollary~\ref{Cor: err_diff},
\begin{equation}\label{eqn:err_diff}
    \begin{aligned}
        &\quad~\left\|C_jU^\dagger_j-\widetilde{C}_{j}\widetilde{U}^\dagger_{j,r_j}\right\|_F \\
        &\leq\dfrac{\left\|U_j^\dagger\right\|_2\left(\left\|C_jU_j^\dagger\right\|_2\left\|E_{I_j,J_j}\right\|_F+\left\|E_{J_j}\right\|_F\right)}{1-4\left\|U_j^\dagger\right\|_2\left\|E_{I_j,J_j}\right\|_2}+2\sqrt{2}\left\|C_jU_j^\dagger\right\|_2\left\|U_j^\dagger\right\|_2\left\|E_{I_j,J_j}\right\|_F\\
        &\leq 2\left\|U_j^\dagger\right\|_2\left(\left\|C_jU_j^\dagger\right\|_2\left\|E_{I_j,J_j}\right\|_F+\left\|E_{J_j}\right\|_F\right)+2\sqrt{2}\left\|C_jU_j^\dagger\right\|_2\left\|U_j^\dagger\right\|_2\left\|E_{I_j,J_j}\right\|_F\\
        &\leq 5\left\|U_j^\dagger\right\|_2\left\|C_jU_j^\dagger\right\|_2\left\|E_{I_j,J_j}\right\|_F+2\left\|U_j^\dagger\right\|_2\left\|E_{J_j}\right\|_F.
    \end{aligned}
\end{equation}
Combining \eqref{eqn: per_CU} and \eqref{eqn:err_diff}, we get
\begin{multline*}
     \left\|\mathcal{A}-\mathcal{A}_{\textnormal{app}}\right\|_F
     \leq\frac{9^n}{4^n}\left(\prod_{i=1}^n\left\|C_iU_i^\dagger\right\|_2\right)\|\mathcal{E}_{\mathcal{R}}\|_F \\
     +\sum_{j=1}^n\frac{9^{n-j}}{4^{n-j}}\left\|\mathcal{R}_{(j)}\right\|_2\left(\prod_{i\neq j}\left\|C_iU_i^\dagger\right\|_2\right)\left(5\left\|U_j^\dagger\right\|_2\left\|C_jU_j^\dagger\right\|_2\left\|E_{I_j,J_j}\right\|_F+2\left\|U_j^\dagger\right\|_2\left\|E_{J_j}\right\|_F\right).
\end{multline*}
This concludes the proof.
\end{proof}

\section{Random Sampling Methods}\label{SEC:SamplingProof}
The main purpose of this section is to prove the results of Section~\ref{SEC:Sampling}. From Theorem~\ref{thm: CUR Char} and \ref{cor: pt_ana_tensor-cur}, we see that the existence of a tensor-CUR decomposition and the error for a tensor-CUR approximation both depend crucially on the chosen indices $I_i$ and $J_i$. Here, we study randomized sampling procedures for selecting these indices to determine how the sampling affects the quality of the tensor approximation, and when it can guarantee an exact decomposition when $\mathcal{A}$ has low multilinear rank. 

\subsection{Guaranteeing Exact Decompositions}

We begin by answering the question: when does random sampling guarantee an exact tensor-CUR decomposition for low multilinear rank $\mathcal{A}$.  While this is not overly important in practice, we are able to use this analysis to provide approximation guarantees in the general case. First, we review a related sampling result for matrix CUR decompositions.  

\begin{lemma}[{\cite[Theorem 4.1]{HH2019}}]\label{THM:ColRowChnM}Let $A\in\mathbb{R}^{m\times n}$ have rank $r$.  Let   $0<\varepsilon<\kappa(A)^{-1}$. Let $d_1\in[m]$, $d_2\in[n]$ satisfy
\begin{equation*}
		d_1\gtrsim \left(\frac{r\log(m)}{\varepsilon^4}\right)\log\left(\frac{r\log(m)}{\varepsilon^4}\right), \qquad d_2\gtrsim \left(\frac{r\log(n)}{\varepsilon^4}\right)\log\left(\frac{r\log(n)}{\varepsilon^4}\right).
\end{equation*}
Choose $I\subseteq[m]$ by sampling $d_1$ rows of $A$ independently without replacement according to probabilities $q_j^{(i)}$ and choose $J\subseteq[n]$ by sampling $d_2$ columns of $A$ independently with replacement according to  $p_j^{(i)}$.  Set $R=A(I,:)$, $C=A(:,J)$, and $U=A(I,J)$. Then with probability at least $(1-\frac{2}{n^{c}})(1-\frac{2}{m^{c}})$,
\[ \rank(U) = r. \]
\end{lemma}

\begin{proof}[Proof of Theorem~\ref{THM:UniformExactCUR}]
Lemma~\ref{THM:ColRowChnM} applied to each $C_i, U_i$ implies condition $\ref{CUR Char:item1}$ of Theorem~\ref{thm: CUR Char}, which implies that $\mathcal{A} = \mathcal{R}\times_1 (C_1U_1^\dagger) \times_2\dots\times_n (C_nU_n^\dagger) $ as required.
\end{proof}

Recall that according to Theorem~\ref{cor: CUR Char1}, if $J_i=\otimes_{j\neq i}I_i$, and  $\rank(\mathcal{A}_i(I_i,:))=r_i$ for all $i$, then $\mathcal{A}=\mathcal{R}\times_{i=1}^n(C_iU_i^\dagger) $. Theorem~\ref{THM:ExactCURChidori} shows a sampling scheme only requiring choosing indices $I_i$ to achieve an exact Chidori CUR decomposition with high probability.

\begin{proof}[Proof of Theorem~\ref{THM:ExactCURChidori}]
This theorem is derived in the same way as Theorem~\ref{THM:UniformExactCUR} but applying Theorem~\ref{cor: CUR Char1} instead of Theorem~\ref{thm: CUR Char}. 
\end{proof}

\subsection{Ingredients for the Proof of Theorem~\ref{thm: per_uniform_sampling}}

We now turn toward the proof of the main result (Theorem~\ref{thm: per_uniform_sampling}) regarding approximation bounds of tensor-CUR approximations via randomized sampling. Before proving Theorem~\ref{thm: per_uniform_sampling}, we state an fundamental result of Tropp about uniform random sampling of matrices with low incoherence .

\begin{lemma}[{\cite[Lemma 3.4]{tropp2011improved}}]\label{lmm: uniform-sample}
Suppose that ${W}\in\R^{d\times r}$ has orthonormal columns.  If $I\subseteq[d]$ with $|{I}|\geq\gamma \mu(W) r$ for some $\gamma>0$ is chosen by sampling $[d]$ uniformly without replacement, then
\begin{equation*}
\begin{aligned}
\Prob\left( \left\|W(I,:)^\dagger\right\|_2\leq \sqrt{\frac{d}{(1-\delta)|I|}}\right) &\geq 1-r\left(\frac{e^{-\delta}}{(1-\delta)^{1-\delta}}\right)^\gamma,\quad \textnormal{for all} \quad \delta\in[0,1),\\
\Prob\left( \left\|W(I,:)\right\|_2\leq \sqrt{\frac{(1+\eta)|I|}{d}}\right) &\geq 1-r\left(\frac{e^{\eta}}{(1+\eta)^{1+\eta}}\right)^\gamma,\quad \textnormal{for all} \quad \eta\geq0.
\end{aligned}
\end{equation*}
\end{lemma}

Lemma~\ref{lmm: uniform-sample} easily results in the following corollary for uniformly sampling indices of a tensor.
\begin{corollary}\label{cor: uniform-sample}
Let $\mathcal{A}\in\mathbb{R}^{d_1\times\cdots\times d_n}$ with multilinear rank $\left(r_1,\cdots,r_n\right)$ and let $\mathcal{A}_{(i)}=W_i\Sigma_i V_i^\top$ be $\mathcal{A}_{(i)}$'s compact SVD. If $I_i\subseteq[d_i]$ with $|I_i|\geq \gamma_{i}\mu(W_i)r_i$ for some $\gamma_{i}>0$ are chosen by sampling $[d_i]$ uniformly without replacement, then with probability at least $1-\sum_{i=1}^n \left(r_i\left(\frac{e^{-\delta}}{(1-\delta)^{1-\delta}}\right)^{\gamma_{i}}+r_i\left(\frac{e^{\eta}}{(1+\eta)^{1+\eta}}\right)^{\gamma_{i}}\right)$
\begin{equation}
\begin{aligned}
\|W_i(I_i,:)^\dagger\|_2\leq \sqrt{\frac{d_i}{(1-\delta)|I_i|}}, \qquad 
\left\|W_i(I_i,:)\right\|_2\leq \sqrt{\frac{(1+\eta)|I_i|}{d_i}}
   \end{aligned}
\end{equation}
for all $i=1,\cdots,n$ and all $\delta\in[0,1)$ and $\eta\geq0$.
\end{corollary}

\begin{proof}[Proof of Theorem~\ref{thm: per_uniform_sampling}]
The proof follows essentially by combining Theorem~\ref{thm: pt_ana_tensor-cur_tube} and Corollary~\ref{cor: uniform-sample}. Suppose that the HOSVD of $\mathcal{A}$ is $\mathcal{A}=\mathcal{T}\times_1 W_1 \times_2\cdots\times_n W_n $. Then
\begin{equation*}
    \begin{aligned}
    \|\mathcal{R}_{(i)}\|_2&=\|W_{i,I_i}\mathcal{T}_{(i)}\left(\kron_{j\neq i}W_{j,I_j} \right)^\top\|_2\\
    &\leq\|\mathcal{T}_{(i)}\|_2\prod_{j=1}^n\|W_{j,I_j}\|_2\\
    &=\sigma_{i1}\prod_{j=1}^n\|W_{j,I_j}\|_2\leq \sigma_{\max}(\mathcal{A})\prod_{j=1}^n\|W_{j,I_j}\|_2.
    \end{aligned}
\end{equation*}
According to Corollary~\ref{cor: uniform-sample}, with probability at least $$1-\sum_{i=1}^n \left(r_i\left(\frac{e^{-\delta}}{(1-\delta)^{1-\delta}}\right)^{\gamma_{i}}+r_i\left(\frac{e^{\eta}}{(1+\eta)^{1+\eta}}\right)^{\gamma_{i}}\right)$$ we have
\[
\begin{aligned}
\|W_i(I_i,:)^\dagger\|_2\leq \sqrt{\frac{d_i}{(1-\delta)|I_i|}},\qquad\left\|W_i(I_i,:)\right\|_2\leq \sqrt{\frac{(1+\eta)|I_i|}{d_i}}
   \end{aligned}.
\]
Therefore,
 \[
     \|\mathcal{R}_{(i)}\|_2\leq \sqrt{\frac{(1+\eta)^n\prod_{j=1}^n|I_i|}{\prod_{j=1}^nd_j}}\sigma_{\max}(\mathcal{A})
 \]
 and 
   \begin{align*}
     \left\|U_i^\dagger\right\|_2=\left\|\mathcal{R}_{(i)}^\dagger\right\|&=\left\| \left(W_{i,I_i}\mathcal{T}_{(i)}\left(\kron_{j\neq i}W_{j,I_j} \right)^\top\right)^\dagger \right\|_2\\
     &\leq \|\mathcal{T}_i^\dagger\|_2\prod_{j=1}^n\|W_{j,I_j}^\dagger\|_2\\
     &\leq \sqrt{\frac{\prod_{j=1}^nd_j}{(1-\delta)^n\prod_{j=1}^n|I_j|}}\left\|\mathcal{A}_{(i)}^\dagger\right\|_2 \\
     &=\sqrt{\frac{\prod_{j=1}^nd_j}{(1-\delta)^n\prod_{j=1}^n|I_j|}}\frac{1}{\sigma_{ir_i}} \\
     &\leq \frac{1}{\sigma_{\min}(\mathcal{A})}\sqrt{\frac{\prod_{j=1}^nd_j}{(1-\delta)^n\prod_{j=1}^n|I_j|}}.
     \end{align*}
 Additionally, $$\sqrt{\frac{(1-\delta)^{n}\prod_{i=1}^n|I_i|}{\prod_{i=1}^nd_i}}\sigma_{\min}(\mathcal{A})\geq 8\|E_{I_i,J_i}\|_2,$$ thus
  $\sigma_{r_i}(U_i)\geq 8\|E_{I_i,J_i}\|$ holds given the condition that $\|W_i(I_i,:)^\dagger\|_2\leq \sqrt{\frac{d_i}{(1-\delta)|I_i|}}$.  Combing this with  Theorem~\ref{thm: pt_ana_tensor-cur_tube}, with probability at least $$1-\sum_{i=1}^n \left(r_i\left(\frac{e^{-\delta}}{(1-\delta)^{1-\delta}}\right)^{\gamma_{i}}+r_i\left(\frac{e^{\eta}}{(1+\eta)^{1+\eta}}\right)^{\gamma_{i}}\right),$$ we have

    \begin{align*}
     &~\left\|\mathcal{A}-\mathcal{A}_{\textnormal{app}}\right\|_F\\
     \leq&~\frac{9^n}{4^n}\left(\prod_{i=1}^n\left\|W_{i,I_i}^\dagger\right\|_2\right)\|\mathcal{E}_{\mathcal{R}}\|_F \\
     &~+\sum_{j=1}^n\frac{9^{n-j}}{4^{n-j}}\left\|\mathcal{R}_{(j)}\right\|_2\left(\prod_{i\neq j}\left\|W_{i,I_i}^\dagger \right\|_2^2\right)\left\|W_{j,I_j}^\dagger\right\|_2\left\|\mathcal{A}_{(j)}^\dagger\right\|_2\left(5\left\|W_{j,I_j}^\dagger\right\|_2\left\|E_{I_j,J_j}\right\|_F+2\left\|E_{J_j}\right\|_F\right)\\
     \leq&~\frac{9^n}{4^n}\sqrt{\frac{\prod_{j=1}^nd_j}{(1-\delta)^n\prod_{j=1}^n|I_j|}}\|\mathcal{E}_{\mathcal{R}}\|_F \\
     &~+\sum_{j=1}^n\frac{\sigma_{\max}(\mathcal{A})9^{n-j}}{4^{n-j}} \sqrt{\frac{(1+\eta)^n\prod_{j=1}^n|I_i|}{\prod_{j=1}^nd_j}}\frac{\prod_{i\neq j}d_i}{(1-\delta)^{n-1}\prod_{i\neq j}|I_i|} \sqrt{\frac{d_j}{(1-\delta)|I_j| }} \left\|\mathcal{A}_{(j)}^\dagger\right\|_2\\\
     &~ \cdot\left(5\sqrt{\frac{d_j}{(1-\delta)|I_j| }}\left\|E_{I_j,J_j}\right\|_F+2\left\|E_{J_j}\right\|_F\right) \Bigg)\\
     \leq&~\frac{9^n}{4^n}\sqrt{\frac{\prod_{j=1}^nd_j}{(1-\delta)^n\prod_{j=1}^n|I_j|}}\|\mathcal{E}_{\mathcal{R}}\|_F \\
     &~+\frac{\sigma_{\max}(\mathcal{A})}{\sigma_{\min}(\mathcal{A})}\sqrt{\frac{(1+\eta)^n}{(1-\delta)^n}}\sum_{j=1}^n\frac{9^{n-j}}{4^{n-j}} \sqrt{\frac{\prod_{i\neq j}d_i}{(1-\delta)^{n-1}\prod_{i\neq j}|I_i|}}\left(5\sqrt{\frac{d_j}{(1-\delta)|I_j|}}\left\|E_{I_j,J_j}\right\|_F+2\left\|E_{J_j}\right\|_F\right)\\
     \leq&~ \frac{9^n\sqrt{\prod_{i=1}^nd_i} }{4^n(1-\delta)^{\frac n 2}\sqrt{\prod_{i=1}^n|I_i|}}\|\mathcal{E}_{\mathcal{R}}\|_F+\frac{9^{n}}{4^{n-1}}\frac{\sigma_{\max}(\mathcal{A})} {\sigma_{\min}(\mathcal{A})}\left(\frac{1+\eta}{1-\delta}\right)^{\frac{n}{2}}\sqrt{\prod_{i=1}^{n}\frac{d_i}{(1-\delta)|I_i|}}\|\mathcal{E}_{\mathcal{R}}\|_F \\
   &~  + \frac{2\sigma_{\max}(\mathcal{A})} {\sigma_{\min}(\mathcal{A})}\left(\frac{1+\eta}{1-\delta}\right)^{\frac{n}{2}} \sum_{j=1}^n\frac{9^{n-j}}{4^{n-j}}\sqrt{\prod_{i\neq j}\frac{d_i}{(1-\delta)|I_i|}}
       \left\|E_{J_j}\right\|_F.
    \end{align*}
The last equation holds because of $\left\|E_{I_i,J_i}\right\|_F=\|\mathcal{E}_{\mathcal{R}}\|_F $ for $i=1,\cdots,n$ and $\sum_{j=1}^{n}\frac{9^{n-j}}{4^{n-j}}\leq \frac{4}{5}\cdot\frac{9^{n}}{4^{n}}$. 
This completes the proof.
\end{proof}

\begin{proof}[Proof of Theorem~\ref{cor:per_uniform_sampling}]
To prove this corollary, we need to use \cite[Theorem 1.5]{tropp2012user}. For the reader's convenience, we state the theorem here:
\begin{theorem}[{\cite[Theorem 1.5]{tropp2012user}}]\label{thm:tropp}
Consider a finite sequence $\{B_k\}$ of fixed matrices with dimension $m_1\times m_2$, and let $\{\xi_{k}\}$ be a finite sequence of independent standard normal variables. Define the variance parameter 
\[\phi^2:=\max\left\{\left\|\sum_{k}B_kB_k^\top \right\|_2,\left\|\sum_{k}B_k^\top B_k \right\|_2 \right\}.
\]
Then, for all $t\geq 0$,
\[ \left\|\sum_{k}\xi_kB_k \right\|_2\leq t
\]
with probability at least $1-(m_1+m_2)e^{-\frac{t^2}{2\phi^2}}$.
\end{theorem}
Now, let's prove Theorem~\ref{cor:per_uniform_sampling} with the case $d\leq \ell^{n-1}$. 
Since $d_1=\cdots=d_n=d$, $r_1=\cdots=r_n$, $|I_1|=\cdots=|I_n|=\ell$,  $\delta=\eta=\frac{1}{2}$, and $\gamma=10r\log(d)$, \eqref{eqn:per_uniform} of Theorem~\ref{thm: per_uniform_sampling} can be restated as \begin{equation*}
\begin{aligned}
   &\quad~ \left\|\mathcal{A}-\mathcal{A}_{\textnormal{app}}\right\|_F\\
   &\leq \frac{9^n\sqrt{\prod_{i=1}^nd_i} }{4^n(1-\delta)^{\frac n 2}\sqrt{\prod_{i=1}^n|I_i|}}\|\mathcal{E}_{\mathcal{R}}\|_F+\frac{9^{n}}{4^{n-1}}\frac{\sigma_{\max}(\mathcal{A})} {\sigma_{\min}(\mathcal{A})}\left(\frac{1+\eta}{1-\delta}\right)^{\frac{n}{2}}\sqrt{\prod_{i=1}^{n}\frac{d_i}{(1-\delta)|I_i|}}\|\mathcal{E}_{\mathcal{R}}\|_F \\
   &\quad~   +\frac{2\sigma_{\max}(\mathcal{A})} {\sigma_{\min}(\mathcal{A})}\left(\frac{1+\eta}{1-\delta}\right)^{\frac{n}{2}} \sum_{j=1}^n\frac{9^{n-j}}{4^{n-j}}\sqrt{\prod_{i\neq j}\frac{d_i}{(1-\delta)|I_i|}} 
       \left\|E_{J_j}\right\|_F\\
      &= \left(\frac{9\sqrt{2d}}{4\sqrt{\ell}}\right)^{n}\|\mathcal{E}_{\mathcal{R}}\|_F+
      \frac{9^{n}}{4^{n-1}}\cdot\frac{\sigma_{\max}(\mathcal{A})}{\sigma_{\min}(\mathcal{A})}\cdot 3^{\frac{n}{2}}\cdot \sqrt{\frac{(2d)^n}{\ell^n}}\|\mathcal{E}_{\mathcal{R}}\|_F \\
      &\quad~
       + \frac{2\sigma_{\max}(\mathcal{A})} {\sigma_{\min}(\mathcal{A})}\cdot 3^{\frac{n}{2}}\cdot \sqrt{
      \frac{(2d)^{n-1}}{\ell^{n-1}}} \cdot
       \sum_{j=1}^{n}\frac{9^{n-j}}{4^{n-j}}\|E_{J_j}\|_F,
      \end{aligned}
\end{equation*}
with probability at least 
\begin{equation*}
\begin{aligned}
    &~1-\sum_{i=1}^n \left(r_i\left(\frac{e^{-\delta}}{(1-\delta)^{1-\delta}}\right)^{\gamma_{i}}+r_i\left(\frac{e^{\eta}}{(1+\eta)^{1+\eta}}\right)^{\gamma_{i}}\right)\\
    =&~1-nr\left(\left(\frac{e^{-\frac{1}{2}}}{(\frac12)^{\frac12}}\right)^{10\log(d)}+\left(\frac{e^{\frac12}}{(\frac32)^{\frac{3}{2}}}\right)^{10\log(d)}\right)\\
    \geq&~ 1-nr\left(\frac{1}{d^{3/2}}+\frac{1}{d^{1.08}}\right)\geq 1-\frac{2nr}{d}.
    \end{aligned}
\end{equation*}
  Additionally, $\mathcal{E}_{i_1,\cdots,i_n}$ are i.i.d.~and $\mathcal{E}_{i_1,\cdots,i_n}\sim\mathcal{N}(0,\sigma)$. Thus,  $E_{I_i,J_i}$ can be rewritten in the form $\sum_{k,s}\xi_{k,s} B_{k,s}$ where $B_{k,s}\in\mathbb{R}^{|I_i|\times|J_i|}$ is a matrix where the only nonzero entry is $\sigma$ with index ($k,s$) and $\xi_{k,s}\sim\mathcal{N}(0,1)$. Then $\phi^2=\max\left\{ \|\sum_{k,s}B_{k,s}^\top B_{k,s}\|_2,\|\sum_{k,s}B_{ks,}B_{k,s}^\top\|_2\right\}=\sigma^2\ell^{n-1}$. 
  By applying Theorem~\ref{thm:tropp}, we have with probability at least $1-(\ell^{n-1}+d)^{1-2\ell^{1-p}}$ 
  \[\|E_{I_i,J_i}\|_2\leq 2\sigma\sqrt{\ell^{n-p}\log(\ell^{n-1}+d)} \quad\text{~and~}\quad\|E_{I_i,J_i}\|_F\leq 2\sigma\sqrt{\ell^{n+1-p}\log(\ell^{n-1}+d)},\]
  where   $1<p<1+\frac{\log(2)}{\log(\ell)}$.
   Similarly, we have  with probability at least $1-(\ell^{n-1}+d)^{1-2\ell^{1-p}}$
  \[\|E_{J_i}\|_2\leq 2\sigma\sqrt{ \ell^{n-p} \log(\ell^{n-1}+d)} \quad\text{~and~}\quad\|E_{J_i}\|_F\leq 2\sigma\sqrt{ d\ell^{n-p}\log(\ell^{n-1}+d)}.\] 
 By using the union bound of all these probabilities, we find that \begin{equation*}
\begin{aligned}
~& \left\|\mathcal{A}-\mathcal{A}_{\textnormal{app}}\right\|_F\\
\leq& \left(\frac{9\sqrt{2d}}{4\sqrt{\ell}}\right)^{n}\cdot 2\sigma\sqrt{\ell^{n+1-p}\log(\ell^{n-1}+d)} \\
      ~& +
      \frac{9^{n}}{4^{n-1}}\cdot\frac{\sigma_{\max}(\mathcal{A})}{\sigma_{\min}(\mathcal{A})}\cdot 3^{\frac{n}{2}}\cdot \sqrt{\frac{(2d)^n}{\ell^n}}\cdot 2\sigma\sqrt{\ell^{n+1-p}\log(\ell^{n-1}+d)} \\
      ~& +\frac{2\sigma_{\max}(\mathcal{A})} {\sigma_{\min}(\mathcal{A})}\cdot 3^{\frac{n}{2}}
       \sum_{j=1}^{n}\frac{9^{n-j}}{4^{n-j}}\sqrt{
      \frac{(2d)^{n-1}}{\ell^{n-1}}} \cdot 2\sigma\sqrt{d\ell^{n-p}\log(\ell^{n-1}+d)}\\
      \leq&\left(\frac{3^{2n}}{2^{\frac{3}{2}n-1}}+
      \frac{\sigma_{\max}(\mathcal{A})} {\sigma_{\min}(\mathcal{A})}\cdot \frac{ 3^{\frac{5n}{2}}}{2^{\frac{3n}{2}-4} } \right)
     \ell^{\frac{1-p}{2}} \sqrt{\log(\ell^{n-1}+d)}d^{\frac{n}{2}}\sigma
\end{aligned}
\end{equation*}
 with probability at least $1-\frac{2rn}{d}-\frac{2n}{(\ell^{n-1}+d)^{2\ell^{1-p}-1}}$.
 
 The proof of the case when $d>\ell^{n-1}$ follows from a slight modification of the same argument and the fact that $\phi^2=\sigma^2d$.
 \end{proof}

\section{Numerical Experiments}\label{SEC:Experiments}

In this section, we evaluate the empirical performance of the proposed tensor CUR decompositions against other state-of-the-art low multilinear rank tensor approximation methods: HOSVD from \cite{de2000best,tucker1966}, sequentially truncated HOSVD (st-HOSVD) from \cite{VVM2012}, and higher-order orthogonal iteration (HOOI) from \cite{de2000best,KD1980}. 
All tests are conducted on a Ubuntu workstation equipped with Intel i9-9940X CPU and 128GB DDR4 RAM, and executed from Matlab R2020a. We use the implementations of HOSVD, st-HOSVD and HOOI from \texttt{tensor\_toolbox v3.1}\footnote{Website: \url{https://gitlab.com/tensors/tensor_toolbox/-/releases/v3.1}.}.  
\RV{The codes for Fiber and Chidori CUR decompositions are available online at 
    \url{https://github.com/caesarcai/Modewise_Tensor_Decomp}.}

To implement the tensor CUR decompositions, we set $C_i$ to have size $d_i\times 2r_i\log(\prod_{j\neq i}d_j)$ and $U_i$ to be $r_i\log(d_i)\times 2r_i\log(\prod_{j\neq i}d_j)$ for the Fiber CUR decomposition, and set  $C_i$ to have size $d_i\times \prod_{j\neq i}r_j\log(d_j)$ and $U_i$ to be $r_i\log(d_i)\times \prod_{j\neq i}r_j\log(d_j)$ for the Chidori CUR decomposition.  The indices used to determine $\mathcal{R}$, $C_i$, and $U_i$ are sampled uniformly at random, and the sampling size comes from the sampling guarantee results of Section~\ref{SEC:Sampling}.

\subsection{Synthetic Dataset}
We generate a $3$-mode tensor $\mathcal{X}\in\mathbb{R}^{d\times d \times d}$ with rank $(r,r,r)$ by 
\begin{equation*}
\mathcal{X}:=\mathcal{T}\times_1 G_1\times_2 G_2\times_3 G_3,
\end{equation*}
where the core tensor $\mathcal{T}\in\mathbb{R}^{r\times r \times r}$ and matrices $G_1,G_2,G_3 \in\mathbb{R}^{d\times r}$ are random tensor/matrices with i.i.d.~Gaussian distributed entries ($\sim \mathcal{N}(0,1)$). In addition, we generate the additive i.i.d.~Gaussian noise $\mathcal{E}\in\mathbb{R}^{d\times d \times d}$ with some given variance $\sigma$ so that we have the noisy observation
\begin{equation*}
    \widetilde{\mathcal{X}}:=\mathcal{X}+\mathcal{E}.
\end{equation*}

We evaluate the robustness and computational efficiency of all tested methods under four different noise levels (i.e., $\sigma=10^{-1},10^{-4},10^{-7}$, and $0$). For all methods, we denote the output of multilinear rank $(r,r,r)$ tensor approximation to be $\mathcal{X}_r$, and the relative approximation error is
\begin{equation*}
    \textrm{err}:=\frac{\|\mathcal{X}_r - \mathcal{X}\|_F}{\|\mathcal{X}\|_F}.
\end{equation*}
The test results are averaged over 50 trials and summarized in Figure~\ref{fig:curs-vs-HOSVD}. One can see that both variations of the proposed tensor CUR decomposition are substantially faster than all other state-of-the-art methods. In particularly, Fiber CUR achieves over $150\times$ speedup when $d$ is large. In the noiseless case (i.e., $\sigma=0$), all methods, including the proposed ones, approximate the low multilinear rank tensor with the same accuracy. However, when additive noise appears, the proposed methods have slightly worse but still good approximation accuracy. We provide more detailed and enlarged runtime plots for only the tensor CUR decompositions in Figure~\ref{fig:size-runtime-curs}. As discussed in Section~\ref{subsec:complxity section}, we verify that both tensor CUR decompositions have much lower computational complexities than the start-of-the-art, with Chidori CUR being slightly slower than Fiber CUR.

 \begin{figure}[ht]
\vspace{-0.08in}
\centering
\subfloat[$\sigma=10^{-1}$]{\includegraphics[width=.25\linewidth]{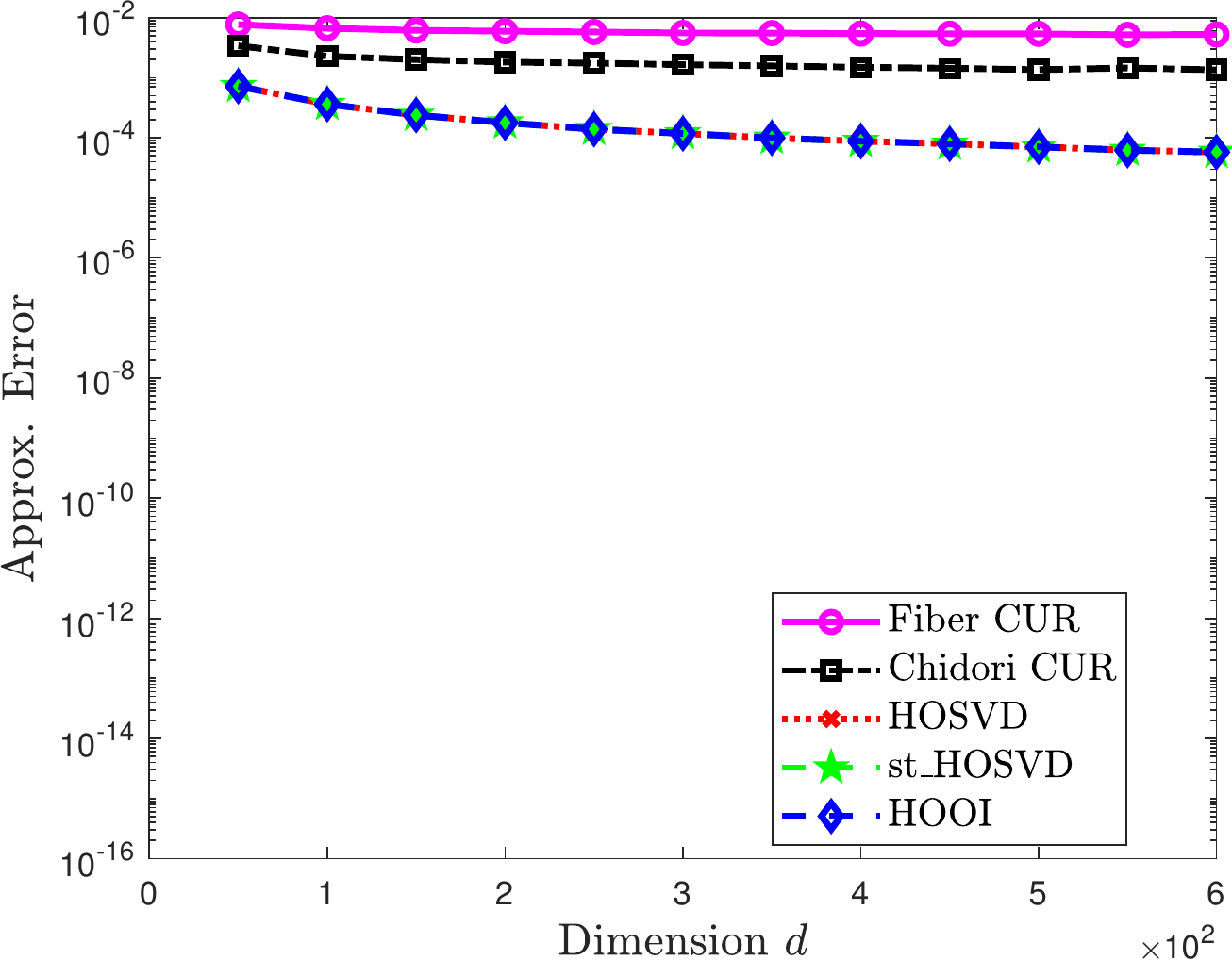}} \hfill
\subfloat[$\sigma=10^{-4}$]{\includegraphics[width=.25\linewidth]{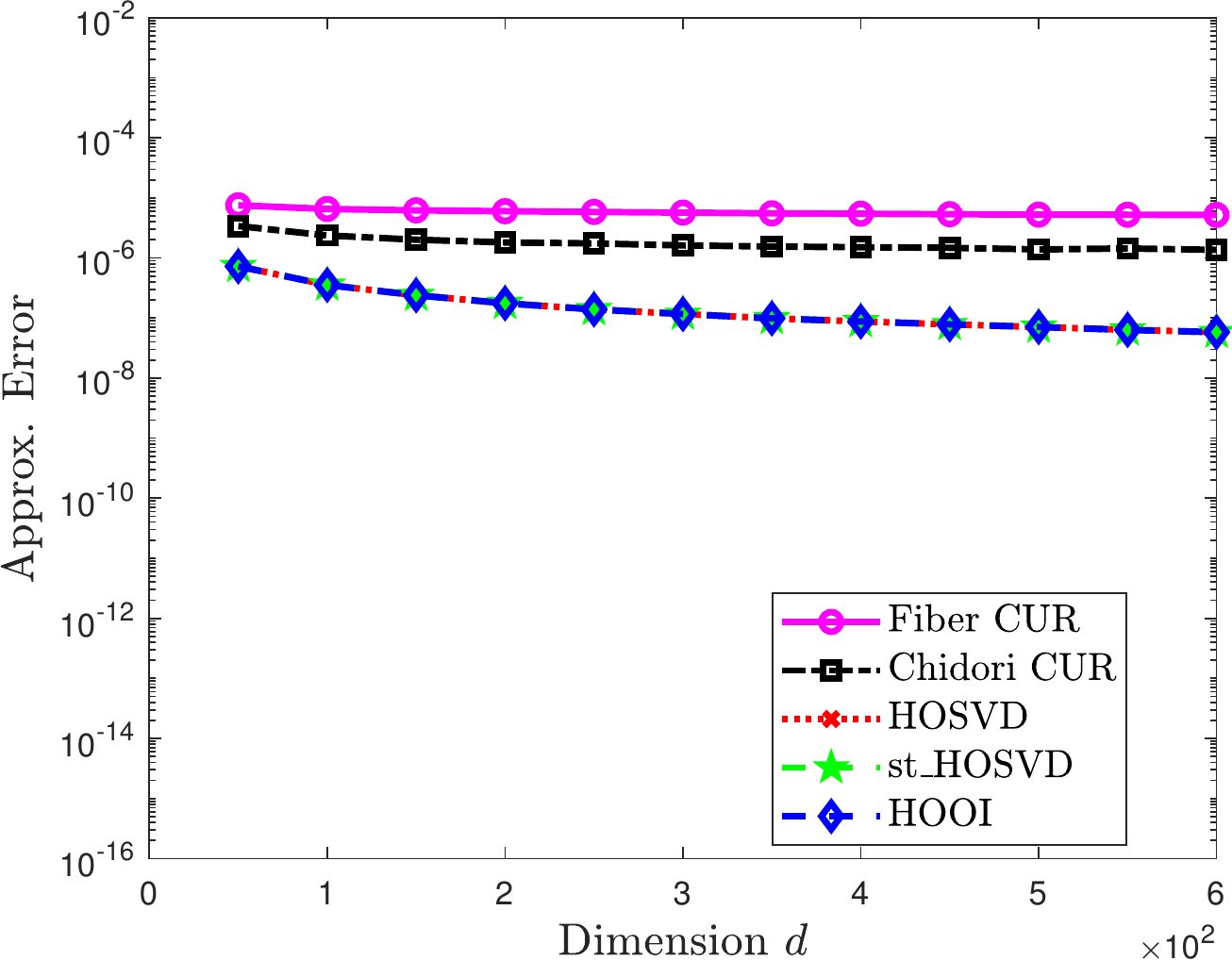}}\hfill
\subfloat[$\sigma=10^{-7}$]{\includegraphics[width=.25\linewidth]{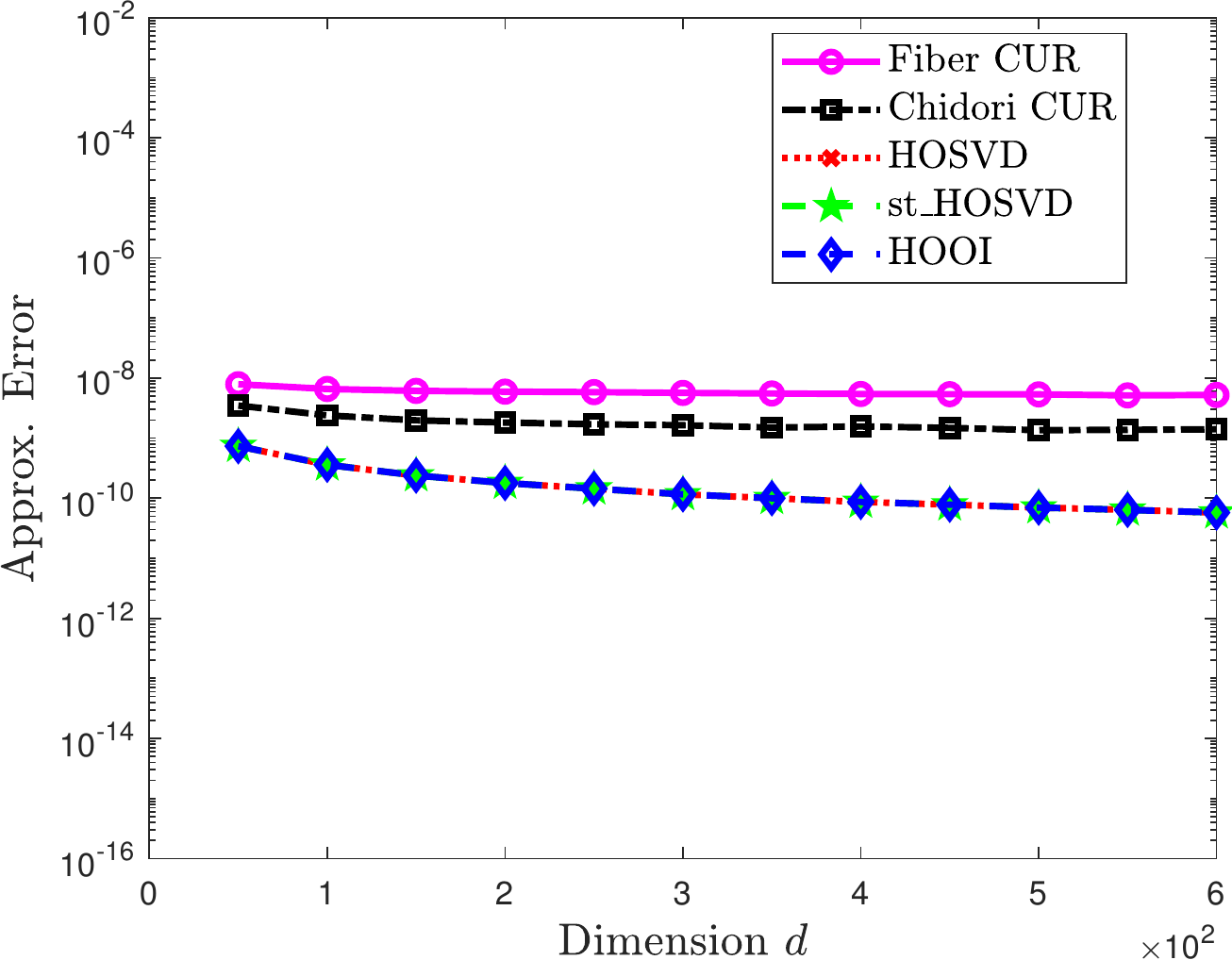}}\hfill
\subfloat[$\sigma=0$]{\includegraphics[width=.25\linewidth]{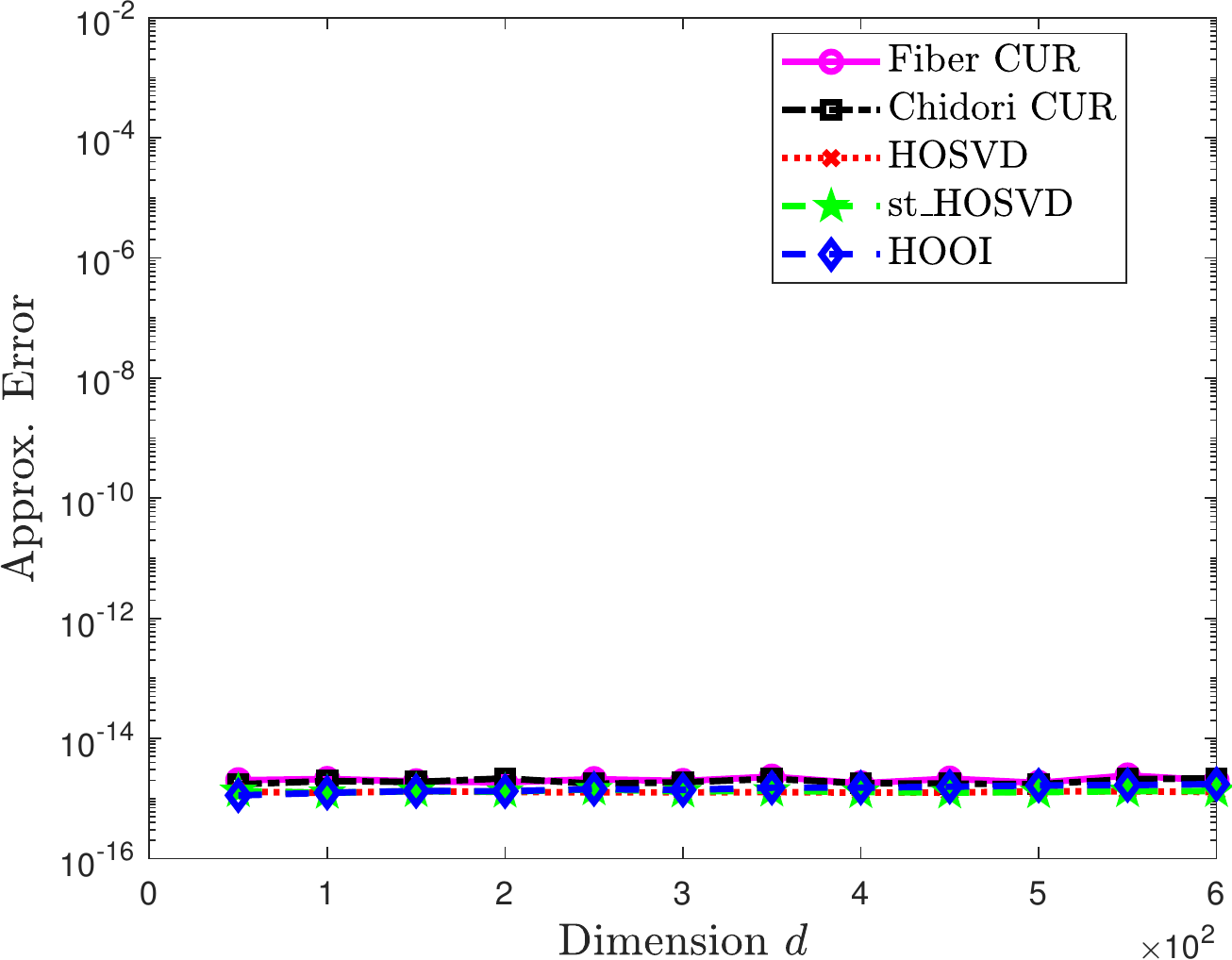}}\hfill
\subfloat{\includegraphics[width=.25\linewidth]{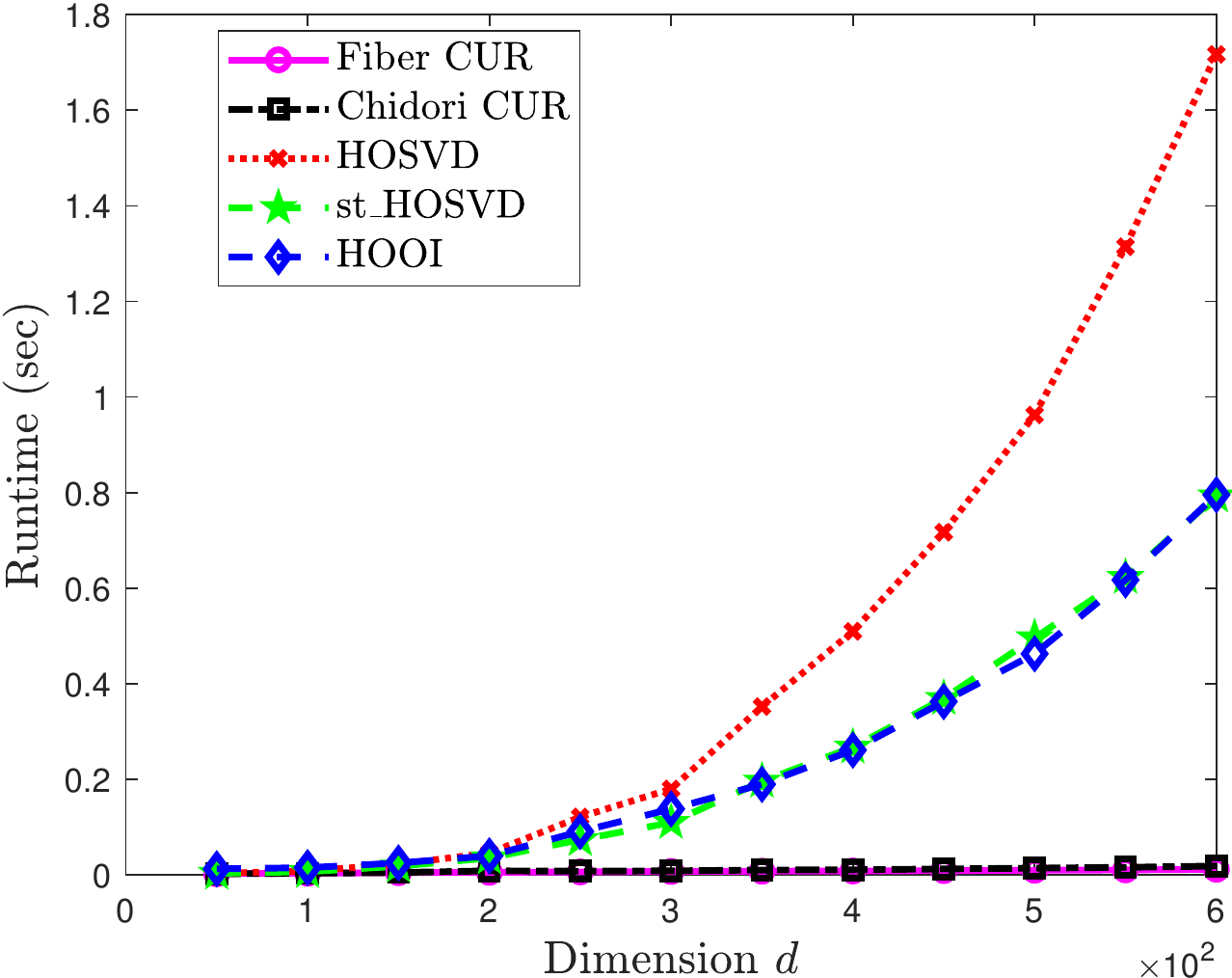}} \hfill
\subfloat{\includegraphics[width=.25\linewidth]{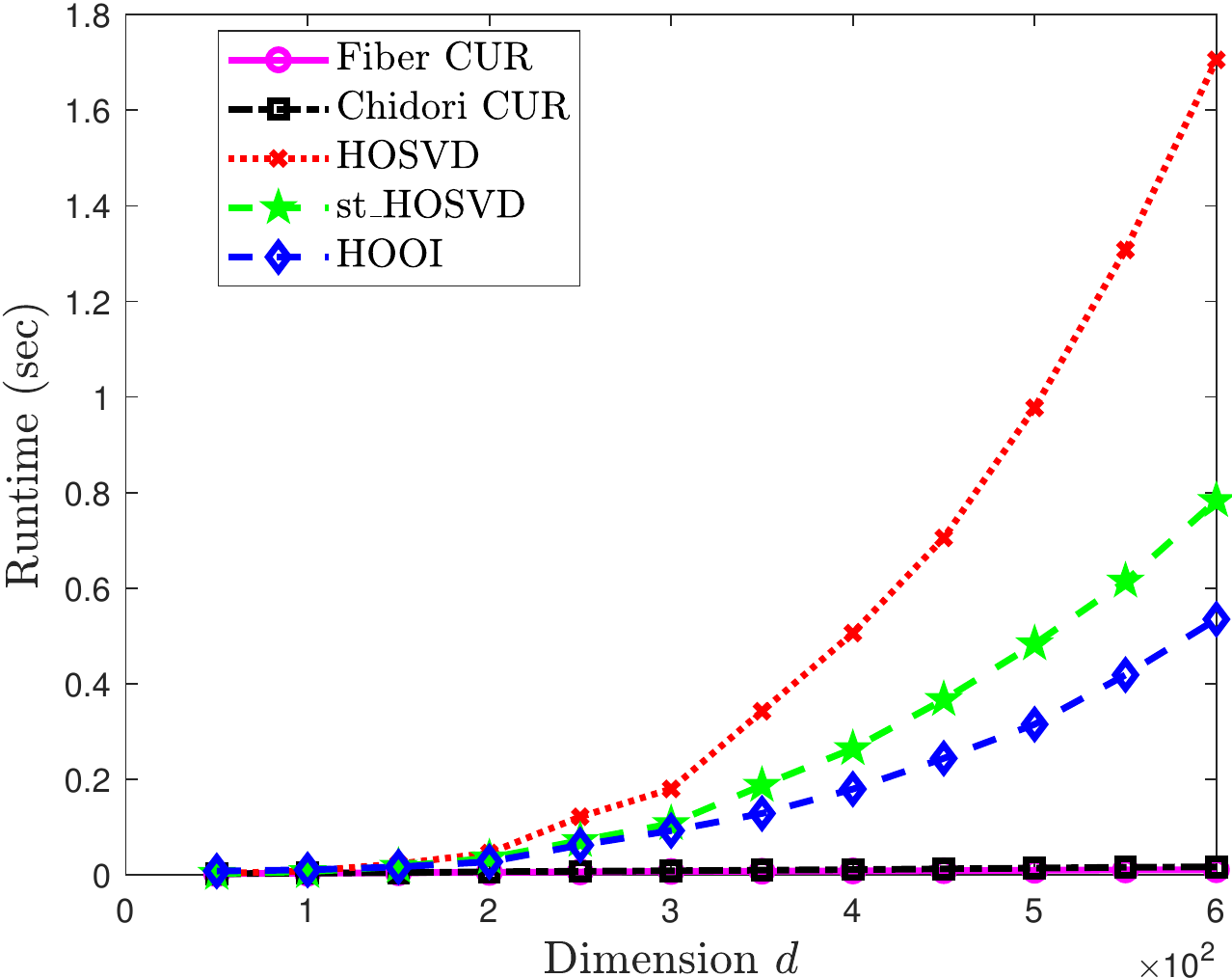}}\hfill
\subfloat{\includegraphics[width=.25\linewidth]{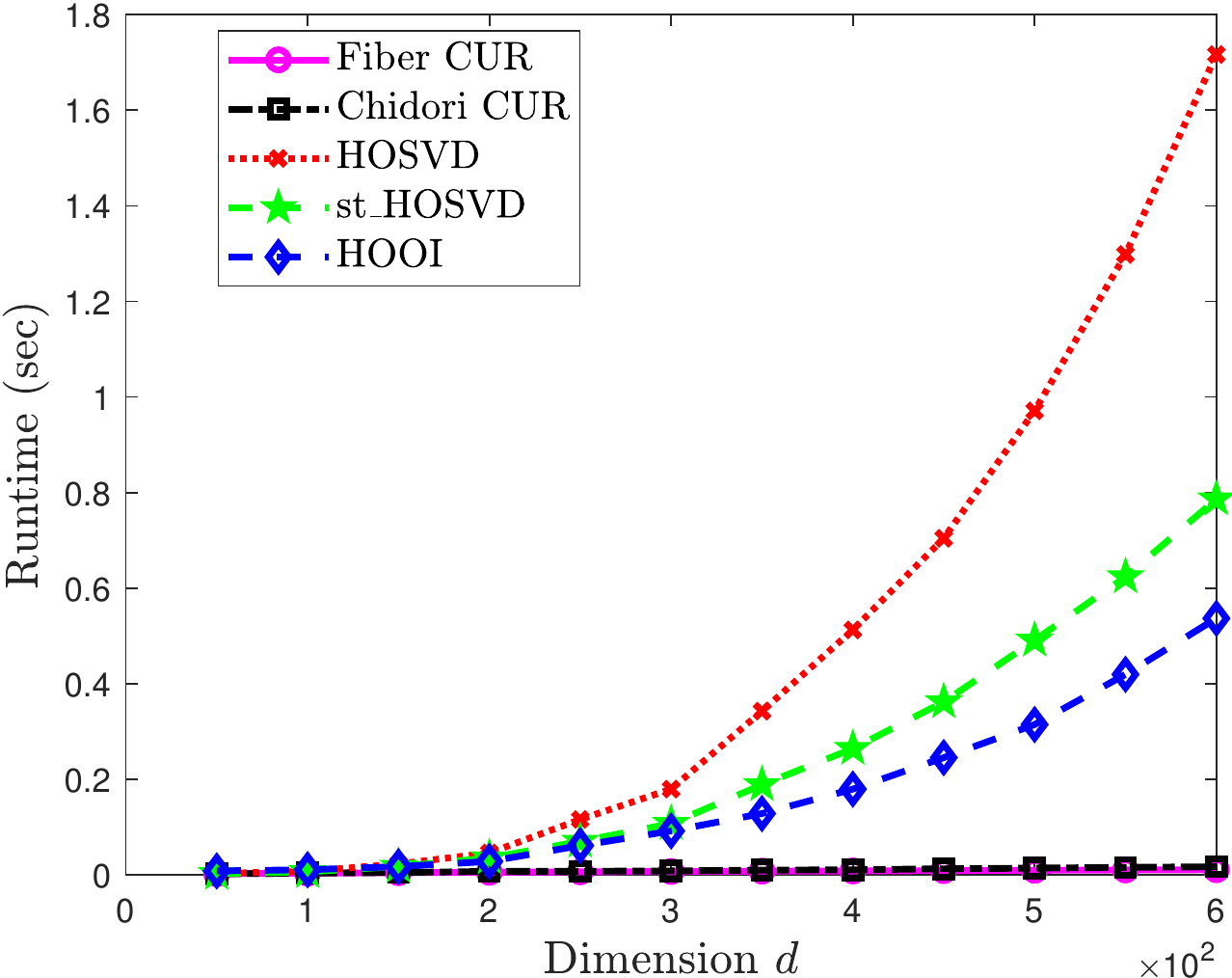}}\hfill
\subfloat{\includegraphics[width=.25\linewidth]{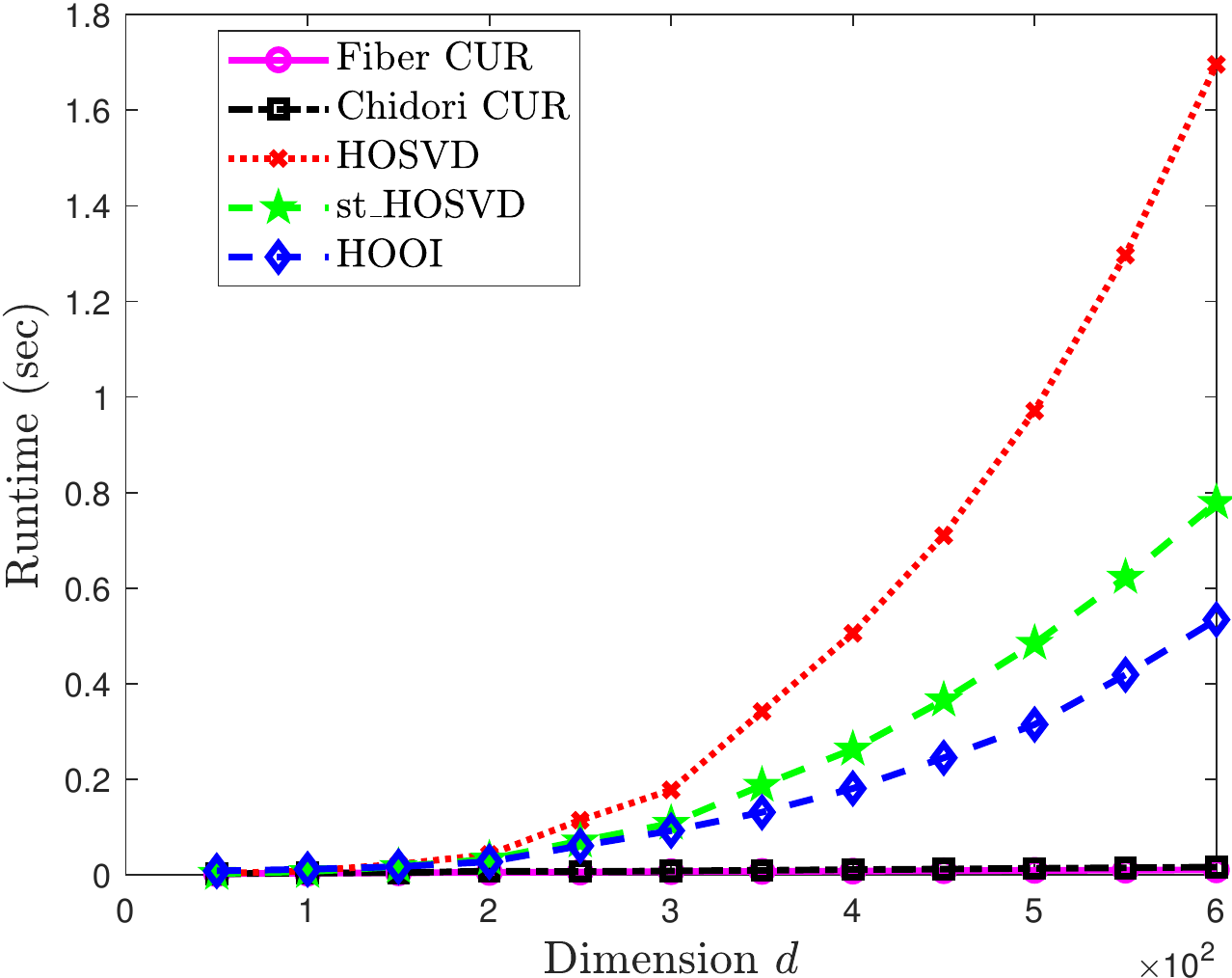}}\hfill
\caption{Comparison of low rank tensor approximation methods under different noise levels $\sigma$. Rank $(5,5,5)$ is used in all tests and $d$ varies from $50$ to $600$. 
\textbf{Top row:} relative approximation errors \textit{vs.} tensor dimensions. \textbf{Bottom row:} runtime \textit{vs.} tensor dimensions.
} \label{fig:curs-vs-HOSVD}
\vspace{-0.05in}
\end{figure}

\begin{figure}[ht]
\vspace{-0.08in}
\centering
\subfloat[Fiber CUR]{\includegraphics[width=.49\linewidth]{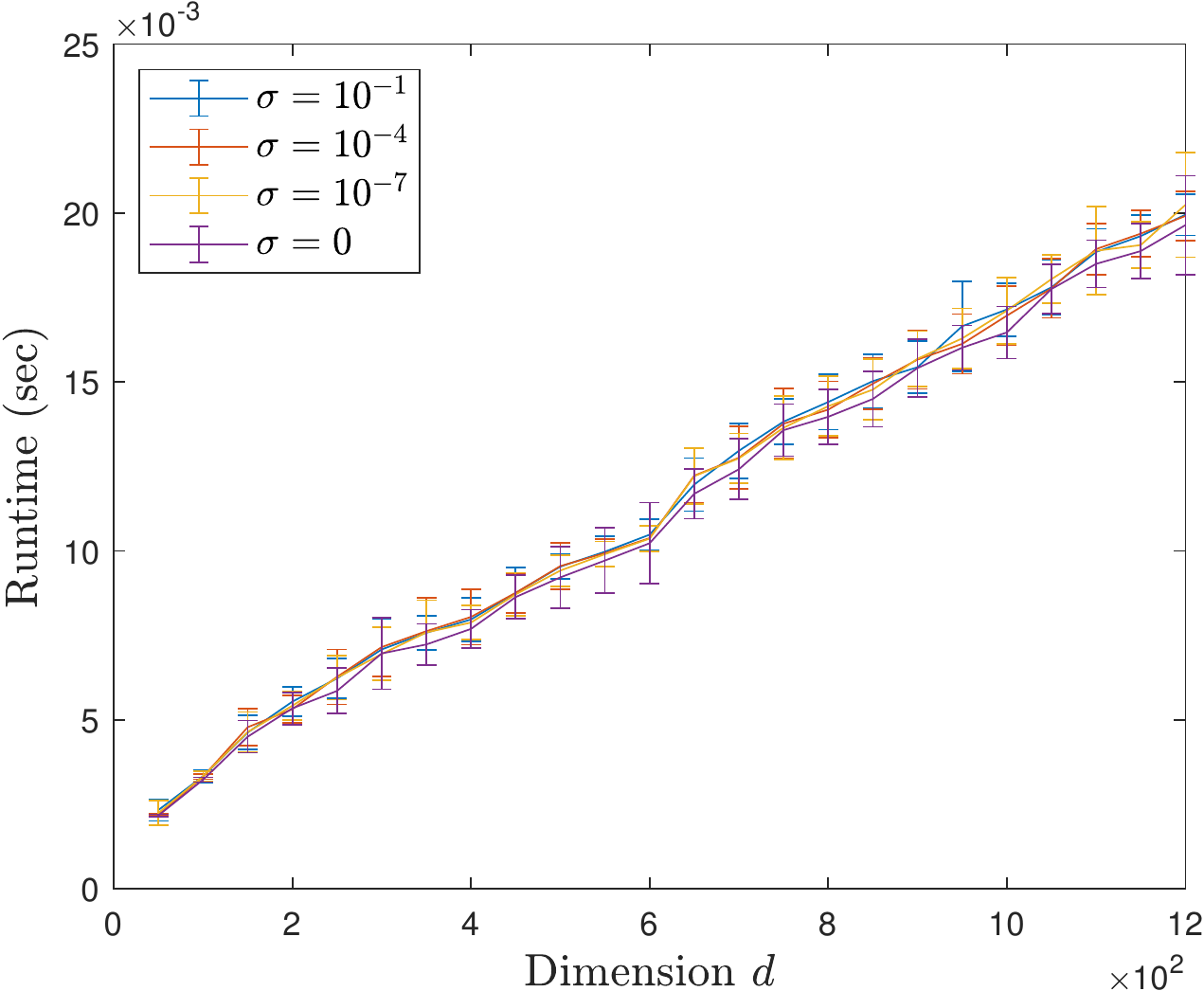}} \hfill
\subfloat[Chidori CUR]{\includegraphics[width=.49\linewidth]{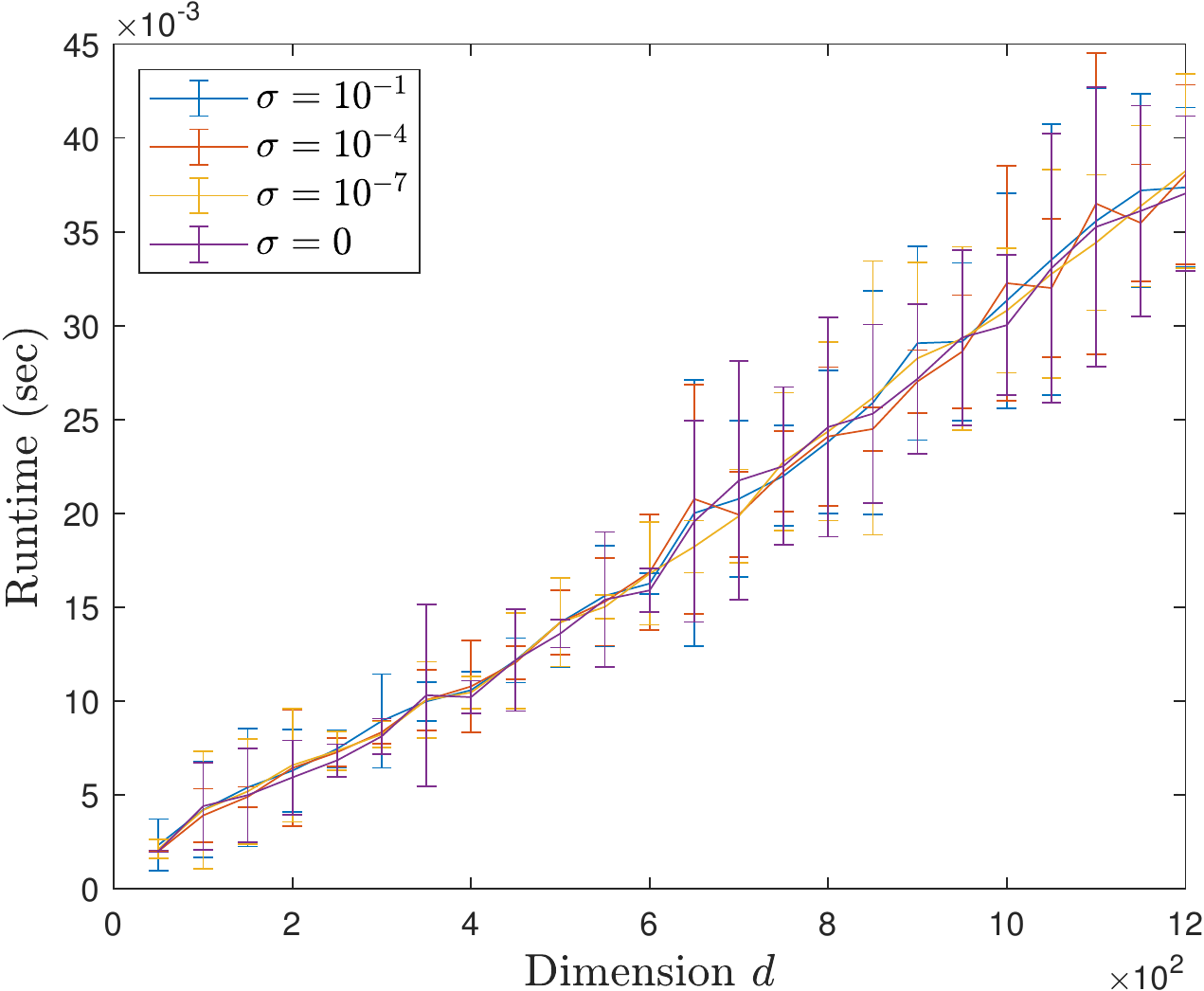}}
\caption{Runtime with variance bar \textit{vs.} tensor dimensions for Fiber and Chidori CUR under different noise levels $\sigma$. Rank $(5,5,5)$ is used for all tests and $d$ varies from $50$ to $1200$. 
}\label{fig:size-runtime-curs}
\vspace{-0.05in}
\end{figure}

\begin{remark}
In Figure~\ref{fig:size-runtime-curs}, the time complexities, with respect to the problem dimension $d$, of both tensor CUR decompositions appear to be $\mathcal{O}(d)$ instead of the claimed $\mathcal{O}(n\log^2(d))$ and $\mathcal{O}(\log^n(d))$ where $n$ is the number of modes of the tensor (see Section~\ref{subsec:complxity section}). This is due to the inefficiency of the subtensor/submatrix extraction in current tensor toolbox, even when the extracted data is partially contiguous.  
If we exclude the time for subtensor/submatrix extraction, then the runtime plots match the theoretical complexities. The empirical time performance of the tensor CUR decompositions will be further improved with future releases of the tensor toolbox when the subtensor/submatrix extraction becomes more efficient.
\end{remark}

\subsection{Hyperspectral Image Compression}

Hyperspectral image compression is a standard real-world benchmark for low multilinear rank tensor approximations \cite{MMD2008,zhang2015compression,du2016pltd,li2021correlation}. 
We consider the use of the Fiber and Chidori CUR decompositions for hyperspectral image compression on three benchmark datasets from \cite{foster2004information}: Ribeira, Braga, and Ruivaes\footnote{The datasets can be found online at  \url{https://personalpages.manchester.ac.uk/staff/d.h.foster/Hyperspectral_images_of_natural_scenes_04.html}.}. The runtime and approximation performance of the tensor CUR decompositions are compared against the other state-of-the-art methods. Approximation performance is evaluated by signal-to-noise ratio (SNR) defined by
\begin{equation*}
    \mathrm{SNR}_\mathrm{dB} = 10\log\left(\frac{\| \mathcal{X}\|^2_F }{\|\mathcal{X}-\mathcal{X}_r\|_F^2}\right),
\end{equation*}
\RV{where $\mathcal{X}$ is the color image corresponding to the original hyperspectral data and $\mathcal{X}_r$ corresponds to the compressed data.} 
The experimental results, along with the size and rank information of the dataset, are summarized in Table~\ref{tab:rr_hyper}. 
\RV{In particular, we determine the ranks based on the original hyperspectral images' HOSVD where we try our best to balance between compression efficiency and quality.} 
One can see both tensor CUR decompositions yield drastically improved speed performance with Fiber CUR being the fastest of the two variants. \RV{On the other hand, Chidori CUR achieves superior approximation results, and Fiber CUR is still very competitive.} 
These datasets are in the form of tall tensors with relatively high ranks, which are conditions typically adverse to CUR methods. Nonetheless, tensor CUR decompositions were able to achieve remarkable advantages in this real-world test.

Finally, we present the visual comparison of the compression results in Figure~\ref{fig:hyper_image_compression}. We find all methods achieve visually good compression regardless of SNR. Once again, the tensor CUR decompositions are able to finish the compression task in a much shorted time without substantially sacrificing quality.

\begin{table}[ht]
\caption{Runtime and relative error. }\label{tab:rr_hyper}
\centering
\begin{tabular}{ |c|c||c|c|c|} 
\hline
\multicolumn{2}{|c||}{~}& Ribeira & Braga  &  Ruivaes \cr
 \hhline {|==||=|=|=|}
\multicolumn{2}{|c||}{Size}  &$1017\times 1340\times 33$              & $1021\times 1338\times 33$                  &   $1017\times 1338\times 33$ \cr
\hline
\multicolumn{2}{|c||}{Rank}    & ($60, 60, 7$)              & ($60, 60, 5$)                  &   ($65,65,4$)     \cr
 \hhline {|==||=|=|=|}
&Fiber CUR &\textbf{0.29} &\textbf{0.26} & \textbf{0.31}\cr
    \cline{2-5}
Runtime &Chidori CUR&$0.66$ &$0.59$& $0.55$\cr
    \cline{2-5}
(seconds)& HOSVD&$1.49$ &$1.41$& $1.42$ \cr
    \cline{2-5}
& st\_HOSVD&$0.83$ & $0.77$ & $0.76$  \cr
     \cline{2-5}
& HOOI &$2.29$ &$ 2.67$ & $3.30$
 \cr
 \hhline {|==||=|=|=|}
&Fiber CUR &$24.14$ &$17.93$& $15.53$ \cr
    \cline{2-5}
SNR &Chidori CUR& \textbf{24.39} & \textbf{18.56} & \textbf{15.84} \cr
    \cline{2-5}
(dB)& HOSVD&$22.99$ & $17.70$ & $15.48$ \cr
    \cline{2-5}
& st\_HOSVD&$22.18 $ &$17.90$& $15.49$ \cr
    \cline{2-5}
&HOOI & $24.33$ & $18.00$ & $15.61$  \cr
\hline
\end{tabular}
\end{table}

\begin{figure}[ht]
\vspace{-0.08in}
\centering

\subfloat[Original]{\includegraphics[width=.16\linewidth]{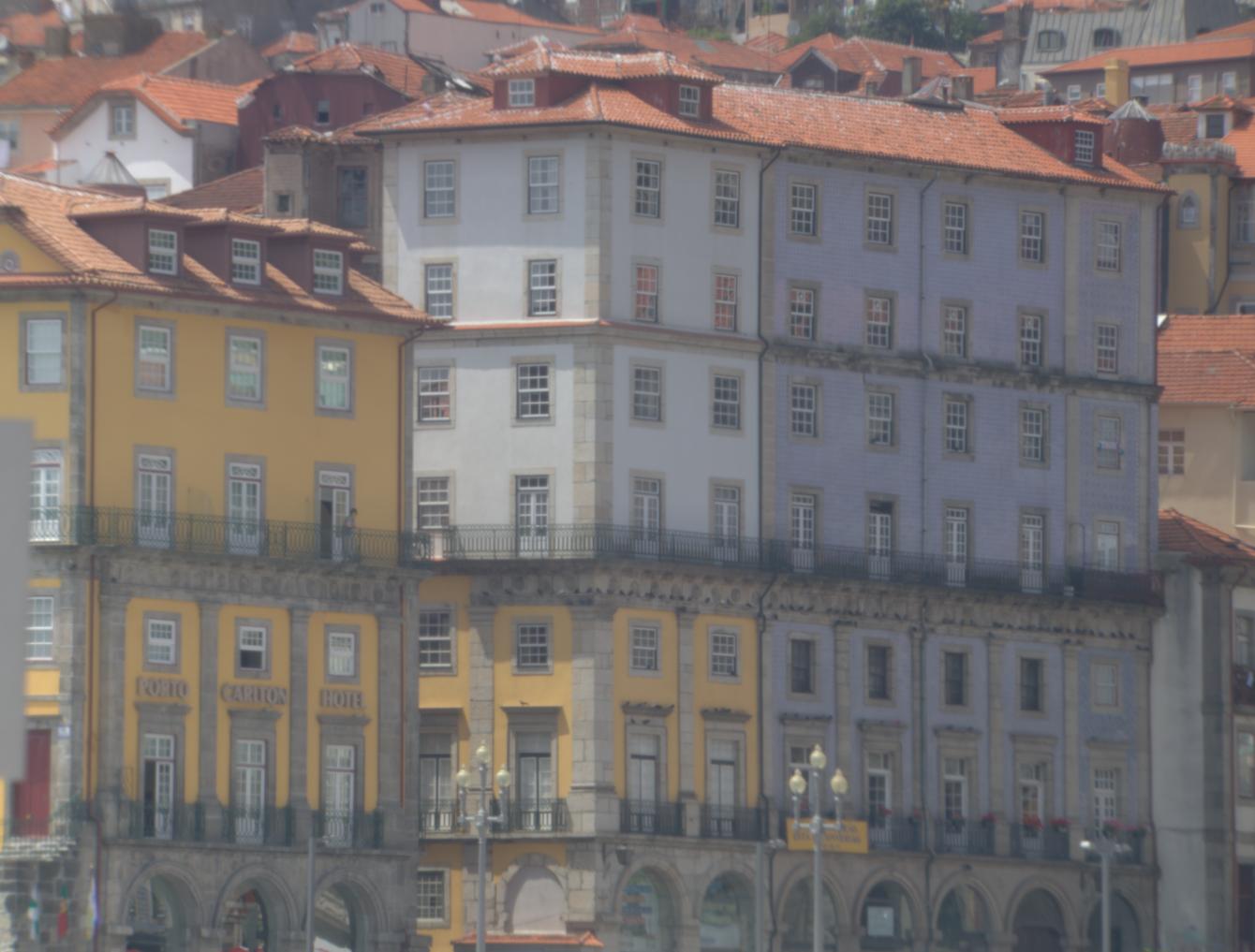}} \hfill
\subfloat[Fiber CUR]{\includegraphics[width=.16\linewidth]{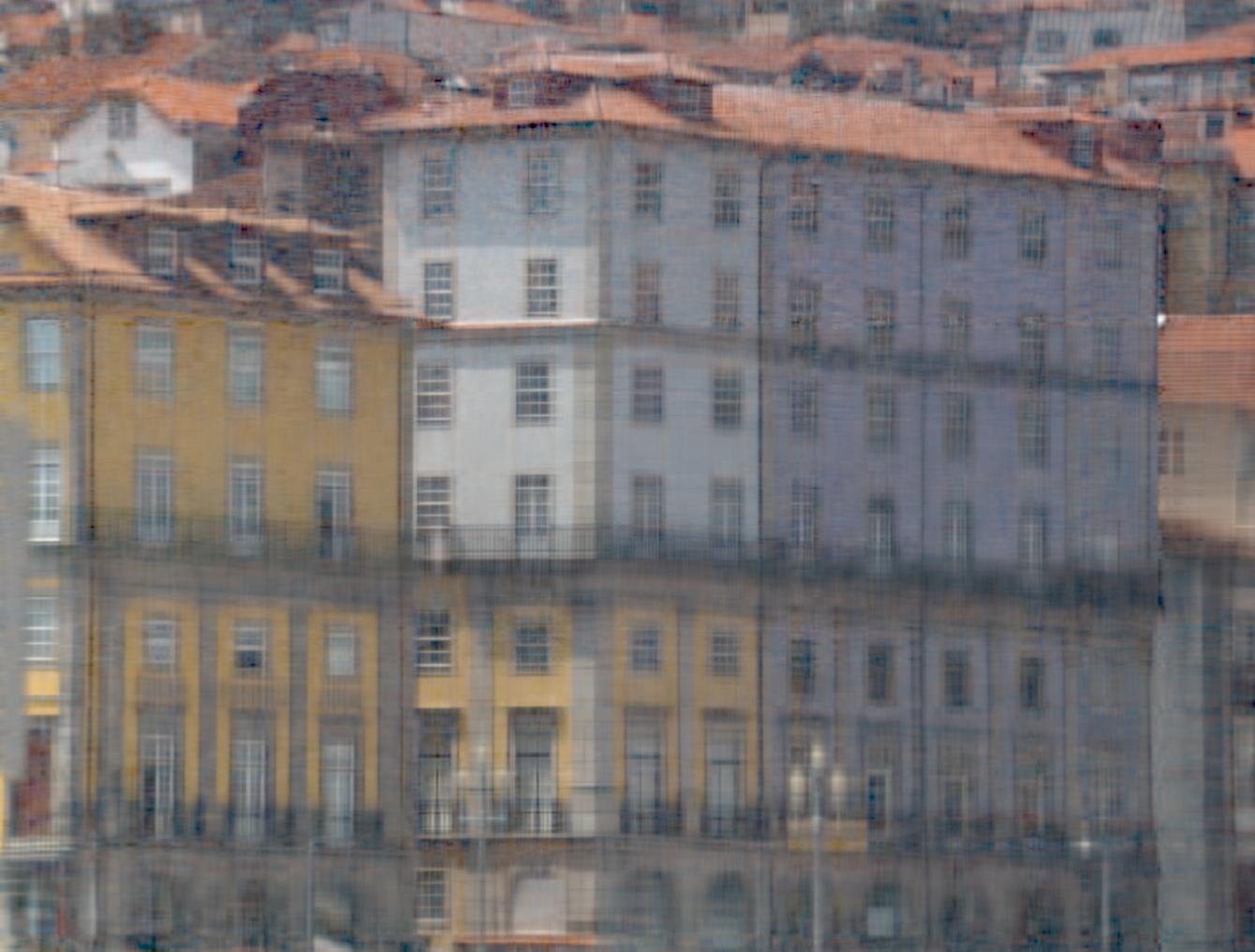}}\hfill
\subfloat[Chidori CUR]{\includegraphics[width=.16\linewidth]{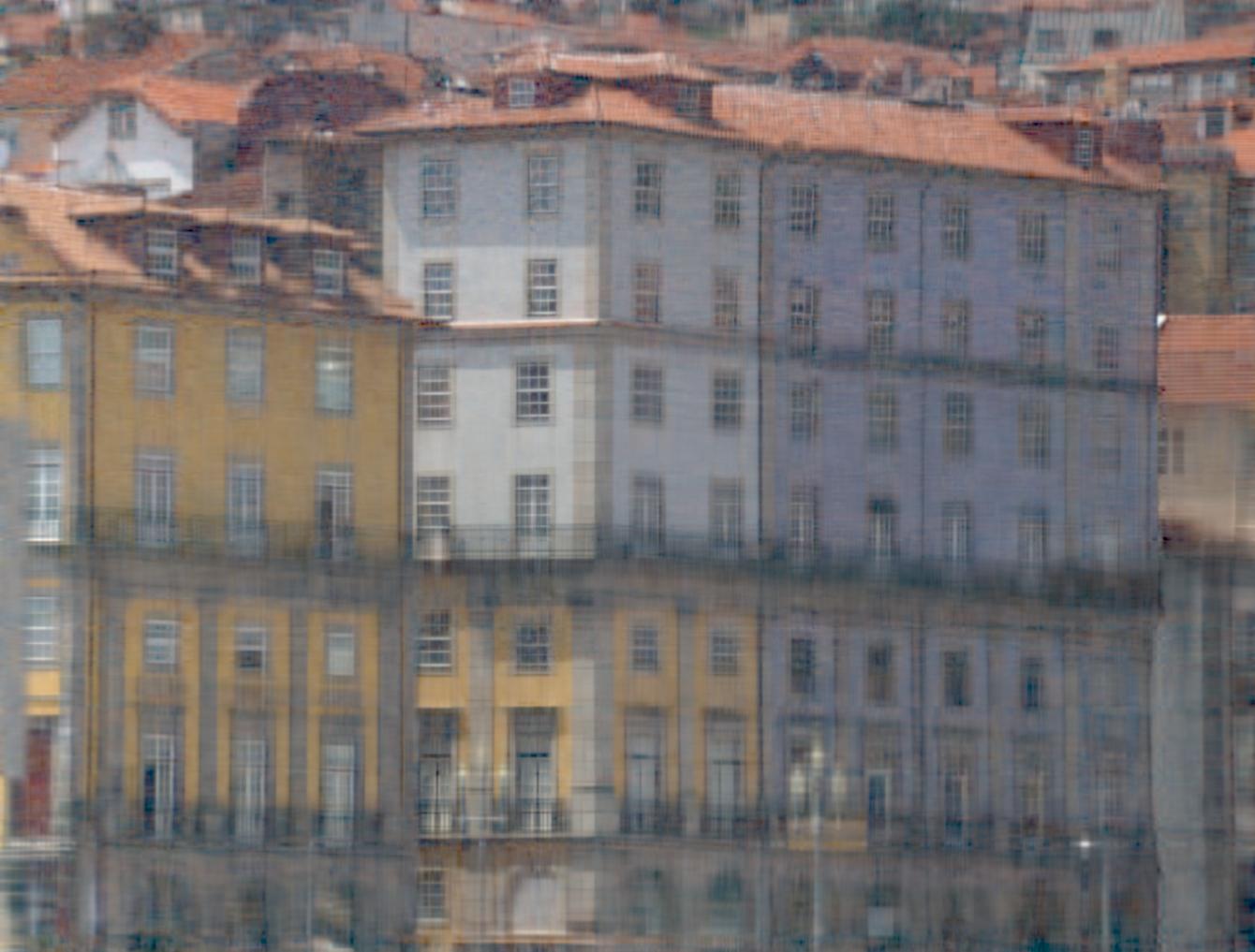}}\hfill
\subfloat[HOSVD]{\includegraphics[width=.16\linewidth]{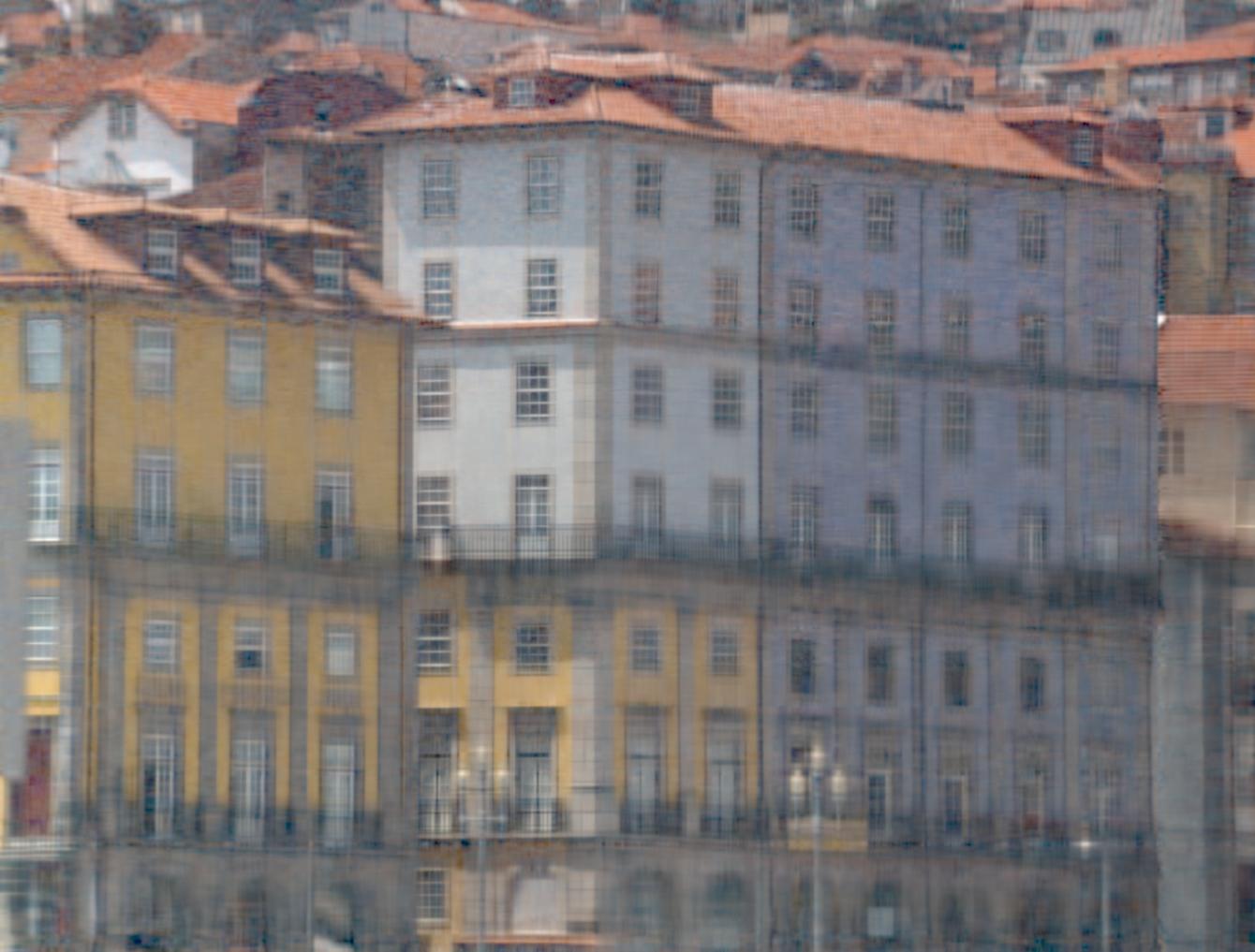}}\hfill
\subfloat[st\_HOSVD]{\includegraphics[width=.16\linewidth]{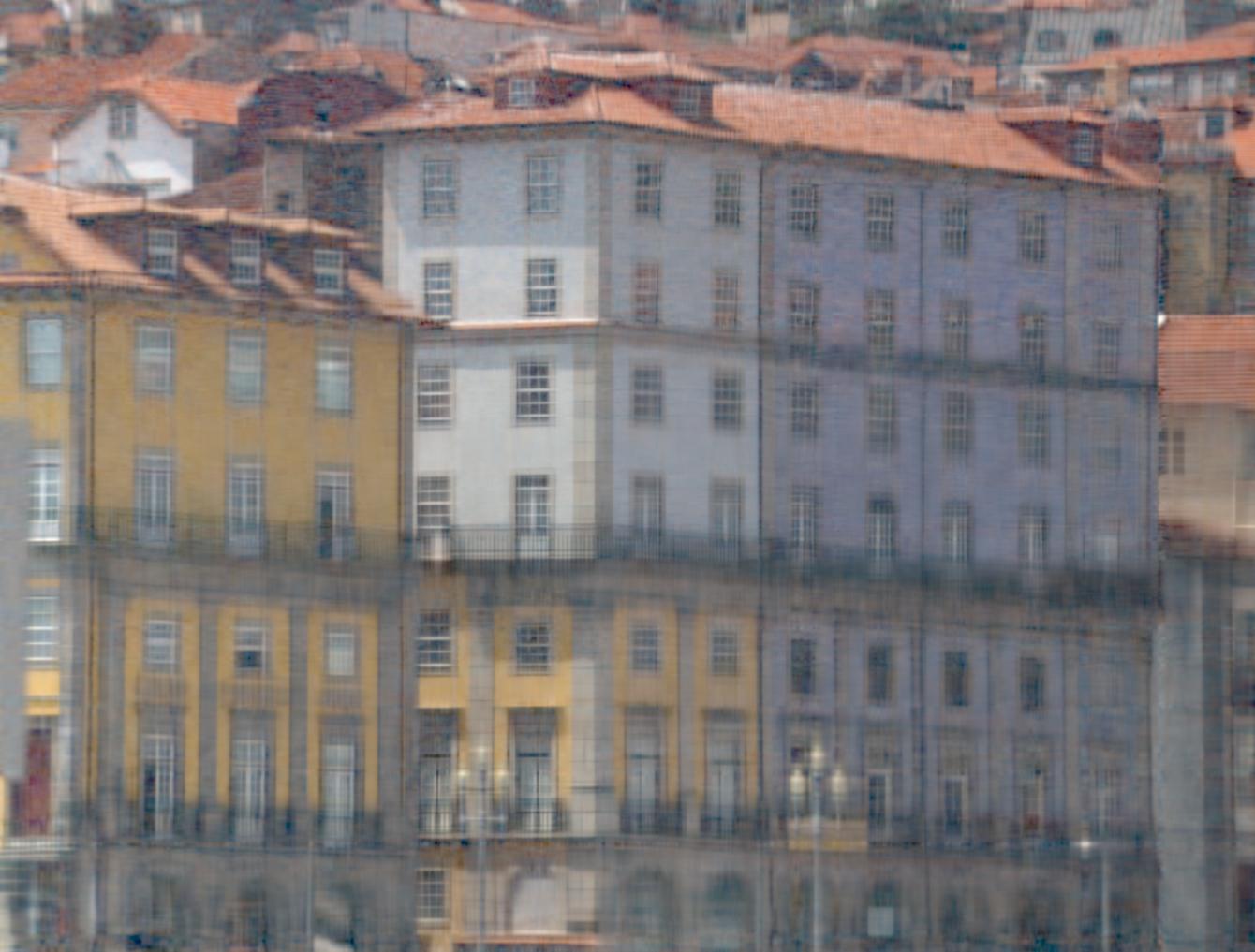}} \hfill
\subfloat[HOOI]{\includegraphics[width=.16\linewidth]{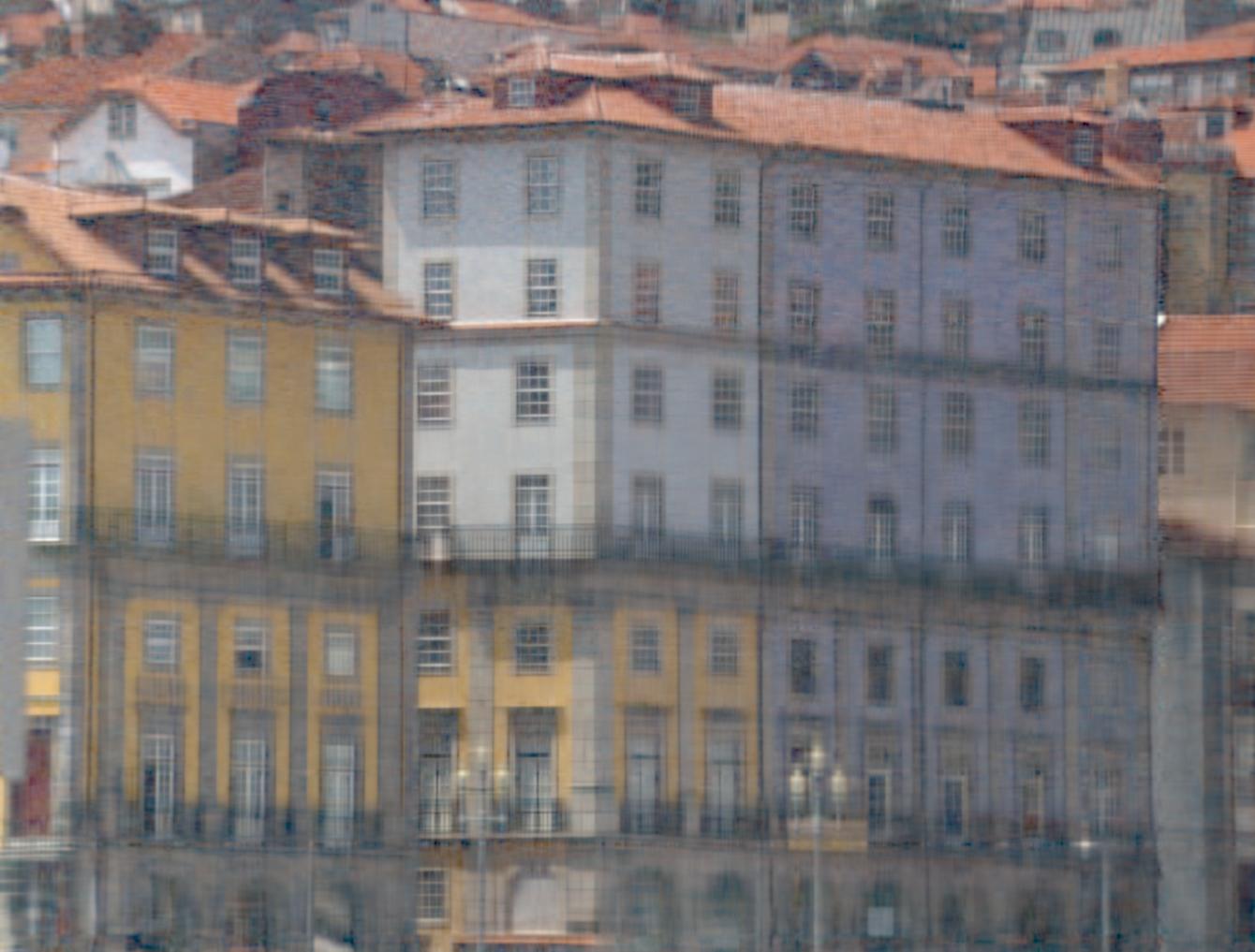}}\hfill
\\
\vspace{-0.1in}

\subfloat{\includegraphics[width=.16\linewidth]{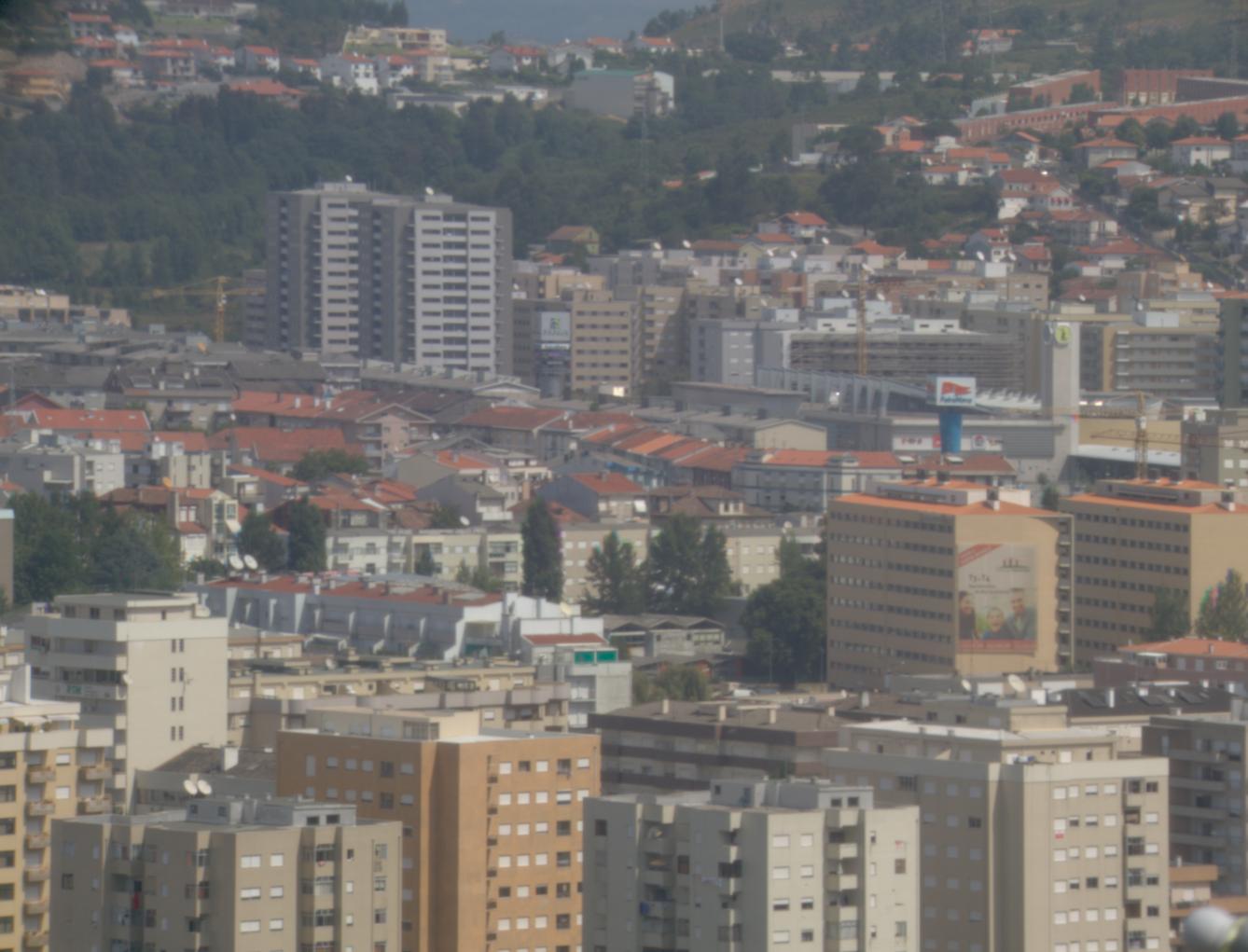}} \hfill
\subfloat{\includegraphics[width=.16\linewidth]{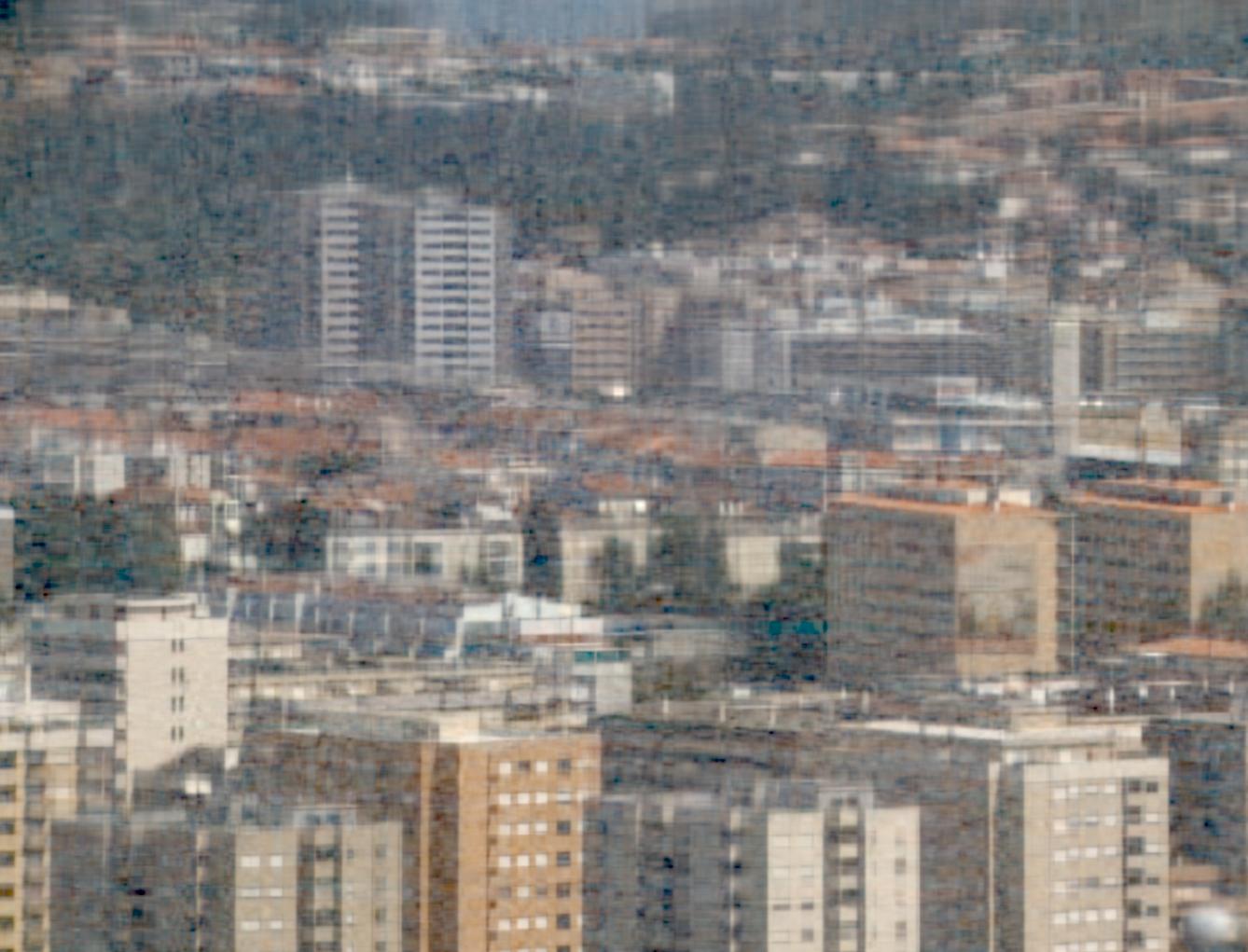}}\hfill
\subfloat{\includegraphics[width=.16\linewidth]{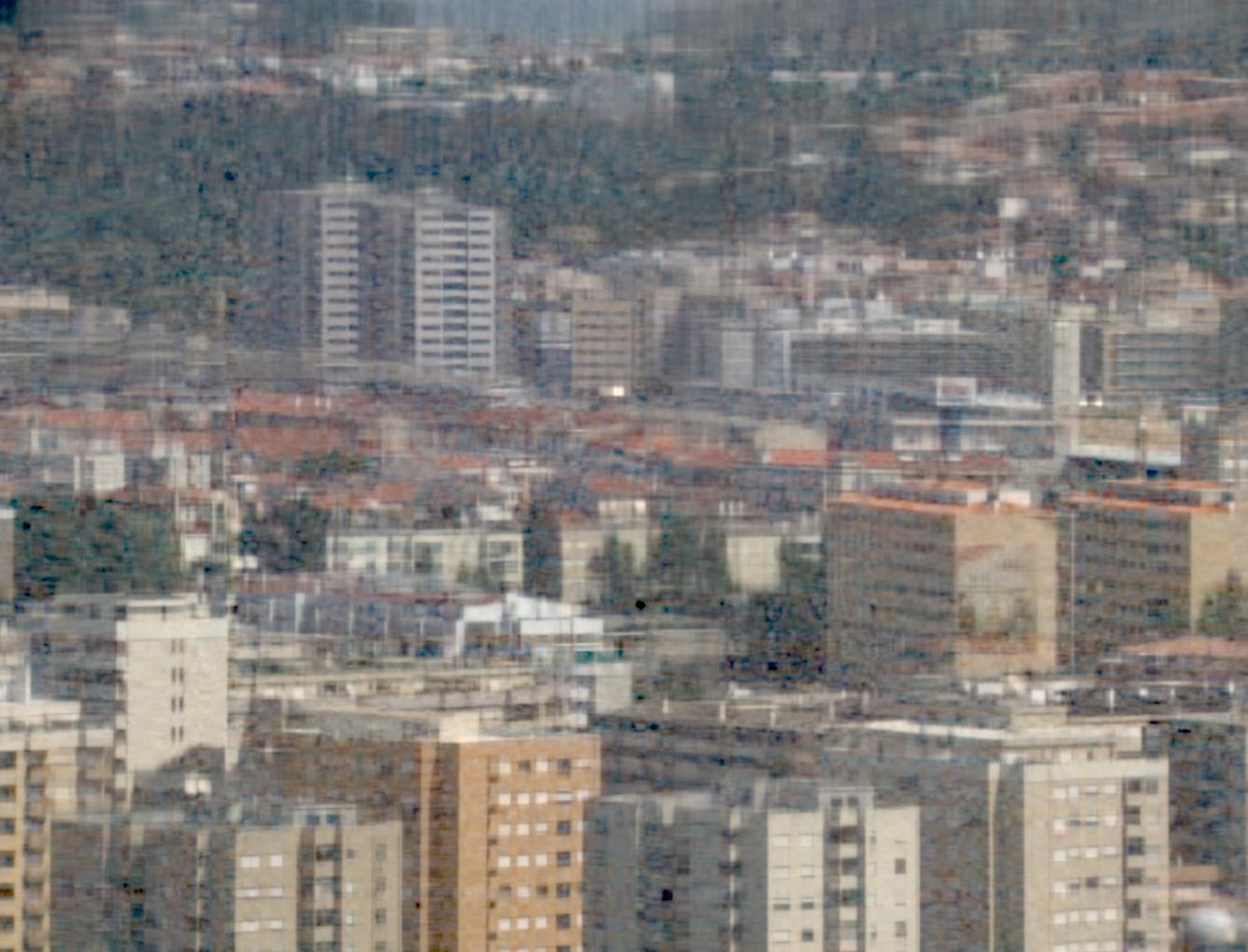}}\hfill
\subfloat{\includegraphics[width=.16\linewidth]{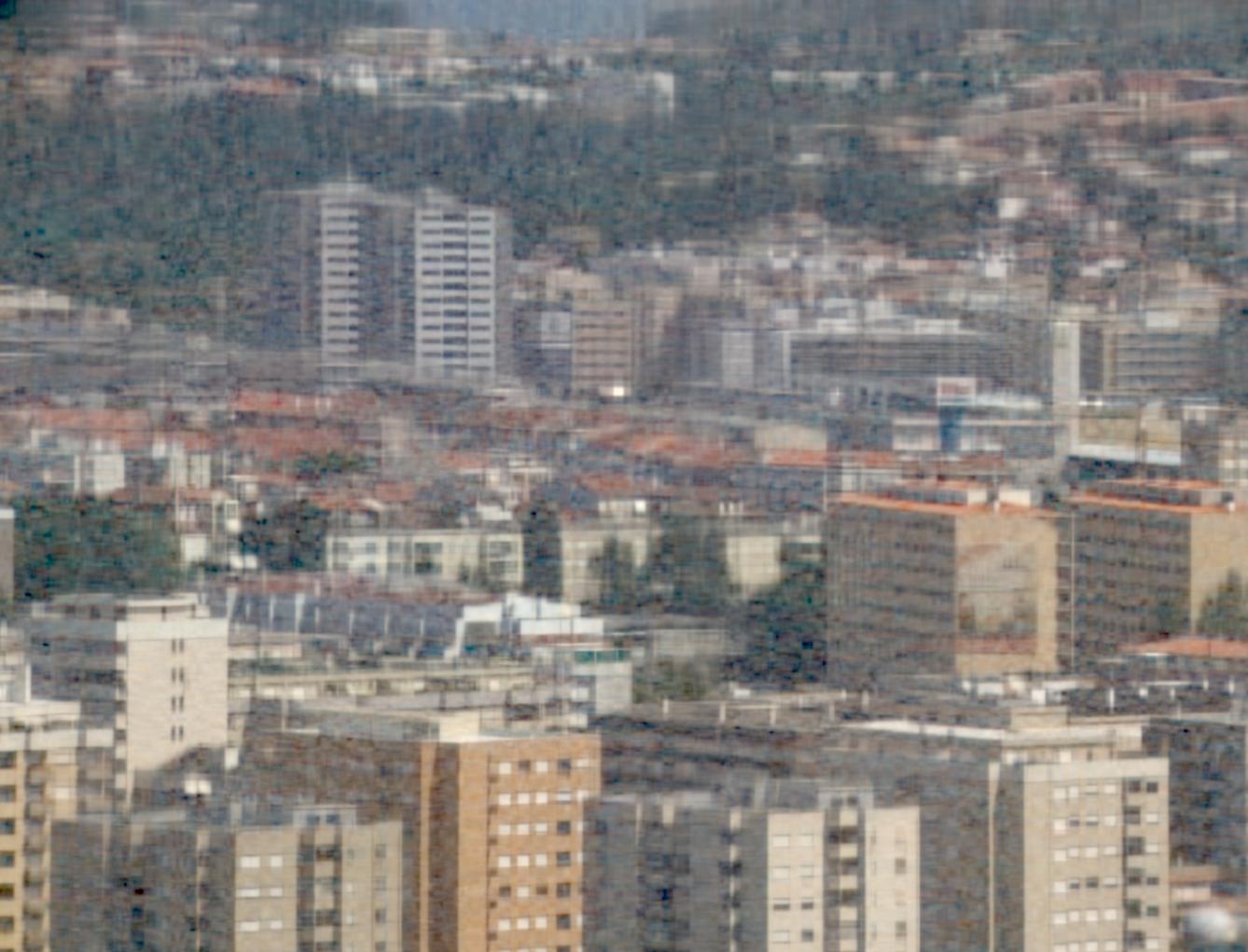}}\hfill
\subfloat{\includegraphics[width=.16\linewidth]{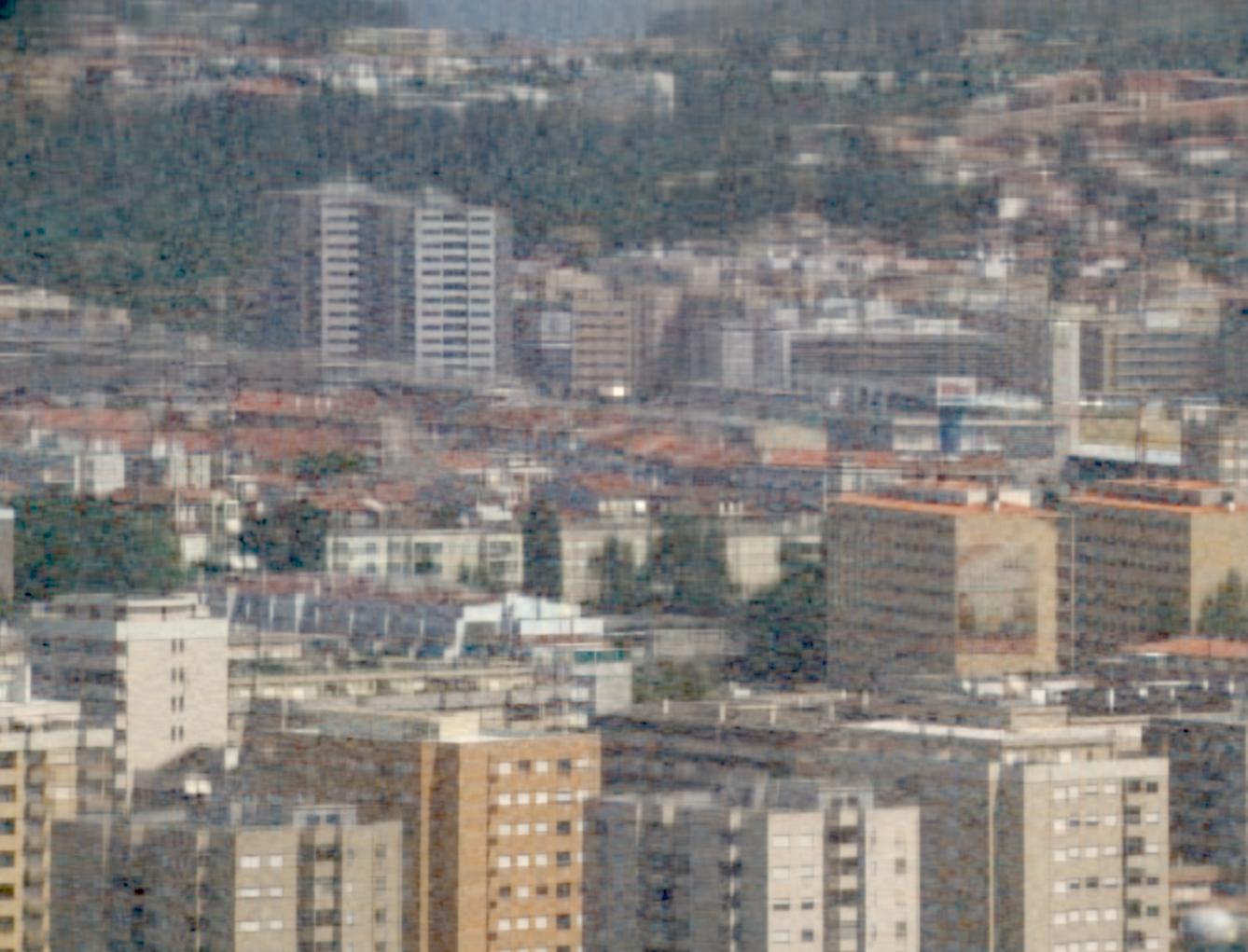}} \hfill
\subfloat{\includegraphics[width=.16\linewidth]{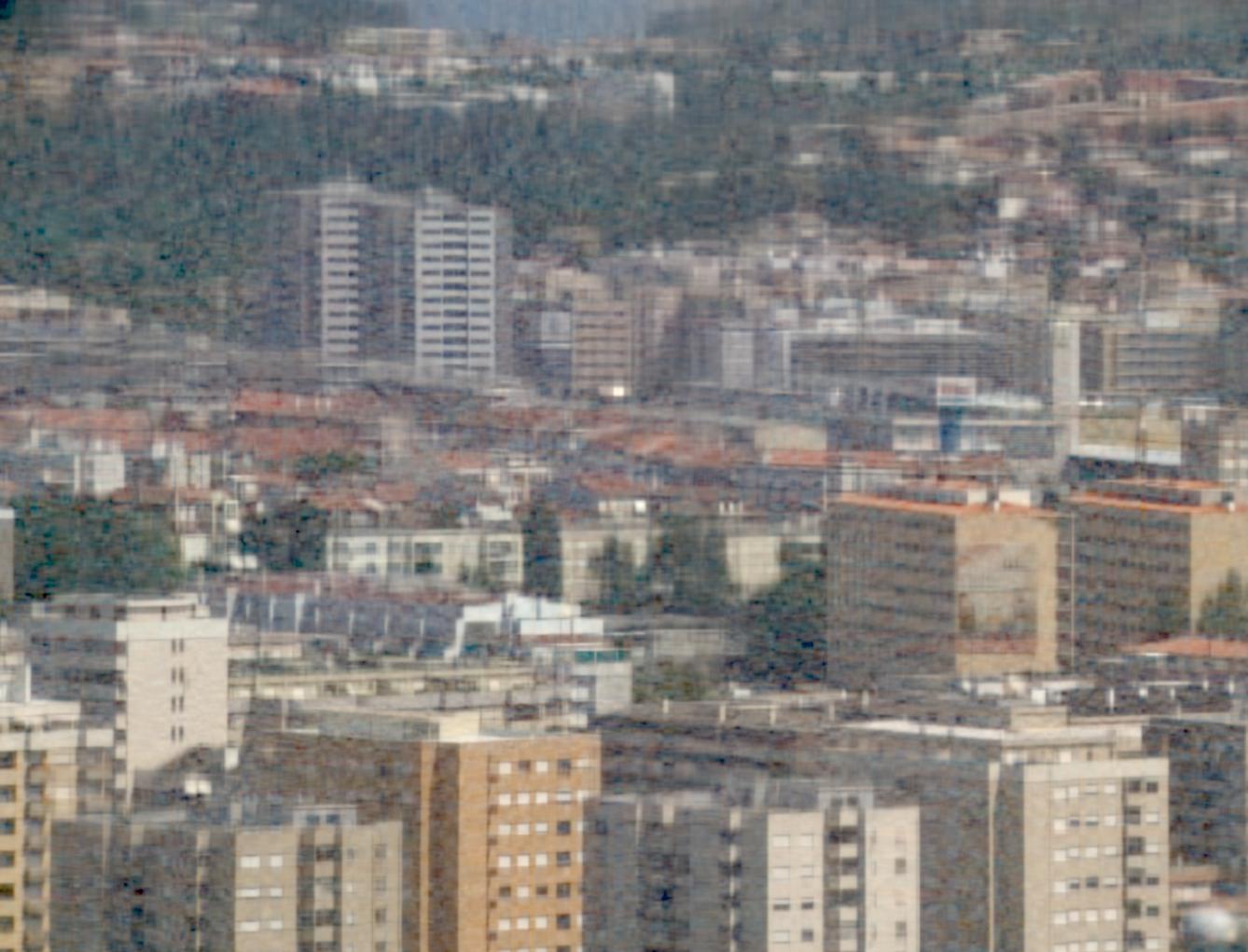}}\hfill
\\
\vspace{-0.1in}

\subfloat{\includegraphics[width=.16\linewidth]{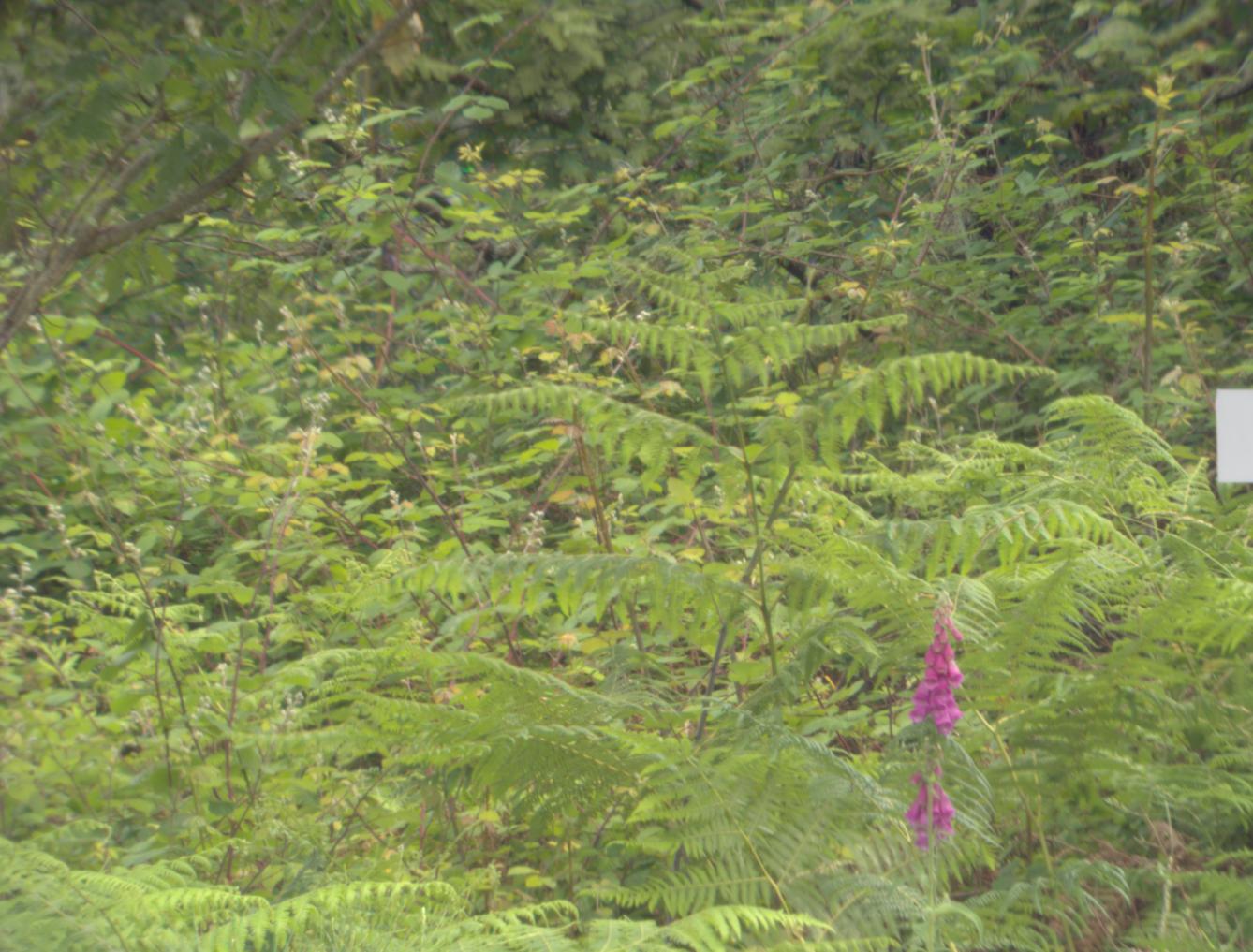}} \hfill
\subfloat{\includegraphics[width=.16\linewidth]{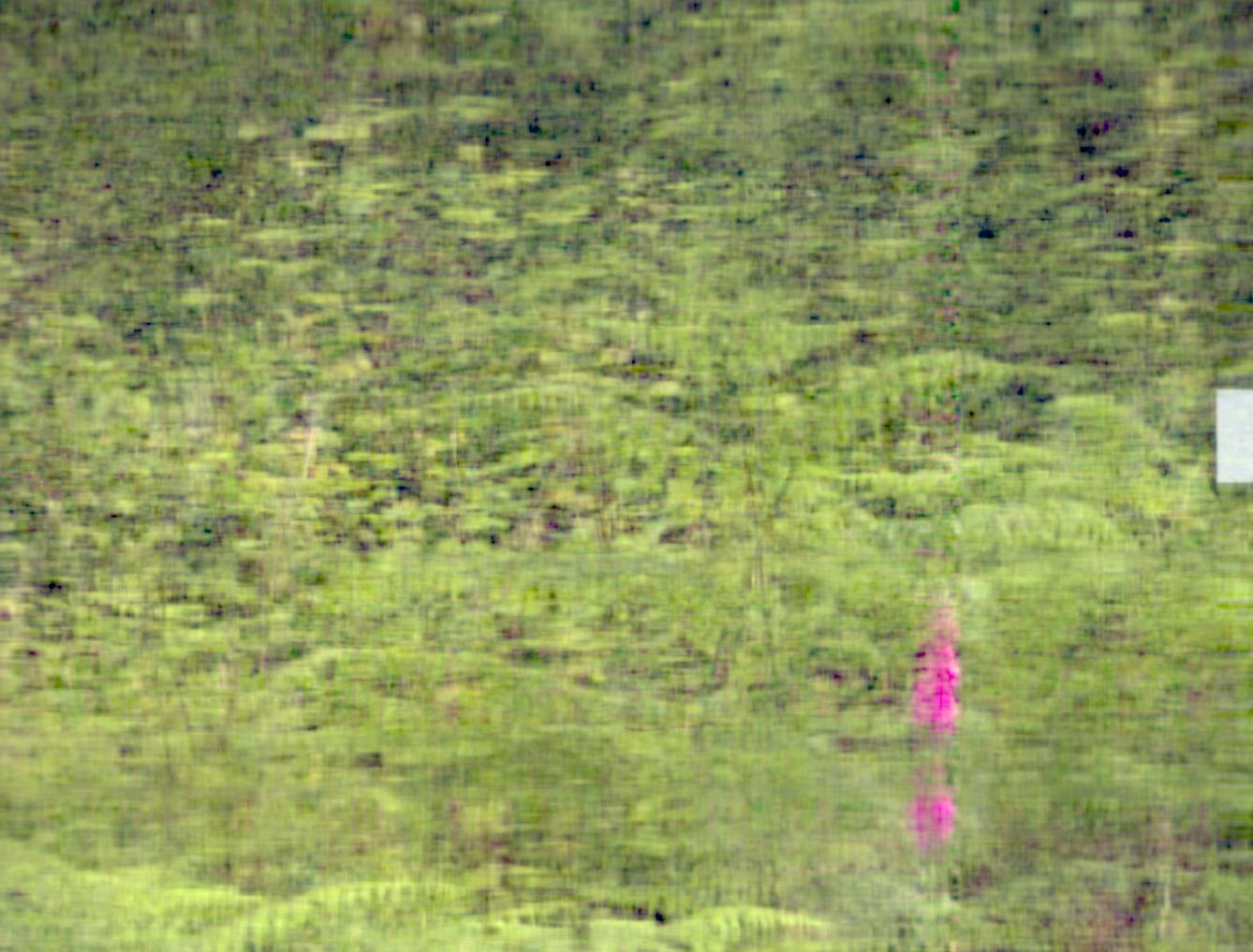}}\hfill
\subfloat{\includegraphics[width=.16\linewidth]{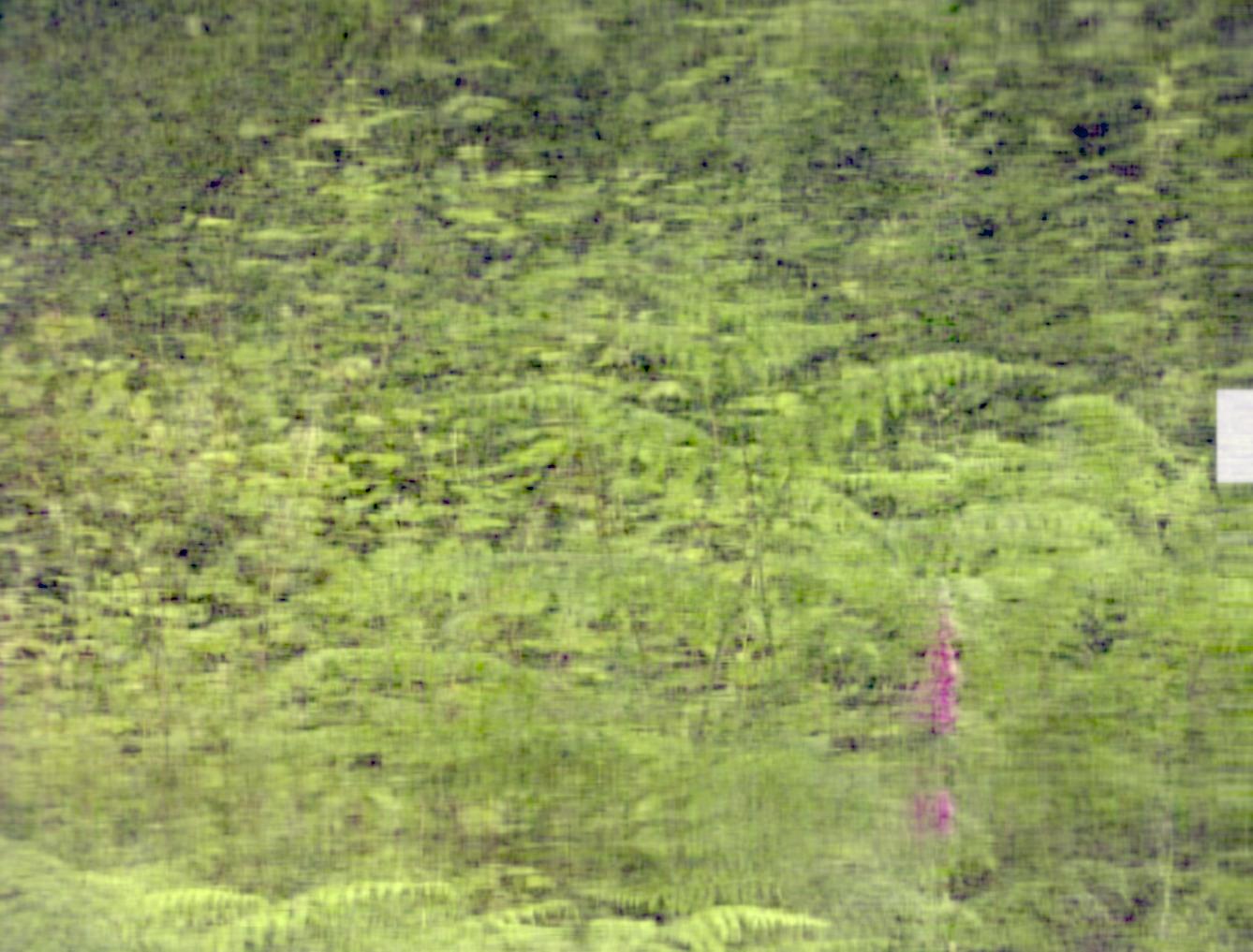}}\hfill
\subfloat{\includegraphics[width=.16\linewidth]{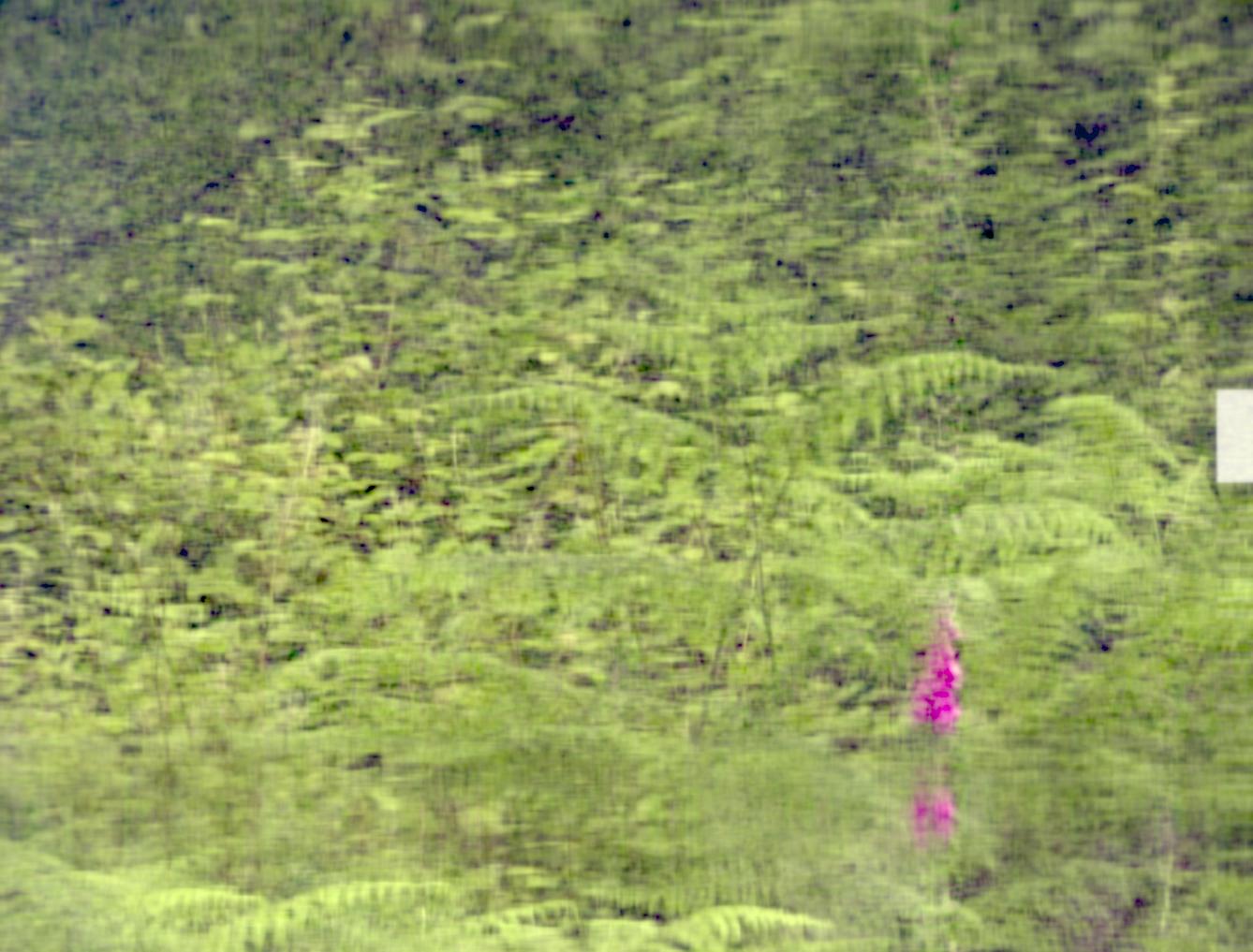}}\hfill
\subfloat{\includegraphics[width=.16\linewidth]{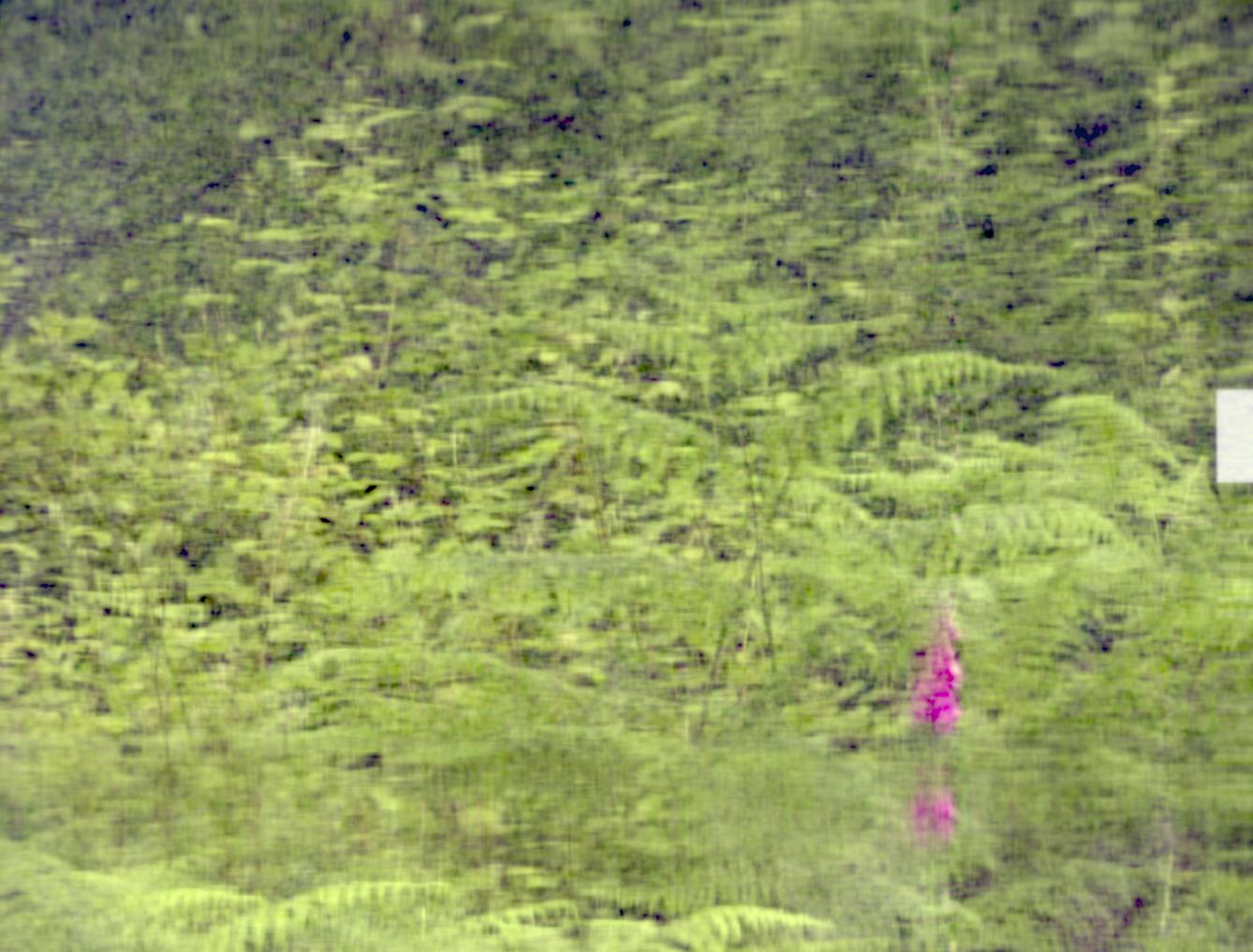}} \hfill
\subfloat{\includegraphics[width=.16\linewidth]{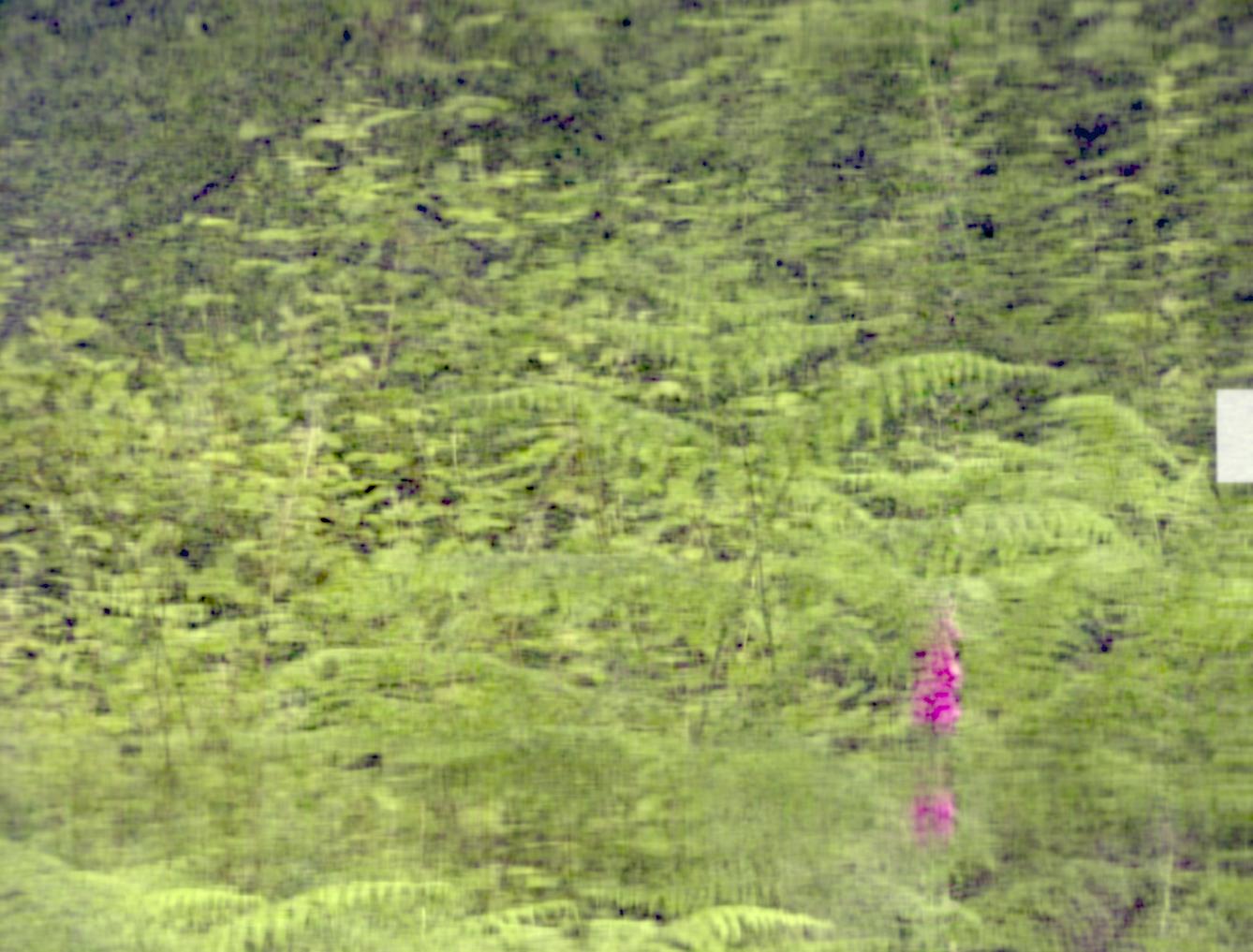}}\hfill

\caption{Visual comparison of the original and compressed hyperspectral images. From top to bottom, each row of the images are for the datasets Ribeira, Braga and Ruivaes, respectively. 
}
\label{fig:hyper_image_compression}
\vspace{-0.05in}
\end{figure}

\section{Conclusions and Future Prospects}\label{SEC:Conclusion}

While extensions of CUR decompositions to tensors are not entirely new, this work gives some of the first nontrivial bounds for quality of approximation of low multilinear rank tensors via two types of approximations: Fiber and Chidori CUR. This was achieved through considering arbitrary tensors as perturbations of low multilinear rank ones. As additional points of interest, we characterized these decompositions and provided error estimates and exact decomposition guarantees under a simple random sampling scheme on the indices.

We also demonstrated that the Fiber and Chidori CUR decompositions obtained via uniformly randomly sampling incoherent tensors are significantly faster than state-of-the-art low-rank tensor approximations without sacrificing quality of reconstruction on both synthetic and benchmark hyperspectral image data sets.  

Given the success of matrix CUR decompositions in a wide range of applications and the ubiquity of tensor data, we expect that tensor CUR decompositions, including those discussed here, will become standard tools for practitioners. In the future, it would be of interest to understand how the idea of reconstructing a tensor from fibers and subtensors of it can be applied to yield fast algorithms for robust decompositions (similar to robust PCA for matrices) and tensor completion.

\section*{Acknowledgments}
DN and LH were supported by NSF DMS $\#2011140$ and NSF BIGDATA $\#1740325$.
KH was sponsored in part by the Army Research Office under grant number W911NF-20-1-0076. The views and conclusions contained in this document are those of the authors and should not be interpreted as representing the official policies, either expressed or implied, of the Army Research Office or the U.S. Government. The U.S. Government is authorized to reproduce and distribute reprints for Government purposes notwithstanding any copyright notation herein.

We thank Amy Hamm Design for the production of Figures~\ref{FIG:TensorCUR} and \ref{FIG:TensorCURIndependent}, and Dustin Mixon for suggesting the name Chidori.

\vskip 0.2in

\bibliographystyle{plain}
\bibliography{tensor}
\end{document}